\documentclass[twoside]{article}[11pt]

\usepackage{amssymb,amsmath,amsthm}
\usepackage[colorlinks,linkcolor=violet,citecolor=violet]{hyperref}
\usepackage[nameinlink]{cleveref}
\usepackage{float}
\usepackage{fullpage}
\usepackage[dvipsnames,table]{xcolor}
\usepackage{fancyhdr,datetime}
\usepackage{graphicx}
\usepackage{subcaption}
\usepackage{tikz,pgf,tikz-3dplot}
\usepackage{tikz-cd}
\usepackage{rotating}
\usepackage{etoolbox}
\usepackage{dynkin-diagrams}
\usepackage{mathrsfs}
\usepackage[percent]{overpic}
\usepackage[export]{adjustbox}
\usepackage{stmaryrd}

\let\oldsection\section 
\renewcommand{\section}{
  \renewcommand{\theequation}{\thesection.\arabic{equation}}
  \oldsection}

\newcommand{\mexp}[1]{\ensuremath{\exp(-2\pi \mathrm{i}\,\sprod{ #1 })}}

\newcommand{\mcos}[1]{\ensuremath{\cos(2 \pi #1 )}}
\newcommand{\gencos}[1]{\ensuremath{\mathfrak{c}_{#1}}}

\newcommand{\N}{\mathbb{N}}

\newcommand\Z{\mathbb{Z}}
\newcommand\C{\mathbb{C}}

\newcommand\R{\mathbb{R}}

\newcommand{\RX}{\ensuremath{\R[z]}}

\newcommand{\CX}{\ensuremath{\C[z]}}

\newcommand{\nops}[1]{\ensuremath{\vert #1  \vert}}

\newcommand{\norm}[1]{\left\lVert#1\right\rVert} 

\newcommand{\rank}{\ensuremath{\mathrm{Rank}}} 
\newcommand{\trace}{\ensuremath{\mathrm{Trace}}} 
\renewcommand{\det}{\ensuremath{\mathrm{Det}}} 
\newcommand{\tbox}[1]{\ensuremath{\quad \mbox{#1} \quad}} 

\DeclareMathOperator{\diag}{diag}

\newcommand{\conj}[1]{\ensuremath{\widehat{#1}}}

\newcommand{\RootA}[1][n-1]{\ensuremath{\mathrm{A}_{#1}}}
\newcommand{\RootB}[1][n]{\ensuremath{\mathrm{B}_{#1}}}
\newcommand{\RootC}[1][n]{\ensuremath{\mathrm{C}_{#1}}}
\newcommand{\RootD}[1][n]{\ensuremath{\mathrm{D}_{#1}}}
\newcommand{\RootE}[1][n]{\ensuremath{\mathrm{E}_{#1}}}
\newcommand{\RootF}[1][n]{\ensuremath{\mathrm{F}_{#1}}}
\newcommand{\RootG}[1][n]{\ensuremath{\mathrm{G}_{#1}}}

\newcommand{\weyl}{\ensuremath{\mathcal{W}}} 

\newcommand{\PC}{\ensuremath{\Lambda\!\!\Lambda}}

\newcommand{\roots}{\ensuremath{\rho}}
\newcommand{\highestroot}{\roots_{0}}
\newcommand{\Roots}{\ensuremath{\mathrm{R}}}
\newcommand{\Base}{\ensuremath{\mathrm{B}}}
\newcommand{\Corootlattice}{\ensuremath{\Lambda}}
\newcommand{\sprod}[1]{\ensuremath{\langle #1 \rangle}} 

\newcommand{\fweight}[1]{\ensuremath{\omega_{#1}}}	
\newcommand{\weight}{\ensuremath{\mu}} 
\newcommand{\Weights}{\Omega} 

\newcommand{\Cheblvl}[1]{\ensuremath{\deg_W(#1)}} 
\newcommand{\TT}{\ensuremath{\conj{T}}} 
\newcommand{\filt}[1]{\ensuremath{\mathcal{F}_{#1}}} 

\newcommand{\Vor}{\ensuremath{\mathrm{Vor}}}
\newcommand{\fundom}{\triangle} 

\newcommand{\Image}{\ensuremath{\mathcal{T}}} 
\newcommand{\posmat}{\ensuremath{\mathbf{P}}} 

\def\row#1/#2!{#1_{\IfStrEq{#2}{}{n-1}{#2}} & \dynkin{#1}{#2}\\}

\newcommand{\momm}{\ensuremath{\mathbf{H}}} 
\newcommand{\measure}{\ensuremath{\mathcal{B}}} 
\newcommand{\functional}{\ensuremath{\mathscr{L}}} 
\newcommand{\primalprob}{\ensuremath{(\mathrm{P}_d)}} 
\newcommand{\dualprob}{\ensuremath{(\mathrm{D}_d)}} 

\newtheorem{lemma}{Lemma}[section]
\newtheorem{example}[lemma]{Example}

\newtheorem{definition}[lemma]{Definition}
\newtheorem{proposition}[lemma]{Proposition}
\newtheorem{remark}[lemma]{Remark}
\newtheorem{theorem}[lemma]{Theorem}

\newtheorem{corollary}[lemma]{Corollary}

\newtheorem{convention}[lemma]{Convention}

{\scshape}{\slshape}

\newcommand{\italgf}{\slshape  }

\title{
	{O}ptimization of trigonometric polynomials\\
	with crystallographic symmetry and\\
	spectral bounds for set avoiding graphs
}

\author{	
	Evelyne Hubert\thanks{Inria d'Universit\'{e} C\^{o}te d'Azur} , 
	Tobias Metzlaff\footnotemark[1] \thanks{RPTU Kaiserslautern--Landau} , 
	Philippe Moustrou\thanks{Universit\'{e} Toulouse Jean Jaures} , 
	Cordian Riener\thanks{UiT The Arctic University}
}

\newcommand{\Addresses}{{
\bigskip
\footnotesize

	E.~Hubert, \textsc{Centre Inria d'Universit\'{e} C\^{o}te d'Azur, 06902 Sophia Antipolis, France}\par\nopagebreak
  	\textit{E-mail address}: \texttt{evelyne.hubert@inria.fr}\par\nopagebreak
  	\textit{ORCID}: \texttt{0000-0003-1456-9524}
  	
\medskip
  	
  	T.~Metzlaff, \textsc{Department of Mathematics, RPTU Kaiserslautern--Landau, 67663 Kaiserslautern, Germany}\par\nopagebreak
  	\textit{Former Association}: \textsc{Centre Inria d'Universit\'{e} C\^{o}te d'Azur, 06902 Sophia Antipolis, France}\par\nopagebreak
  	\textit{E-mail address}: \texttt{tobias.metzlaff@rptu.de}\par\nopagebreak
  	\textit{ORCID}: \texttt{0000-0002-0688-7074}

\medskip

  	P.~Moustrou, \textsc{Institut de Math\'{e}matiques de Toulouse, Universit\'{e} Toulouse Jean Jaures, 31100 Toulouse, France}\par\nopagebreak
  	\textit{E-mail address}: \texttt{philippe.moustrou@math.univ-toulouse.fr}\par\nopagebreak
  	\textit{ORCID}: \texttt{0000-0003-3432-4954}

\medskip

  	C.~Riener, \textsc{Department of Mathematics, UiT The Arctic University, 9037 Troms\o, Norway}\par\nopagebreak
  	\textit{E-mail address}: \texttt{cordian.riener@uit.no}\par\nopagebreak
  	\textit{ORCID}: \texttt{0000-0002-1192-3500}
  
}}

\renewcommand{\headsep}{8mm}
\renewcommand{\footskip}{15mm}

\begin{document}

\maketitle
\thispagestyle{empty}

\begin{abstract}
Trigonometric polynomials are usually defined on the lattice of integers.
We consider the larger class of weight and root lattices with crystallographic symmetry.
This article gives a new approach to minimize trigonometric polynomials, which are invariant under the associated reflection group.
The invariance assumption allows us to rewrite the objective function in terms of generalized Chebyshev polynomials. 
The new objective function is defined on a compact basic semi--algebraic set, so that we can benefit from the rich theory of polynomial optimization.

We present an algorithm to compute the minimum: 
Based on the Hol--Scherer Positivstellensatz, we impose matrix--sums of squares conditions on the objective function in the Chebyshev basis.
The degree of the sums of squares is weighted, defined by the root system. 
Increasing the degree yields a converging Lasserre--type hierarchy of lower bounds.
This builds a bridge between trigonometric and polynomial optimization, allowing us to compare with existing techniques.

The chromatic number of a set avoiding graph in the Euclidean space is defined through an optimal coloring.
It can be computed via a spectral bound by minimizing a trigonometric polynomial. 
If the to be avoided set has crystallographic symmetry, our method has a natural application.
Specifically, we compute spectral bounds for the first time for boundaries of symmetric polytopes.
For several cases, the problem has such a simplified form that we can give analytical proofs for sharp spectral bounds.
In other cases, we certify the sharpness numerically.
\clearpage
\end{abstract}

\tableofcontents

~\\~

\Addresses

\clearpage

\section{Introduction}
\label{section_introduction}
\setcounter{equation}{0}

Given a $n$--dimensional lattice $\Weights \subseteq \R^n$,
a trigonometric polynomial is a function
\[
	f: \R^n \to \R ,\, u\mapsto f(u) := \sum\limits_{\weight\in \Weights} c_\weight \, \mexp{\weight,u},
\]
where 
$\sprod{\cdot,\cdot}$ denotes the Euclidean scalar product 
and the finitely many nonzero coefficients $c_\weight\in \C$ satisfy $ c_{-\weight} =\overline{c_{\weight}}$. 
Such functions are good $L^2$--aproximations for 
$\Lambda$--periodic functions, where $\Lambda$ is the dual lattice. 
This paper offers a new approach to optimizing such a 
trigometric function, over $\R^n$, when this latter is invariant under a crystalographic reflection group. 
We show how the problem can then be reduced to polynomial optimization on a semi--algebraic set 
and handled with a variation on 
Lasserre hierarchy. The resulting algorithm is applied to the exploration of the spectral bound on the chromatic numbers of set avoiding graphs.

In the literature of trigonometric optimization, one often regards the lattice simply as a free $\Z$--module, that is, $\Weights=\Z^n$, ignoring the geometry and only taking central symmetry into account. 
For the purpose of optimization, a hierarchy of Hermitian sums of squares reinforcements provides a numerical solution \cite{dumitrescu07,bach22}.
Alternatively, one can apply Lasserre's hierarchy with complex variables \cite{josz18}, where one has to restrict to the compact torus.

In this article, $\Weights$ is the weight lattice of a crystallographic root system in $\R^n$. 
Root  and weight lattices  provide optimal configurations for a variety of problems in  geometry and information theory, with incidence in  physics and chemistry. 
The $\RootA[2]$ root lattice (the hexagonal lattice) is classically known to be optimal for sampling, packing, covering, and quantization in the plane \cite{conway1988a,kunsch05}, 
but also proved, or conjectured, to be optimal for energy minimization problems \cite{Petrache20,faulhuber23}.  
More recently, the $\RootE[8]$ lattice was  proven to give an optimal solution for the sphere packing problem and a large class of energy minimization problems in dimension $8$
\cite{Petrache20,Viazovska17,Viazovska22}.
From an approximation point of view, weight lattices of root systems describe Gaussian cubature \cite{Xu09,Moody2011}, 
a rare occurence on multidimensional domains.
In a different direction, the triangulations associated with 
infinite families of root systems are relevant in graphics and computational geometry, see for instance \cite{Choudhary20} and references within.

The distinguishing feature of the lattices associated to crystallographic root system is their intrisic symmetry.
This latter is given by the so called Weyl group $\weyl$, a finite group generated by orthogonal reflections w.r.t. $\sprod{\cdot,\cdot}$. 
It is this feature that we emphasize and offer to exploit in an optimization context.  
We present a new approach to numerically solve the 
trigonometric optimization problem 
\begin{equation}\label{OptiProblemExpo}
f^* := \min\limits_{u\in \R^n} f(u)
\end{equation}
under the assumption of crystallographic symmetry, that is, for $A\in\weyl$, we have $f(A\,u) = f(u)$, or equivalently $c_{A\,\weight} = c_\weight$.
The first step of our approach, 
in \Cref{section_trigonometric}, 
is a symmetry reduction that translates the trigonometric optimization above to the problem of optimizing a polynomial 
over a semi--algebraic set, a subject that 
ripened in the last two decades \cite{lasserre01,parrilo03,sturmfels03,parrilo05,klerk05,lasserre09,laurent09,blekherman12,lasserre21}. 
The second step of our approach, 
in \Cref{section_optimization}, is thus an adaptation of 
Lasserre's hierarchy of moment relaxations and sums of squares reinforcements.
We indeed modify the hierarchy introduced in \cite{holscherer05,holscherer06,lasserre06} to work directly in the basis of generalized Chebyshev polynomials. 
These  are not homogeneous but naturally 
filtered by a weighted degree, different from the usual degree. 

The simplest case of this symmetry reduction scheme, the univariate case, is obvious but maybe worth reviewing 
to get the initial idea.
The group is then $\weyl=\{1,-1\}$ 
and the invariance condition is thus
$f(-u)=f(u)$ for all $u\in\R$.
That implies that one can write 
$$f(u)= \sum_{k\in \N} \frac{c_k}{2}	
\left(\exp(2 \pi \mathrm{i}\, k u) + \exp(-2 \pi \mathrm{i} \,ku)\right) 
= \sum_{k\in \N} c_k \,\cos(2\pi\, k u)
= \sum_{k\in \N} c_k \, T_k(\cos(2\pi\, u)),$$ 
where $\{T_k\}_{k\in\N}$ 
are the Chebyshev polynomials of the first kind. 
We  thus have
\[
f^* := \min\limits_{u\in \R^n} f(u) = 
\min\limits_{z^2\leq 1} \sum_{k\in \N} c_k  T_k(z)
\]
the right hand side being a polynomial optimization problem with semi--algebraic constraints.

With $\Omega=\Z^n$ and $\weyl=\{1,-1\}^n$, 
one can use products of univariate Chebyshev polynomials to operate a similar symmetry reduction. 
This is the $\RootA[1]\times \ldots \times\RootA[1]$ case. 
We look at all the lattices associated to crystallographic root systems, 
offering a wider range of  domains of periodicity 
(hexagon, rhombic dodecahedron, icositetrachoron, \ldots) 
and simplices of any dimension, 
or cartesian products of these, as fundamental domains.
The key to the symmetry reduction then is the existence and properties of generalized Chebyshev polynomials. 
They allow to rewrite any invariant trigonometric 
polynomials as polynomials of the fundamental generalized cosines.
These generalized Chebyshev polynomials arose in different contexts, in particular in the search of multivariate 
orthogonal polynomials \cite{DunnLidl1980,EierLidl1982,HoffmanWithers,MacDonald1990,beerends91}.
A more recent development is the description of their domain of orthogonality, the image of the generalized cosines,
as a compact semi--algebraic set given by a
unified and explicit polynomial matrix inequality \cite{chromaticissac22,TOrbits,TobiasThesis}.
Such a description is necessary to proceed algorithmically with the obtained polynomial optimization problem.

In the algorithmic approach, we solve a primal/dual semi--definite program (SDP) that models a moment--relaxation/sums of squares reinforcement in terms of generalized Chebyshev polynomials. 
Our \textsc{Maple} package 
\textsc{GeneralizedChebyshev}\footnote{\href{https://github.com/TobiasMetzlaff/GeneralizedChebyshev}{https://github.com/TobiasMetzlaff/GeneralizedChebyshev}} 
allows to compute the parameters of the SDP, specifically the matrices which impose the semi--definite constraints. 
The user can then solve the problem with a SDP solver of their personal preference. 
Beyond that, the package offers a large variety of tools, including the matrices from \cite{TOrbits}, 
a function to rewrite invariants in terms of generalized Chebyshev polynomials and an implemented recurrence formula for their computation. 
We can thus compare our method with the one in \cite{dumitrescu07} in practice. 
We observe in several examples throughout \Cref{sec_casestudy} that the quality of the approximation is improved, while the computational complexity is reduced. 

As a second set of contributions, in \Cref{section_chromatic},
we apply  our method 
to the computation of spectral bounds 
for chromatic numbers of set avoiding graphs. 
The first such graph considered was the Euclidean distance graph 
\cite{Soifer09,BdCOV,BPT15,deGrey18}, where the vertices are the points of $\R^n$ and the set to be avoided is the sphere. 
As set of vertices we consider either $\R^n$, 
or a lattice thereof.
As for the set to be avoided we mostly consider the boundary of a polytope with crystallographic symmetry.
Choosing  appropriate discrete measures on the polytope, 
the  spectral bound from \cite{BdCOV} 
made specific to the chromatic number can be expressed 
as the solution of  a max--min  optimization problem on a trigonometric polynomial. 
Our symmetry reduction technique 
of \Cref{section_trigonometric} then allows us to retrieve,
with simple proofs, the chromatic number 
of the $\RootA$ lattice (\Cref{thm_LatticeChromatic}), 
of the graph avoiding the crosspolytope of radius $2$ in 
$\Z^n$ (\Cref{thm_ZnL1r2bound}), 
and of the graph avoiding the cube in $\R^n$  
(\Cref{thm_cubeRnbound}). 
In other cases, 
we apply the algorithm in \Cref{section_optimization} to
compute lower bounds numerically. We improve on \cite{furedi04}  
by $+2$ for the chromatic number of $\Z^4$ avoiding the crosspolytope of radius $4$ (\Cref{B4C4D4L1Table2}).
We also give further bounds 
for the rhombic dodecahedron (\Cref{A3RhombicTable}) 
as well as the icositetrachroron 
(\Cref{B4D4IcositetrachoronTable2}). 

\section{Crystallographic symmetries}
\label{section_trigonometric}
\setcounter{equation}{0}

To rewrite the trigonometric optimization problem in \Cref{OptiProblemExpo} to a polynomial optimization problem, we require the lattice $\Weights$ to be full--dimensional and stable under some finite reflection group $\weyl$, that is, $\weyl \, \Weights =\Weights$. 
Then $\weyl$ must be the Weyl group of some crystallographic root system \cite[Chapter 9]{kane13} and $\Weights$ is the associated weight lattice. 
We need several facts from the theory of Lie algebras, root systems and lattices, which come from \cite{bourbaki456,humphreys12,conway1988a}. 
In particular, we need \Cref{theorem_BourbakiGenerators}, which states that any trigonometric polynomial with crystallographic symmetry can be written uniquely as a polynomial in fundamental invariants, also known as the generalized cosines. 
Subsequently, the feasible region of the so obtained polynomial optimization problem is the image of the fundamental invariants, a compact basic semi--algebraic set whose equations were given explicitely in \cite{chromaticissac22,TOrbits,TobiasThesis}.

\subsection{Root systems and Weyl groups}

Denote by $\sprod{\cdot , \cdot }$ the Euclidean scalar product. A subset $\Roots\subseteq \R^n$ is called a \textbf{root system} in $\R^n$, if the following conditions hold.
\begin{enumerate}
	\item[R1] $\Roots$ is finite, spans $\R^n$ and does not contain $0$.
	\item[R2] If $\rho, \tilde{\rho} \in \Roots$, then $\langle\tilde{\rho},\rho^\vee\rangle \in \Z$, where $\rho^\vee:=\frac{2\,\rho}{\langle\rho,\rho\rangle}$.
	\item[R3] If $\rho, \tilde{\rho} \in \Roots$, then $s_\rho(\tilde{\rho}) \in \Roots$, where $s_\rho$ is the  reflection defined by $ s_\rho(u) = u - \langle u,\rho^\vee\rangle \rho$ for $u \in  \R^n$.
\end{enumerate}
The elements of $\Roots$ are called \textbf{roots} and the \textbf{rank} of $\Roots$ is $\rank(\Roots):=\dim(\R^n)$. The elements $\rho^\vee$ are called the \textbf{coroots}. Furthermore, $\Roots$ is called \textbf{reduced}, if additionally the following condition holds.
\begin{enumerate}
	\item[R4] For $\rho \in \Roots$ and $c \in \R$, we have $c\rho \in \Roots$ if and only if $c = \pm 1$.$\phantom{\frac{1}{2}}$
\end{enumerate}

We assume that the ``reduced'' property R4 always holds when we speak of a ``root system''. 
Sometimes the ``crystallographic'' property R2 is emphasized as a seperate condition \cite{kane13}. For visualizations, see \Cref{example_2RootSys}.

\subsubsection{Weyl group and weights}

The \textbf{Weyl group} $\weyl$ of $\Roots$ is the group generated by the reflections $s_\roots$ for $\roots \in \Roots$. 
This is a finite subgroup of the orthogonal group on $\R^n$ with respect to the inner product $\sprod{\cdot , \cdot }$. 
The Weyl groups are the groups we consider in this article and now we define the lattices of interest.

A subset $\Base=\{\rho_1,\ldots,\rho_n\}\subseteq \Roots$ is called a \textbf{base}, if the following conditions hold.
\begin{enumerate}
	\item[B1] $\Base$ is a basis of $\R^n$.
	\item[B2] Every root $\rho \in \Roots$ can be written as $\rho = \alpha_1\,\rho_1 + \ldots +\alpha_n \, \rho_n$ or $\rho = -\alpha_1\,\rho_1 - \ldots -\alpha_n \,\rho_n$ for some $\alpha \in \N^n$.
\end{enumerate}
Every root system contains a base \cite[Chapitre VI, \S 1, Theorem 3]{bourbaki456}. A partial ordering $\succeq$ on $\R^n$ is defined by $u \succeq v$ if and only if $u-v  = \alpha_1\,\roots_1 + \ldots +\alpha_n\, \roots_n$ for some $\alpha \in \N^n$.

A \textbf{weight} of $\Roots$ is an element $\weight\in \R^n$, such that, for all $\roots\in\Roots$, we have $\sprod{\weight,\roots^\vee}  \in \Z$. The set of weights forms a lattice $\Weights$, called the \textbf{weight lattice}. By the condition R2, every root is a weight. For a base $\Base=\{\rho_1,\ldots,\rho_n\}$, the \textbf{fundamental weights} are the elements $\{ \fweight{1}, \ldots , \fweight{n}\}$, such that, for $1\leq i,j \leq n$, $\sprod{ \fweight{i}, \roots_j^\vee} = \delta_{i,j}$. The weight lattice is left invariant under the Weyl group, that is, $\weyl\Weights = \Weights$.

The \textbf{fundamental Weyl chamber} of $\weyl$ relative to $\Base$ is
\[
	\PC
:=	\{ u \in \R^n \,\vert\, \forall \, \roots\in\Base:\, \sprod{ u,\roots_i} > 0 \}.
\]
The closure $\overline{\PC}$ is a fundamental domain of $\weyl$ \cite[Chapitre V, \S 3, Th\'{e}or\`{e}me 2]{bourbaki456}. Hence, $\overline{\PC}$ contains exactly one element per $\weyl$--orbit and the weights in $\overline{\PC}$ are called \textbf{dominant}. We denote $\Weights^+ := \Weights\cap\overline{\PC}$.
 
\begin{proposition}\label{remark_PermutationOrb}
For $\weight\in\Weights^+$, there exists a unique $\conj{\weight}\in\Weights^+$ with $-\weight\in\weyl\conj{\weight}$. Furthermore, there exists a permutation $\sigma\in\mathfrak{S}_n$ of order at most $2$, such that, for all $1 \leq i \leq n$, we have $\conj{\fweight{i}} = \fweight{\sigma(i)}$.
\end{proposition}
\begin{proof}
Let $A$ be the longest element of $\weyl$ \cite[Chapitre VI, \S 1, Proposition 17 et Corollaire 3]{bourbaki456}. Then $A \PC = -\PC$ and so $\conj{\weight} = -A \, \weight\in\Weights^+$. We define the permutation $\sigma\in\mathfrak{S}_n$ via the property $\rho_i=-A\,\rho_{\sigma(i)}$ for $1\leq i\leq n$. Since $A$ is an involution and the inner product is $\weyl$--invariant, we obtain
\[
\conj{\fweight{i}}
=	-A\,\fweight{i}
=	\sum\limits_{j=1}^n \sprod{-A\,\fweight{i},\roots_{j}^\vee} \, \fweight{j}
=	\sum\limits_{j=1}^n \sprod{\fweight{i},-A\,\roots_{j}^\vee} \, \fweight{j}
=	\sum\limits_{j=1}^n \sprod{\fweight{i},\roots_{\sigma(j)}^\vee} \, \fweight{j}
=	\fweight{\sigma(i)}.
\]
\end{proof}

\subsubsection{The Vorono\"{i} cell}

The set of all coroots $\roots^\vee$ spans a lattice $\Corootlattice$ in $\R^n$, called the \textbf{coroot lattice}. This Abelian group acts on $\R^n$ by translation and is the dual lattice of the weight lattice, that is, $\Weights^*=\{u\in \R^n\,\vert\,\forall\,\weight\in\Weights:\,\sprod{\weight,u}\in\Z\}=\Corootlattice$.

Denote by $\norm{\cdot}$ the Euclidean norm. The \textbf{Vorono\"{i} cell} of $\Corootlattice$ is 
\[
\Vor(\Corootlattice)
:=	\{ u\in \R^n \,\vert\, \forall \lambda\in\Corootlattice:\,\norm{u}\leq \norm{u - \lambda} \}
\]
and tiles $\R^n$ by $\Corootlattice$--translation, that is,
\begin{equation}\label{eq_VoronoiTiles}
	\R^n
	=	\bigcup\limits_{\lambda\in\Corootlattice} (\Vor(\Corootlattice)+\lambda),
\end{equation}
where ``$+$'' denotes the Minkowski sum. The intersection of two distinct cells $\Vor(\Corootlattice)+\lambda$ and $\Vor(\Corootlattice)+\tilde{\lambda}$ is either empty or a common facet, that is a face of dimension $n-1$ \cite[Chapter 2, \S 1.2]{conway1988a}.

The \textbf{affine Weyl group} is the group generated by the reflections $s_{\roots,\ell}$ for $\roots\in\Roots$ and $\ell\in\Z$, where $s_{\roots,\ell}$ is defined via $s_{\roots,\ell}(u)=s_{\roots}(u)+\ell\,\roots^\vee$, see \cite[Chapitre VI, \S 2, D\'{e}finition 1]{bourbaki456}. It can also be seen as the semi--direct product $\weyl\ltimes \Corootlattice$ \cite[Chapitre VI, \S 2, Proposition 1]{bourbaki456}. We are interested in the chambers of this infinite reflection group, which are called \textbf{alcoves} to avoid confusion. In particular, the closure of any alcove is a fundamental domain for $\weyl\ltimes \Corootlattice$.

\begin{proposition}
\emph{\cite[Chapitre VI, \S2, Proposition 4]{bourbaki456} and \cite[Chapter 21, \S 3, Theorem 5]{conway1988a}} There is a unique alcove of $\weyl\ltimes \Corootlattice$ in $\PC$, which contains $0$ in its closure $\fundom$. We have $\Vor(\Corootlattice) = \weyl\,\fundom$.
\end{proposition}

The rest of this subsection is devoted to describe $\fundom$. Assume that $\R^n=V^{(1)}\oplus\ldots\oplus V^{(k)}$ is the direct sum of proper orthogonal subspaces and that, for each $1\leq i\leq k$, $\Roots^{(i)}$ is a root system in $V^{(i)}$. Then $\Roots:=\Roots^{(1)}\cup\ldots\cup\Roots^{(k)}$ is a root system in $\R^n$ and called the \textbf{direct sum} of the $\Roots^{(i)}$. If a root system is not the direct sum of at least two root systems, then it is called \textbf{irreducible}, see \cite[Chapitre VI, \S 1.2]{bourbaki456}.

The Weyl group $\weyl$ is the product of the Weyl groups corresponding to the irreducible components, see the discussion before \cite[Chapitre VI, \S 1, Proposition 5]{bourbaki456}. Furthermore, any alcove of the affine Weyl group is the product of alcoves corresponding to the irreducible components, see the discussion after \cite[Chapitre VI, \S 2, Proposition 2]{bourbaki456}. We are thus left to determine $\fundom$ for irreducible root systems. If $\Roots$ is irreducible with base $\Base$, then there is a unique positive root $\highestroot \in \Roots^+$, which is maximal with respect to the partial ordering $\succeq$ induced by $\Base$ \cite[Chapitre VI, \S 1, Proposition 25]{bourbaki456}. We call $\highestroot$ the \textbf{highest root}.

\begin{proposition}\label{prop_FundomAffineWeyl}
\emph{\cite[Chapitre VI, \S 2, Proposition 5 et Corollaire]{bourbaki456}} Let $\Roots$ be an irreducible root system and $\Base = \{\roots_1, \ldots , \roots_n\}$ be a base, such that $\highestroot = \alpha_1 \, \roots_1^\vee + \ldots + \alpha_n \, \roots_n^\vee$ is the highest root of $\Roots$ for some $\alpha\in\R^n$. Then
\[
	\fundom
=	\{ u \in \R^n \,\vert\, \forall \, 1\leq i \leq n: \, \sprod{u,\roots_i} \geq 0 \mbox{ and } \sprod{u,\highestroot} \leq 1 \}
\]
is a fundamental domain for $\weyl\ltimes \Corootlattice$. Furthermore, for $1\leq i\leq n$, we have $\alpha_i > 0$ and
\[
	\fundom
=	\mathrm{ConvHull} \left(0,\,\frac{\fweight{1}}{\alpha_1},\,\ldots,\,\frac{\fweight{n}}{\alpha_n}\right).
\]
\end{proposition}

In particular, if $\Roots$ is irreducible, then any closed alcove of the affine Weyl group is a simplex.

Every root system can be uniquely decomposed into irreducible components \cite[Chapitre VI, \S 1, Proposition 6]{bourbaki456} and there are only finitely many cases \cite[Chapitre VI, \S 4, Th\'{e}or\`{e}me 3]{bourbaki456} denoted by $\RootA$, $\RootB$, $\RootC$ $(n\geq 2)$, $\RootD$ $(n\geq 4)$, $\RootE[6,7,8]$, $\RootF[4]$ and $\RootG[2]$. Throughout this article, we shall focus on the four infinite families $\RootA$, $\RootB$, $\RootC$, $\RootD$ and the special case $\RootG[2]$. For those root systems, the base, fundamental weights and Weyl group are recalled in \Cref{Appendix_IrredRootSys}.

\begin{example}\label{example_2RootSys}
In the $2$--dimensional case, we can consider the following irreducible root systems.

\begin{minipage}{0.45\textwidth}
	\begin{figure}[H]
		\begin{minipage}{0.4\textwidth}
			\begin{flushright}
				\begin{overpic}[width=\textwidth,grid=false,tics=10]{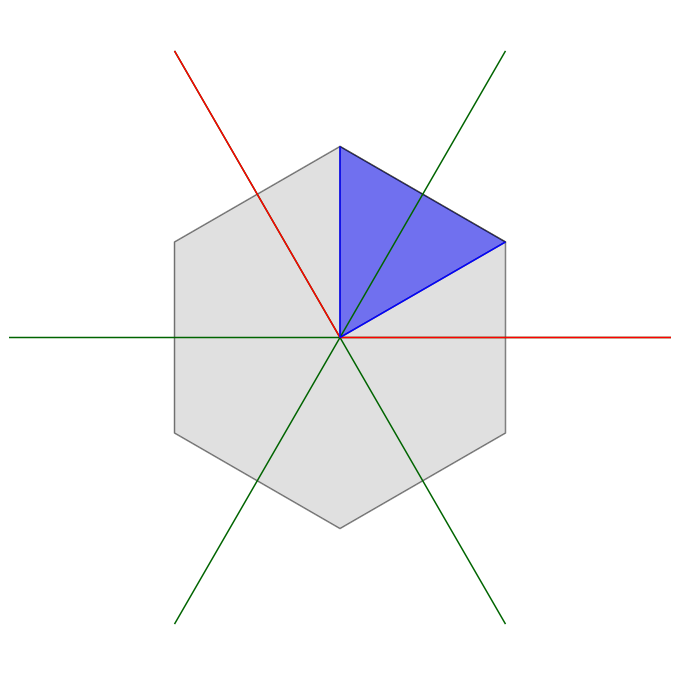}
					\put (20,95) {\large $\displaystyle \rho_2$}
					\put (95,55) {\large $\displaystyle \rho_1$}
					\put (75,66) {\large $\displaystyle \fweight{1}$}
					\put (50,81) {\large $\displaystyle \fweight{2}$}
				\end{overpic}
			\end{flushright}
		\end{minipage} \hfill
		\begin{minipage}{0.5\textwidth}
			$\weyl(\RootA[2]) \cong \mathfrak{S}_3$\\
			$\fweight{1}=[ 2,-1,-1]^t /3$\\
			$\fweight{2}=[ 1, 1,-2]^t /3$\\
			$\rho_{1}=[1,-1,0]^t =\rho_{1}^\vee$\\
			$\rho_{2}=[0,1,-1]^t =\rho_{2}^\vee$\\
			$\highestroot = \rho_{1}^\vee + \rho_{2}^\vee$
		\end{minipage}
		\centering
		\caption{The root system $\RootA[2]$ in $\R^3/\langle [1,1,1]^t \rangle$.}\label{figA2}
		\label{example_rootsystemA2}
	\end{figure}
\end{minipage}\hfill
\begin{minipage}{0.45\textwidth}
	\begin{figure}[H]
		\begin{minipage}{0.4\textwidth}
			\begin{flushright}
				\begin{overpic}[width=\textwidth,,tics=10]{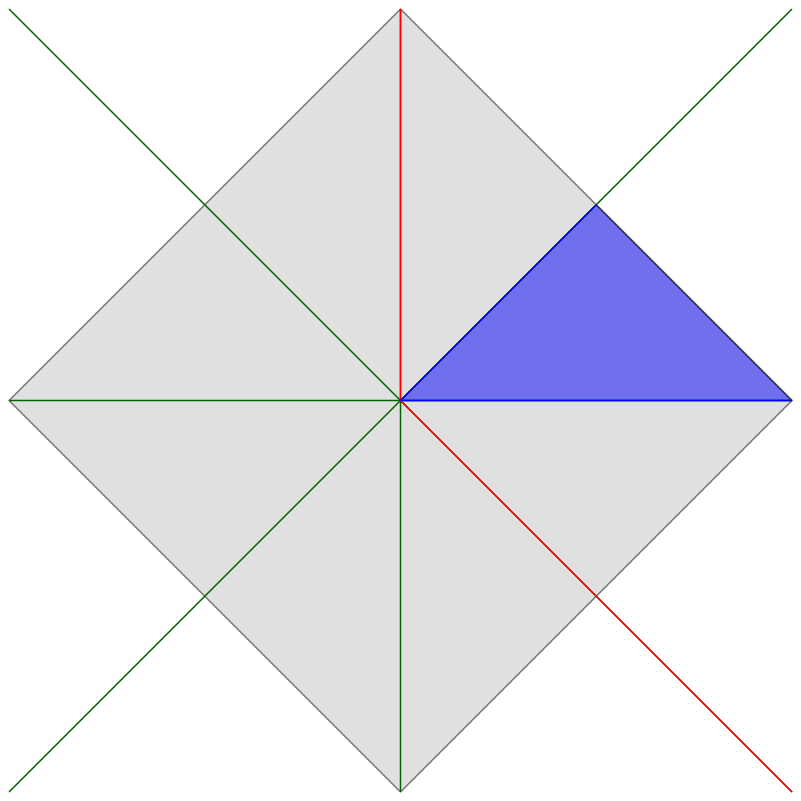}
					\put (53,100) {\large $\displaystyle \rho_2$}
					\put (99,  8) {\large $\displaystyle \rho_1$}
					\put (69, 83) {\large $\displaystyle \fweight{2}$}
					\put (95, 55) {\large $\displaystyle \fweight{1}$}
				\end{overpic}
			\end{flushright}
		\end{minipage} \hfill
		\begin{minipage}{0.5\textwidth}
			$\weyl(\RootB[2]) \cong \mathfrak{S}_2\ltimes\{\pm 1\}^2$\\
			$\fweight{1}=[ 1,0]^t$\\
			$\fweight{2}=[ 1,1]^t /2$\\
			$\rho_{1}=[1,-1]^t =\rho_{1}^\vee$\\
			$\rho_{2}=[0,1]^t =\rho_{2}^\vee/2$\\
			$\highestroot=\rho_{1}^\vee+\rho_{2}^\vee$
		\end{minipage}
		\centering
		\caption{The root system $\RootB[2]$ in $\R^2$.}\label{example_rootsystemB2}\label{figB2}
	\end{figure}
\end{minipage}

~\vspace{.2cm}

\begin{minipage}{0.45\textwidth}
	\begin{figure}[H]
		\begin{minipage}{0.4\textwidth}
			\begin{flushright}
				\begin{overpic}[width=\textwidth,grid=false,tics=10]{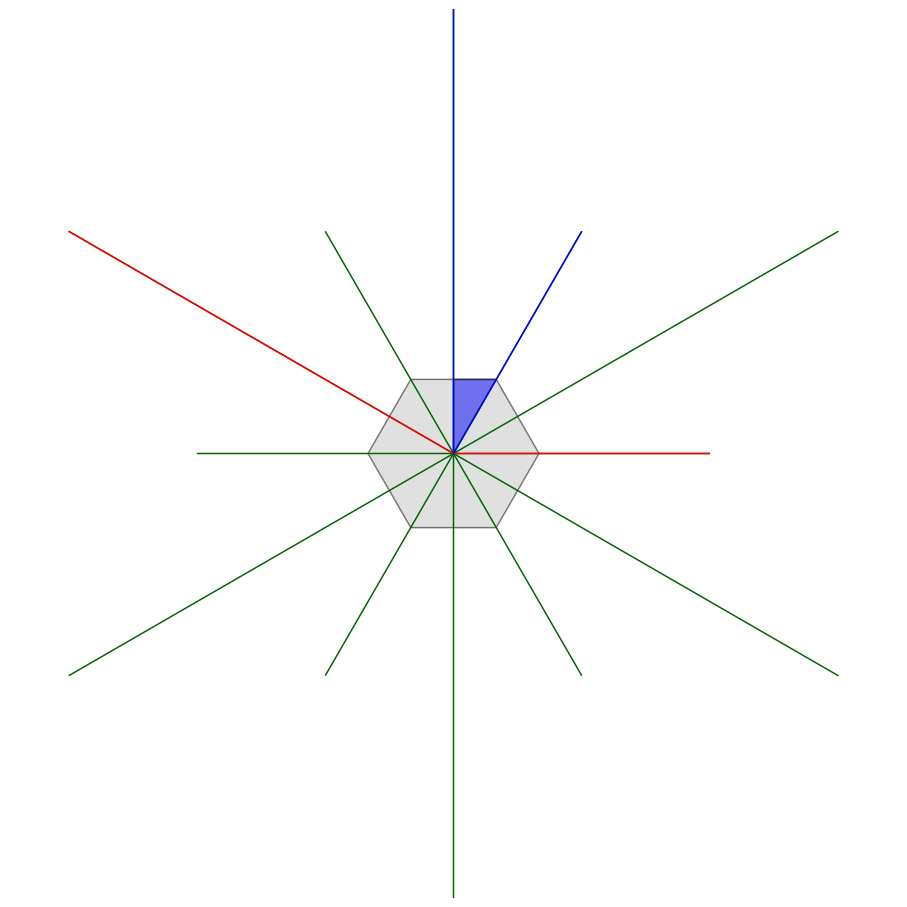}
					\put (15,75) {\large $\displaystyle \rho_2$}
					\put (80,55) {\large $\displaystyle \rho_1$}
					\put (52,95) {\large $\displaystyle \fweight{2}$}
					\put (65,75) {\large $\displaystyle \fweight{1}$}
				\end{overpic}
			\end{flushright}
		\end{minipage} \hfill
		\begin{minipage}{0.5\textwidth}
			$\weyl(\RootG[2])\cong\mathfrak{S}_3\ltimes \{\pm 1\}$\\
			$\fweight{1}=[0,-1,1,]^t$\\
			$\fweight{2}=[-1,-1,2]^t$\\
			$\rho_{1}=[1,-1,0]^t = \rho_{1}^\vee$\\
			$\rho_{2}=[-2,1,1]^t = 3\,\rho_{1}^\vee $\\
			$\highestroot = 3\,\rho_{1}^\vee + 6\,\rho_{2}^\vee$
		\end{minipage}
		\centering
		\caption{The root system $\RootG[2]$ in $\R^3/\langle [1,1,1]^t \rangle$.}\label{example_rootsystemG2}
	\end{figure}
\end{minipage}\hfill
\begin{minipage}{0.45\textwidth}
	\begin{figure}[H]
		\begin{minipage}{0.4\textwidth}
			\begin{flushright}
				\begin{overpic}[width=\textwidth,,tics=10]{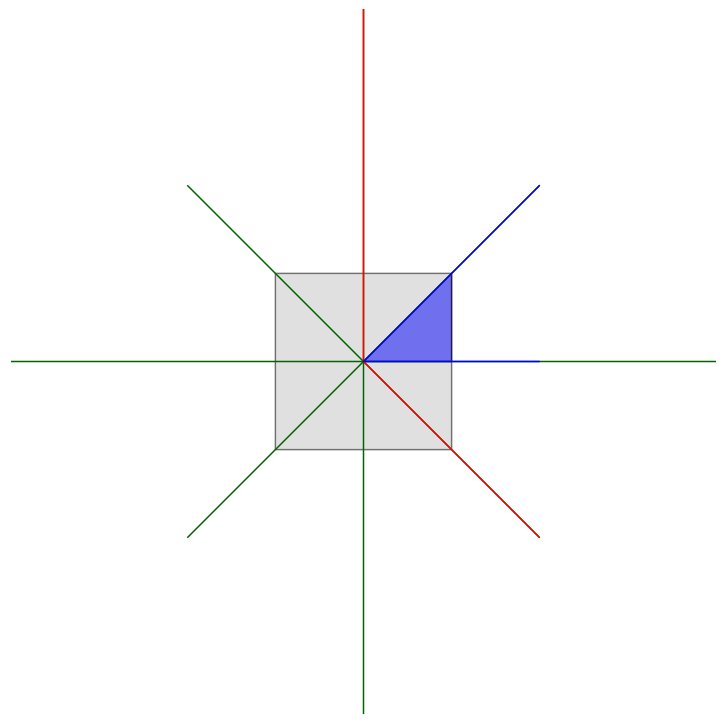}
					\put (53, 95) {\large $\displaystyle \rho_2$}
					\put (77, 20) {\large $\displaystyle \rho_1$}
					\put (69, 80) {\large $\displaystyle \fweight{2}$}
					\put (77, 54) {\large $\displaystyle \fweight{1}$}
				\end{overpic}
			\end{flushright}
		\end{minipage} \hfill
		\begin{minipage}{0.5\textwidth}
			$\weyl(\RootC[2]) \cong \mathfrak{S}_2\ltimes\{\pm 1\}^2$\\
			$\fweight{1}=[ 1,0]^t $\\
			$\fweight{2}=[ 1,1]^t $\\
			$\rho_{1}=[1,-1]^t =\rho_{1}^\vee$\\
			$\rho_{2}=[0,2]^t = 2\,\rho_{2}^\vee$\\
			$\highestroot = 2\,\rho_{1}^\vee + 2\,\rho_{2}^\vee$
		\end{minipage}
		\centering
		\caption{The root system $\RootC[2]$ in $\R^2$.}\label{example_rootsystemC2}\label{figC2}
	\end{figure}
\end{minipage}

~\\
~\\
Here, the roots are depicted in green, the base in red and the fundamental weights in blue. The Vorono\"{i} cell of the coroot lattice $\Corootlattice$ is the gray shaded region, we have two squares $(\RootC[2]$ and $\RootB[2])$ and two hexagons $(\RootA[2]$ and $\RootG[2])$. The fundamental domain of the affine Weyl group is the blue shaded simplex.
\end{example}

\subsection{Trigonometric polynomials with Weyl group symmerty}

From now on, $\Roots$ is a root system in $\R^n$ with Weyl group $\weyl$, weight lattice $\Weights = \Z\,\fweight{1} \oplus \ldots \oplus \Z\,\fweight{n}$ and coroot lattice $\Corootlattice=\Weights^*$. For $\weight \in \Weights$, we define the function
\[
\mathfrak{e}^{\weight} :
\begin{array}[t]{ccl}
\R^n & \to & \C, \\
u & \mapsto & \mexp{\weight,u}.
\end{array}
\]
A $\C$--linear combination of these functions is a \textbf{trigonometric polynomial}. The set of all trigonometric polynomials forms an algebra  that we denote by $\C[\Weights]$.

The set $\{\mathfrak{e}^{\weight}\,\vert\,\weight\in\Weights\}$ is closed under multiplication $\mathfrak{e}^{\weight}\,\mathfrak{e}^{\tilde{\weight}} = \mathfrak{e}^{\weight + \tilde{\weight}}$ and thus a group with neutral element $\mathfrak{e}^0$ and inverse $(\mathfrak{e}^{\weight})^{-1}=\mathfrak{e}^{-\weight}$. Since $\Weights$ is the free $\Z$--module generated by the $\fweight{i}$, $\C[\Weights]$ is generated by $\{\mathfrak{e}^{\pm\fweight{1}},\ldots,\mathfrak{e}^{\pm\fweight{n}}\}$.

Since the coroot lattice $\Corootlattice$ is the dual lattice of $\Weights$, any element $f\in\C[\Weights]$ is $\Corootlattice$--periodic, that is, for all $u\in \R^n$ and $\lambda\in\Corootlattice$, we have $f(u+\lambda) = f(u)$.

\subsubsection{Generalized cosines and Chebyshev polynomials}

The Weyl group $\weyl$ acts linearly on $\C[\Weights]$ by the action described on its basis as
\[
	\cdot: \begin{array}[t]{ccl}
	\weyl	\times	\C[\Weights]	&	\to		&	\C[\Weights],\\
	(A,\mathfrak{e}^{\weight})		&	\mapsto	&	\mathfrak{e}^{A\weight}.
	\end{array}
\]
A trigonometric polynomial $f\in\C[\Weights]$ is called \textbf{$\weyl$--invariant}, if, for all $A\in\weyl$, we have $A \cdot f = f$. The \textbf{generalized cosine function} associated to $\weight\in\Weights$ is the $\weyl$--invariant trigonometric polynomial
\begin{equation}\label{eq_gencos}
\gencos{\weight} :
\begin{array}[t]{ccl}
\R^n & \to & \C,\\
u & \mapsto & \dfrac{1}{\nops{\weyl}} \sum\limits_{A\in\weyl} \mathfrak{e}^{A \weight}(u) .
\end{array}
\end{equation}

\begin{theorem}\label{theorem_BourbakiGenerators}
\emph{\cite[Chapitre VI, \S 3, Th\'eor\`eme 1]{bourbaki456}} The following statements hold.
\begin{enumerate}
\item The $\gencos{\fweight{1}}, \ldots, \gencos{\fweight{n}}$ are algebraically independent.
\item The set of $\weyl$--invariants is the polynomial $\C$--algebra generated by the $\gencos{\fweight{1}},\ldots,\gencos{\fweight{n}}$, that is, 
\[
	\C[\Weights]^\weyl=\C[\gencos{\fweight{1}},\ldots,\gencos{\fweight{n}}].
\]
\end{enumerate}
\end{theorem}

The above \Cref{theorem_BourbakiGenerators} states that, for every $f\in\C[\Weights]^\weyl$, there exists a unique polynomial $g\in\CX:=\C[z_1,\ldots,z_n]$ with the property $f(u) = g(\gencos{}(u))$, where $\gencos{}$ is the function
\[
\gencos{} :
\begin{array}[t]{ccl}
	\R^n & \to & \C^n,\\
	u & \mapsto & \left(\gencos{\fweight{1}}(u),\ldots,\gencos{\fweight{n}}(u)\right) .
\end{array}
\]
This property is exclusive for Weyl groups \cite{farkas86}.

\begin{definition}
The \textbf{generalized Chebyshev polynomials of the first kind} associated to $\weight\in\Weights$ is the unique $T_\weight\in\CX$, such that $T_\weight(\gencos{}(u)) = \gencos{\weight}(u)$. 
\end{definition}

The coefficients of the $T_\weight$ are real. We have $T_0 = 1$, $T_{\fweight{i}}=z_i$ and, for $\weight,\nu\in\Weights$,
\begin{equation}\label{eq_TPolyRecurrence}
\vert\weyl \vert\,T_{\weight}\,T_{\nu}=\sum\limits_{A \in\weyl} T_{\weight + A \nu}.
\end{equation}
The set $\{T_\weight\,\vert\,\weight\in\Weights^+\}$ forms a vector space basis of $\CX$ \cite{lorenz06}.

This definition is a generalization of the univariate Chebyshev polynomials of the first kind $T_\ell(\cos(u)) = \cos(\ell \, u)$ with $\ell \in \Z$, which correspond to the root system $\RootA[1]$. 

\subsubsection{Real cosines and Chebyshev polynomials}

For our approach in \Cref{section_optimization}, we need the generalized Chebyshev polynomials to be defined on a real domain. This is always true for $-I_n\in\weyl$ and what follows is only necessary for $-I_n\notin\weyl$. Let $\weight,\conj{\weight}\in\Weights^+$ with $-\weight\in\weyl\conj{\weight}$. The \textbf{real generalized cosines} associated to the pair $(\weight,\conj{\weight})$ are
\[
	\Re (\gencos{\weight})
=	\frac{\gencos{\weight}+\gencos{\conj{\weight}}}{2}
	\tbox{and}
	\Im (\gencos{\weight})
=	\frac{\gencos{\weight}-\gencos{\conj{\weight}}}{2\mathrm{i}}.
\]
By construction, those are real--valued $\weyl$--invariant trigonometric polynomials. We are interested in the pairs $(\weight,\conj{\weight})$ with $\weight=\fweight{i}$ a fundamental weight. Let $\sigma\in \mathfrak{S}_n$ be the permutation from \Cref{remark_PermutationOrb}. Then $\conj{\weight} = \fweight{\sigma(i)}$ is also a fundamental weight and we define the function
\begin{equation}\label{eq_realgencos}
\gencos{\R}:
\begin{array}[t]{ccl}
\R^n		& \to 		&	\R^{n}, \\
u 		& \mapsto 	&	\left( \gencos{\fweight{1},\R}( u),\ldots, \gencos{\fweight{n},\R}( u)\right),
\end{array}
\end{equation}
where $\gencos{\fweight{i},\R}:=\gencos{\fweight{i}}$ for $i=\sigma(i)$ and $\gencos{\fweight{i},\R} := \Re (\gencos{\fweight{i}}),\gencos{\fweight{\sigma(i)},\R} := \Im (\gencos{\fweight{i}})$ for $i<\sigma(i)$.

\begin{proposition}\label{LemmaComplexConjugatePoly}
Let $\weight,\conj{\weight} \in \Weights$ with $-\weight\in\weyl\conj{\weight}$. Then there exist unique $\TT_\weight , \TT_{\conj{\weight}} \in \RX$, such that
\[
	T_{{\weight}}	(\gencos{}(u))
=	\TT_{{\weight}}	(\gencos{\R}(u)) + \mathrm{i} \, \TT_{\conj{\weight}}	(\gencos{\R}(u))
\tbox{and}
	T_{\conj{\weight}}	(\gencos{}(u))
=	\TT_{{\weight}}	(\gencos{\R}(u)) - \mathrm{i}\, \TT_{\conj{\weight}}	(\gencos{\R}(u)) .
\]
\end{proposition}
\begin{proof}
Note that
\[
	(T_\weight + T_{\conj{\weight}}) (\gencos{}(u))
=	\frac{1}{\nops{\weyl \weight}} \sum\limits_{\tilde{\weight} \in \weyl \weight} \mathfrak{e}^{\tilde{\weight}}(u) + \mathfrak{e}^{-\tilde{\weight}}(u)
\]
is invariant under both $\weyl$ and $\{\pm I_n\}$. Let $\sigma\in \mathfrak{S}_n$ be the permutation from \Cref{remark_PermutationOrb}. Then the $\C$--algebra $(\C[\Weights]^{\weyl})^{\{\pm I_n\}}$ is generated by the $ \gencos{\fweight{i}}+\gencos{\fweight{\sigma(i)}} $ with $ 1 \leq i \leq \sigma(i) \leq n $. Thus, $(T_\weight + T_{\conj{\weight}})(\gencos{}(u))/2$ can be written as a polynomial $\TT_\weight$ in $\gencos{\R}(u)$. Similarly,
\[
	(T_\weight - T_{\conj{\weight}}) (\gencos{}(u))
=	\frac{1}{\nops{\weyl\weight}} \sum\limits_{\tilde{\weight} \in \weyl\weight} \mathfrak{e}^{\tilde{\weight}}(u) - \mathfrak{e}^{-\tilde{\weight}}(u)
\]
is invariant under $\weyl$, but anti--invariant under $\{\pm I_n\}$. The elements of $\C[\Weights]^{\weyl}$, which are anti--invariant under $\{\pm I_n\}$, are, as an $\C$--algebra, generated by the $ \gencos{\fweight{i}}-\gencos{\fweight{\sigma(i)}}$ with $ 1\leq \sigma(i) < i \leq n $. Hence, $(T_\weight - T_{\conj{\weight}})(\gencos{}(u))/(2\mathrm{i})$ can be written as a polynomial $\TT_{\conj{\weight}}$ in $\gencos{\R}(u)$. As polynomials, $\TT_{\weight}$ and $\TT_{\conj{\weight}}$ are analytical functions and $\Image_\R$ has nonempty interior. Hence, they are unique.
\end{proof}

\begin{convention}\label{convention_real}
From now on, we will write $T_\weight$ and $\gencos{}$ for $\TT_\weight$ and $\gencos{\R}$, even if $-I_n\notin\weyl$. As we have shown above, the reformulation follows from a permutation $\sigma$ and a substitution $z_i \mapsto z_i \pm \mathrm{i}\,z_{\sigma(i)}$. For our implementation, it is important to remember this caveat, but for the article itself, we shall simplify the notation.
\end{convention}

\subsection{The image of the generalized cosines as a basic semi--algebraic set}

We call $\Image:=\gencos{}(\R^n)$ the \textbf{image of the generalized cosines}. If $\fundom$ is a fundamental domain for the affine Weyl group $\weyl\ltimes\Corootlattice$, then $\Image=\gencos{}(\fundom)$ due to the $\weyl$--invariance and $\Corootlattice$--periodicity. In particular, $\Image$ is compact. With \Cref{convention_real}, $\Image$ is a real set and contained in the cube $[-1,1]^n$.

For the purpose of optimization, we need a polynomial description of $\Image$ as a basic semi--algebraic set. Recently, a closed formula was given via a polynomial matrix inequality. This formula is available in the standard monomial basis $z$ \cite{chromaticissac22,TOrbits}, and in the basis of generalized Chebyshev polynomials $T_{\weight}$ \cite{TobiasThesis}.

\begin{theorem}\label{thm_HermiteCharacterization}
\emph{\cite[Theorem 2.19]{TobiasThesis}} Let $\Roots$ be a root system of type $\RootA$, $\RootB$, $\RootC$, $\RootD$ or $\RootG[n-1]$ and define the symmetric matrix polynomial $\posmat\in\RX^{n\times n}$ via
\begin{align*}
	2^{i+j}\,\posmat(z)_{ij}
=&	- T_{(i+j)\, \fweight{1}}(z) + \sum\limits_{\ell=1}^{\lceil (i+j)/2 \rceil -1} \left( 4 \binom{i+j-2}{\ell-1} - \binom{i+j}{\ell} \right) T_{(i+j-2\,\ell)\, \fweight{1}}(z) \\
&+	\frac{1}{2} \begin{cases}
	4\binom{i+j-2}{(i+j)/2-1} - 	\binom{i+j}{(i+j)/2}	,&	\tbox{if} i+j \tbox{is even}	\\
	0														,&	\tbox{if} i+j \tbox{is odd}
	\end{cases}.
\end{align*}
Then $\Image = \{z\in \R^n \,\vert \, \posmat(z)\succeq 0 \}$.
\end{theorem}

The matrix polynomial $\posmat\in\RX^{n\times n}$ from \Cref{thm_HermiteCharacterization} follows the pattern
\[
\begin{bmatrix}
	\frac{T_{0}-T_{2\,\fweight{1}}}{4}& 
	\frac{T_{\fweight{1}} -T_{3\,\fweight{1}}}{8}& 
	\frac{T_{0}- T_{4\,\fweight{1}}}{16}&
	\frac{2\,T_{\fweight{1}}- T_{3\,\fweight{1}} - T_{5\,\fweight{1}}}{32}&
	\cdots\\
	
	\frac{T_{\fweight{1}} -T_{3\,\fweight{1}}}{8}& 
	\frac{T_{0}- T_{4\,\fweight{1}}}{16}&
	\frac{2\, T_{\fweight{1}}- T_{3\,\fweight{1}} - T_{5\,\fweight{1}}}{32}&
	\frac{2\, T_{0} +  T_{2\,\fweight{1}} - 2\, T_{4\,\fweight{1}} -  T_{6\,\fweight{1}}}{64}&
	\cdots\\
	
	\frac{T_{0}- T_{4\,\fweight{1}}}{16}& 
	\frac{2\,T_{\fweight{1}}- T_{3\,\fweight{1}} - T_{5\,\fweight{1}}}{32}& 
	\frac{2 \,T_{0} +  T_{2\,\fweight{1}}-2\, T_{4\,\fweight{1}} -  T_{6\,\fweight{1}}}{64}&
	\frac{5 \,T_{\fweight{1}} - T_{3\,\fweight{1}} - 3 \,T_{5\,\fweight{1}} - T_{7\,\fweight{1}}}{128}&
	\cdots\\
	
	\frac{2\,T_{\fweight{1}}- T_{3\,\fweight{1}} - T_{5\,\fweight{1}}}{32}&
	\frac{2 \,T_{0} +  T_{2\,\fweight{1}}-2\, T_{4\,\fweight{1}} -  T_{6\,\fweight{1}}}{64}&
	\frac{5 \,T_{\fweight{1}} - T_{3\,\fweight{1}} - 3\, T_{5\,\fweight{1}} - T_{7\,\fweight{1}}}{128}&
	\frac{5 \,T_{0} + 4 \,T_{2\,\fweight{1}} - 4 \,T_{4\,\fweight{1}} - 4\,T_{6\,\fweight{1}} - T_{8\,\fweight{1}}}{256}&
	\cdots\\
	
	\vdots & \vdots & \vdots & \vdots & \ddots
\end{bmatrix}.
\]
\begin{figure}[H]
	\begin{center}
		\begin{subfigure}{.2\textwidth}
			\centering
			\includegraphics[width=\textwidth]{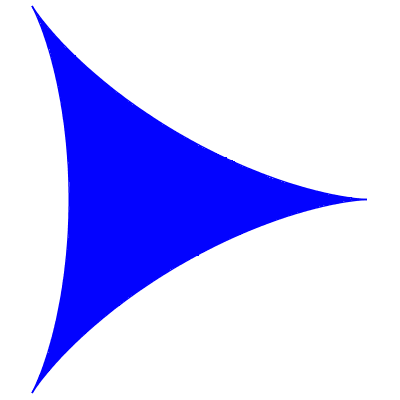}
			\caption{$\RootA[2]$}
		\end{subfigure}\quad
		\begin{subfigure}{.2\textwidth}
			\centering
			\includegraphics[width=\textwidth]{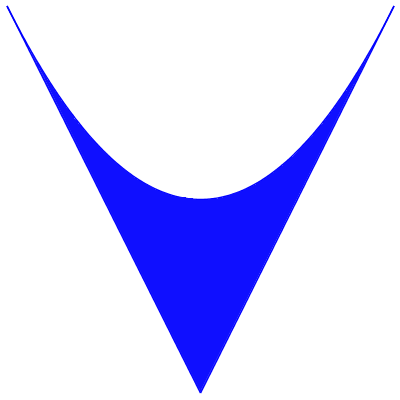}
			\caption{$\RootC[2]$}
		\end{subfigure}\quad
		\begin{subfigure}{.2\textwidth}
			\centering
			\includegraphics[width=\textwidth]{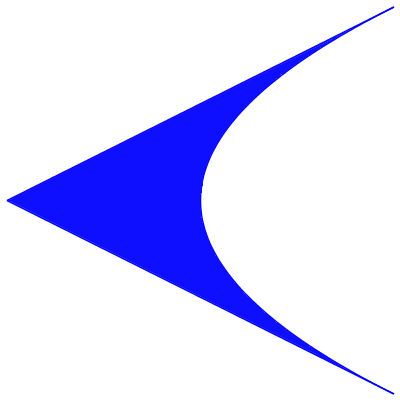}
			\caption{$\RootB[2]$}
		\end{subfigure}\quad
		\begin{subfigure}{.2\textwidth}
			\centering
			\includegraphics[width=\textwidth]{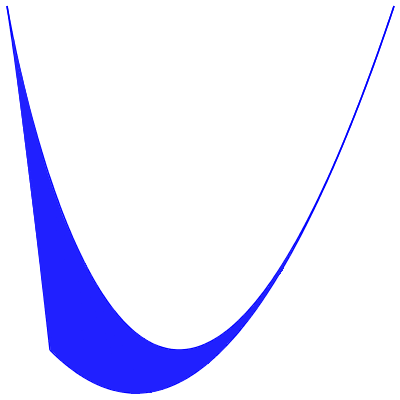}
			\caption{$\RootG[2]$}
		\end{subfigure}\quad
		\begin{subfigure}{.3\textwidth}
			\centering
			\includegraphics[width=\textwidth]{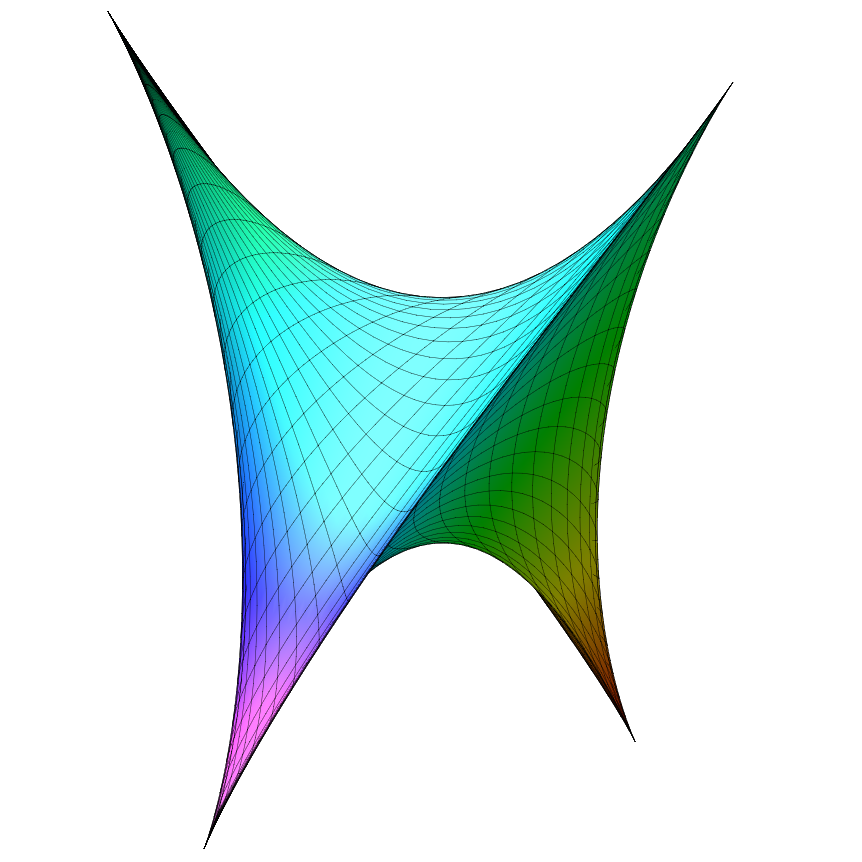}
			\caption{$\RootA[3]$}
		\end{subfigure}\quad
		\begin{subfigure}{.3\textwidth}
			\centering
			\includegraphics[width=\textwidth]{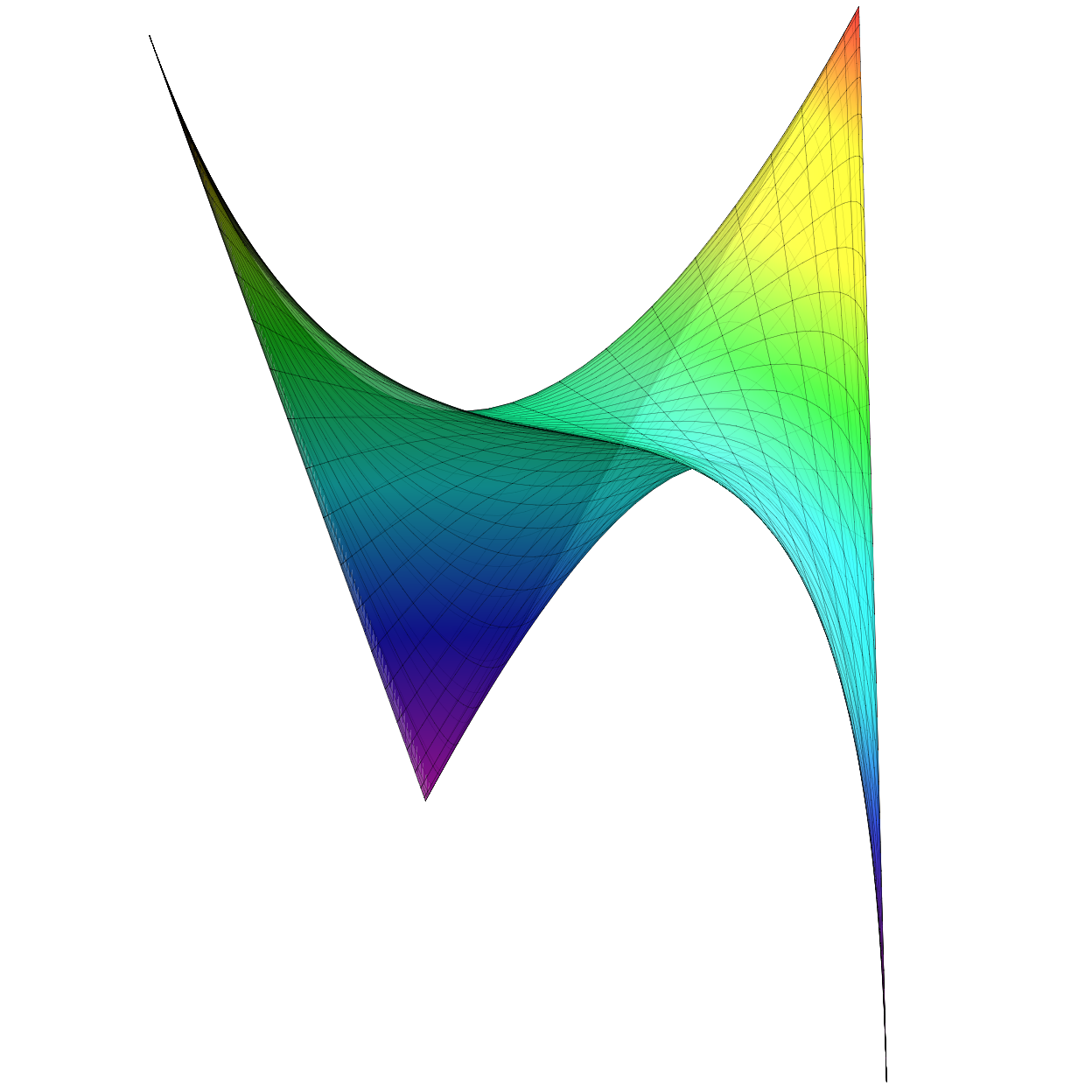}
			\caption{$\RootC[3]$}
		\end{subfigure}\quad
		\begin{subfigure}{.3\textwidth}
			\centering
			\includegraphics[width=\textwidth]{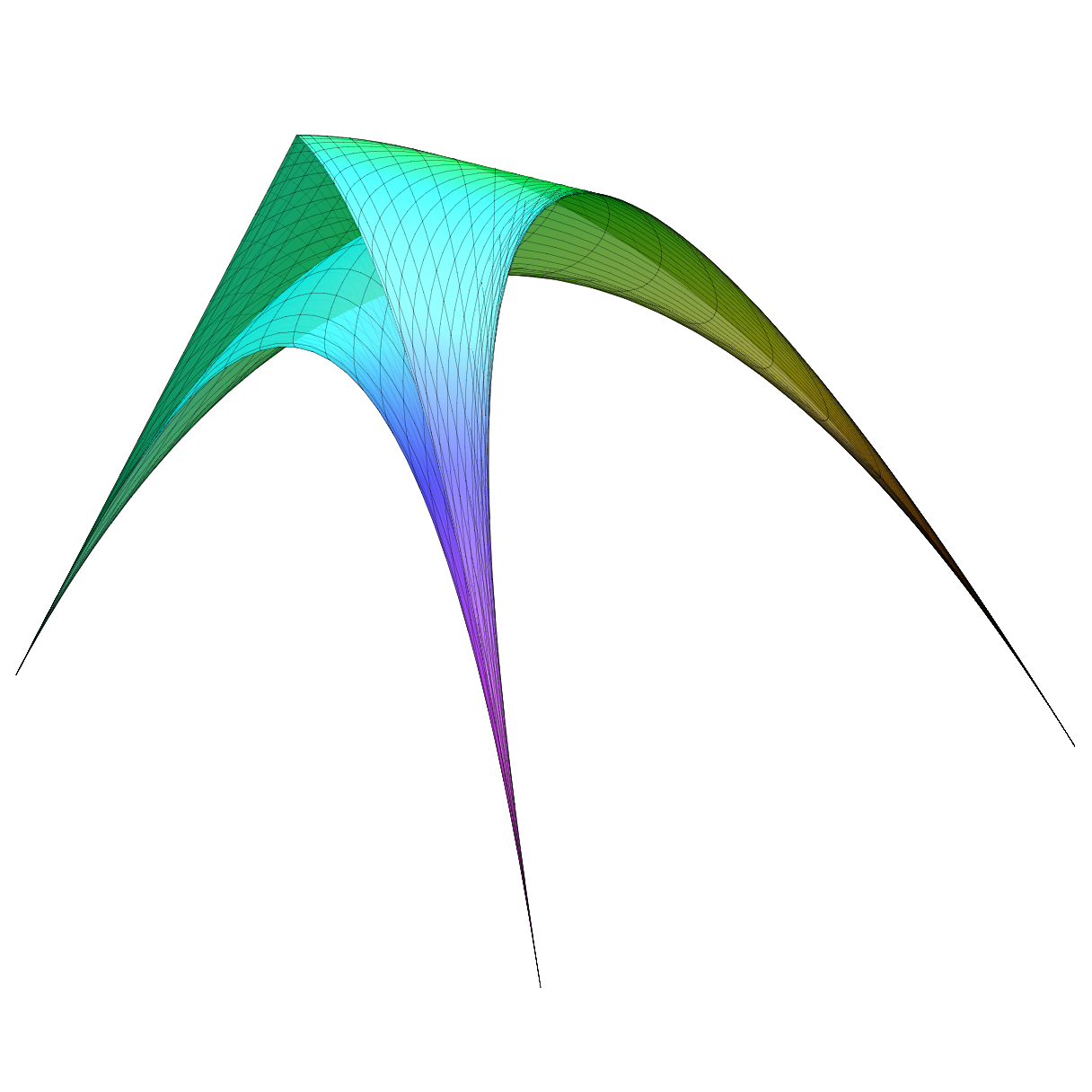}
			\caption{$\RootB[3]$}
		\end{subfigure}
		\caption{The image of the generalized cosines for the irreducible root systems of rank $2$ and $3$.}
		\label{fig_OrbitSpaceIntro}
	\end{center}
\end{figure}

\begin{remark}~
\begin{enumerate}
\item If we are in one of the special cases $\RootE[6,7,8]$ or $\RootF[4]$, then such a polynomial description of $\Image$ can also be obtained with \emph{\cite[\S 4]{procesischwarz85}}. In this case, one obtains a Gram matrix of differentials and has to rewrite the entries in the coordinates $z$ of $\Image$.

\item The root system may not be irreducible, that is, $\Roots=\Roots^{(1)} \cup \ldots \cup \Roots^{(k)}$ for some $k\in \N$.
Hence, we can write the fundamental domain of the affine Weyl group as $\fundom = \fundom^{(1)} \times \ldots \times \fundom^{(k)}$ and thus $\Image=\gencos{\fundom}$ is the positivity locus of a block--diagonal matrix polynomial
\[
\posmat(z^{(1)},\ldots,z^{(k)})
=	\diag(\posmat^{(1)}(z^{(1)}),\ldots,\posmat^{(k)}(z^{(k)})),
\]
where the $\posmat^{(i)}$ are matrix polynomials corresponding to the irreducible $\Roots^{(i)}$.

As an example, take $k$ orthogonal copies of $\RootA[1]$.
Then $\Image=[-1,1]^k$ is the positivity locus of the matrix polynomial $\posmat=\diag(1-z_1^2,\ldots,1-z_k^2)$.
\end{enumerate}
\end{remark}

\subsection{Optimizing trigonometric polynomials with crystallographic symmetry}

We now address the trigonometric optimization problem from \Cref{OptiProblemExpo}. With the theory that was presented in the previous subsections, we can rewrite the objective function uniquely in terms of generalized Chebyshev polynomials using \Cref{theorem_BourbakiGenerators}. Indeed, with the generalized cosines from \Cref{eq_gencos} we can write any $f\in\C[\Weights]^\weyl$ uniquely as
\[
f
=	\sum\limits_{\weight\in S} c_\weight\,\gencos{\weight}
\]
for some finite set $S\subseteq\Weights^+$ of dominant weights. If $c_{\weight} = \overline{c_{\conj{\weight}}} \in \R$ whenever $-\weight\in\weyl\conj{\weight}$, then $f$ takes only real values and
\begin{equation}\label{coro_OptiProbelmRewrite}
	f^*
:=	\min\limits_{u \in \R^n} f(u)
=	\min\limits_{z\in\Image} \sum\limits_{\weight\in S} c_\weight \, T_\weight(z)
\end{equation}
is the global minimum of $f$ on $\R^n$. This transforms the region of optimization from $\R^n$ into the image $\Image$ of the generalized cosines. Thanks to \Cref{thm_HermiteCharacterization}, we can describe the latter explicitly as a compact basic semi--algebraic set with the Chebyshev basis. This makes it possible to solve the problem numerically with techniques from classical polynomial optimization, which is subject to \Cref{section_optimization}.

\begin{example}\label{example_A2PolyRewrite}
The symmetric group $\mathfrak{S}_3$ acts on $\R^3/\langle [1,1,1]^t\rangle$ by permutation of coordinates and leaves the lattice $\Weights:=\Z\,\fweight{1} + \Z\,\fweight{2}:=\Z\,[0,-1,-1]^t + \Z\,[-1,-1,2]^t$ invariant. This is the weight lattice of the root system $\RootG[2]$ with Weyl group $\weyl := \mathfrak{S}_3 \times \{\pm1\}$. We consider the $\weyl$--invariant trigonometric polynomial
\begin{align*}
	f(u)
:=&	\,\gencos{\textcolor{red}{2\,\fweight{1}}}(u)+2\,\gencos{\textcolor{blue}{\fweight{2}}}(u)\\
=&	\,(
	\mcos{\sprod{2\,\fweight{1},u}} + 
	\mcos{\sprod{2\,\fweight{1} - 2\,\fweight{2},u}} + 
	\mcos{\sprod{4\,\fweight{1} - 2\,\fweight{2},u}} \\
&	\,
	+ 2\,\mcos{\sprod{\fweight{2},u}} + 
	2\,\mcos{\sprod{3\,\fweight{1} - \fweight{2},u}} + 
	2\,\mcos{\sprod{3\,\fweight{1} - 2\,\fweight{2},u}} 
	)/3
\end{align*}
with $u = (u_1,u_2,-u_1-u_2)\in \R^3/\langle [1,1,1]^t\rangle$. In the coordinates $z=\gencos{}(u)=(\gencos{\fweight{1}}(u),\gencos{\fweight{2}}(u))\in\Image$, we have
\[
	f(z)
=	T_{\textcolor{red}{2\,\fweight{1}}}(z)+2\,T_{\textcolor{blue}{\fweight{2}}}(z)
=	(6\,z_1^2 - 2\,z_1 - 2\,z_2 - 1) + 2 \, (z_2)
=	6\,z_1^2 - 2\,z_1 - 1.
\]
This univariate polynomial is minimal in $z_1=1/6$ and $z=(1/6,z_2)\in\Image$ if and only if $z_2\in [-11/24,-1/3]$. Hence, the minimum of $f$ is
\[
	f^*
=	\min\limits_{u\in\R^2} f(u)
=	\min_{z\in \Image} 6\,z_1^2 - 2\,z_1 - 1
=	-\frac{7}{6}.
\]
\begin{figure}[H]
\begin{center}
	\begin{overpic}[height=4cm,grid=false,tics=10]{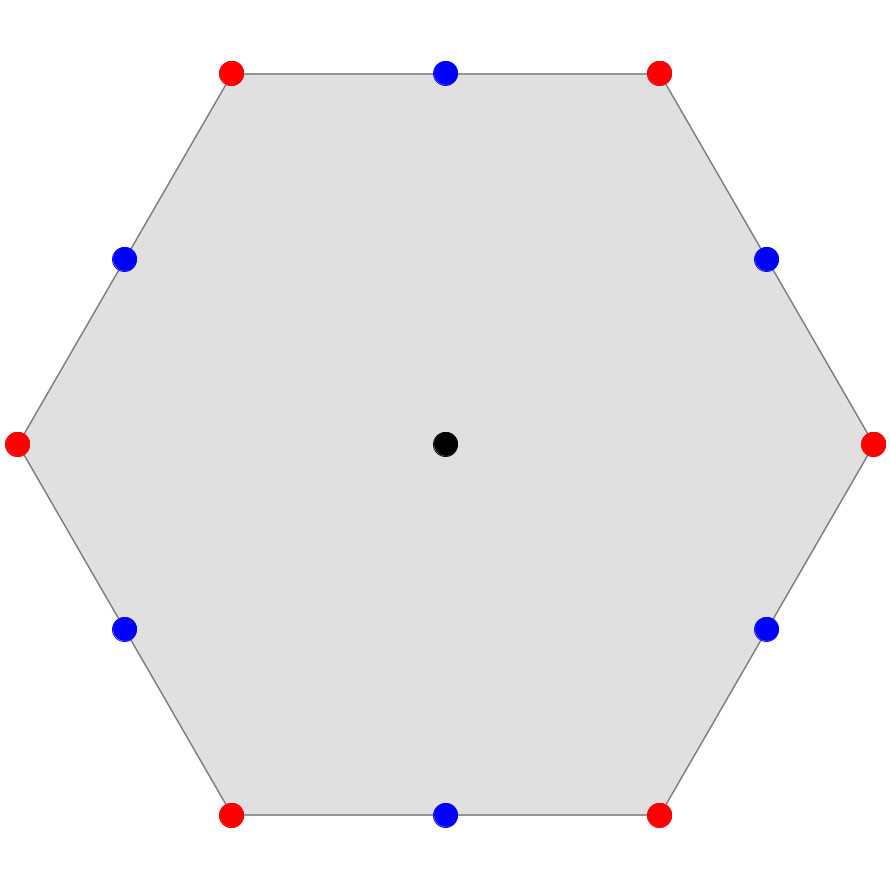}
	\put ( 83, 87) {\large $\displaystyle \textcolor{red}{2\,\fweight{1}}$}
	\put ( 53, 97) {\large $\displaystyle \textcolor{blue}{\fweight{2}}$}
	\end{overpic}
	\hspace{1cm}
	\includegraphics[height=4cm]{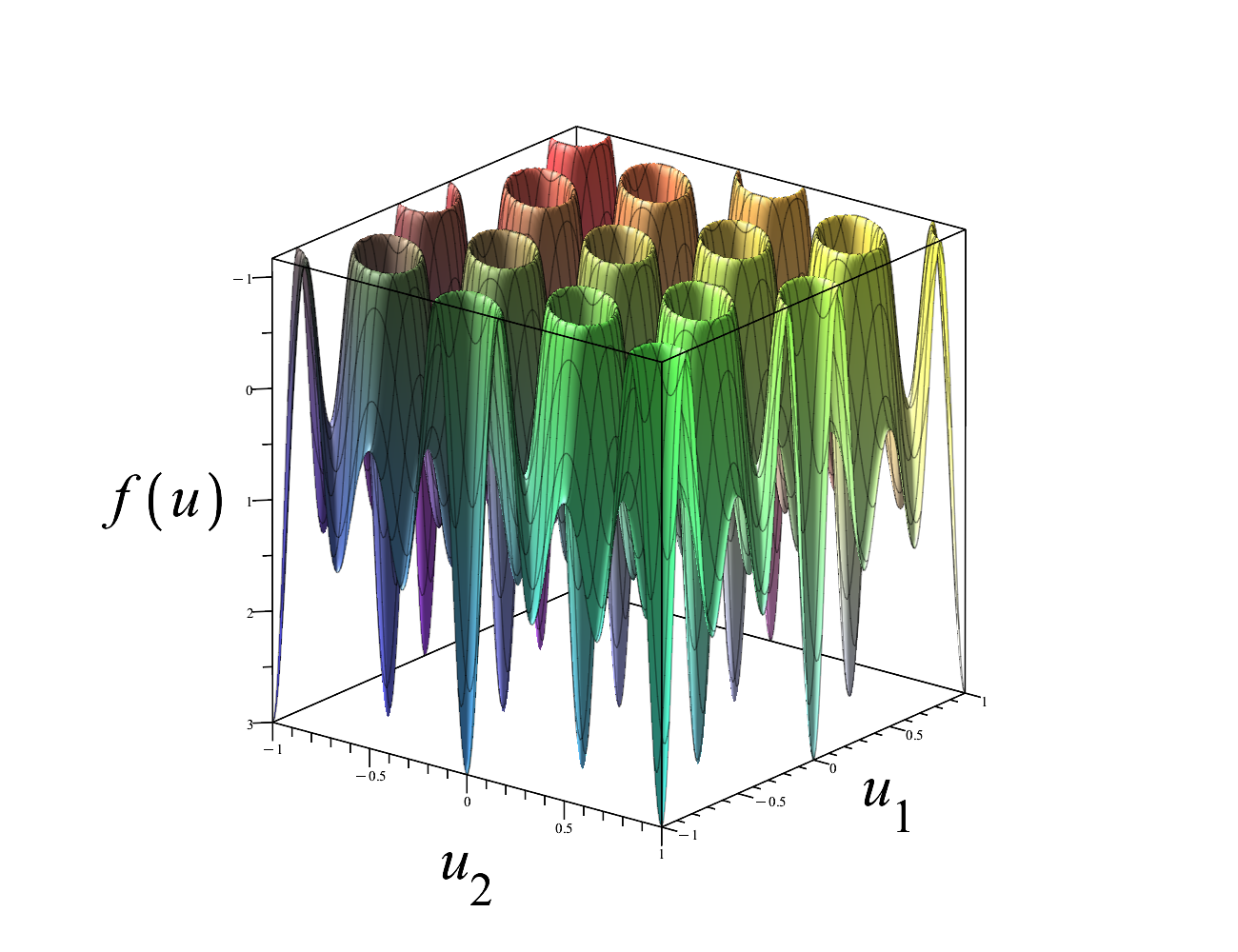}
	\hspace{1cm}
	\includegraphics[height=4cm]{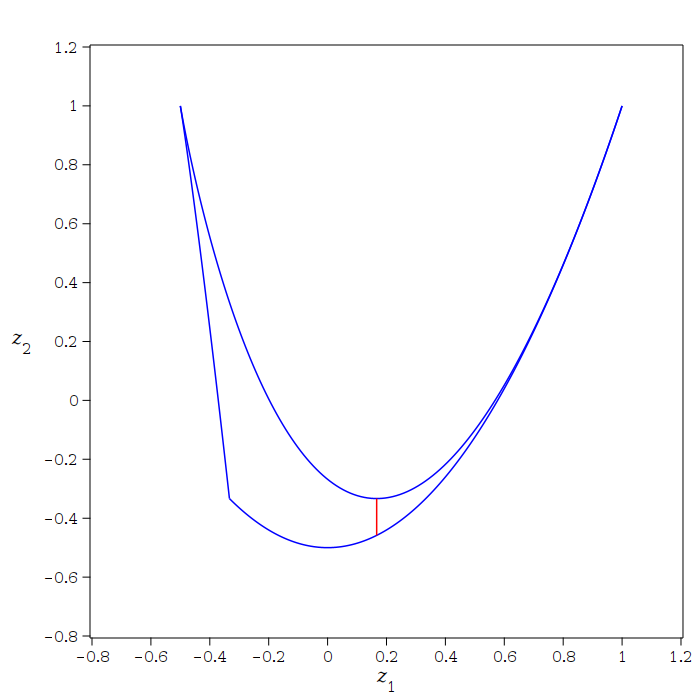}
\caption{The support of $f$ as a trigonometric polynomial on the left consists of the $\weyl$--orbits of $\textcolor{red}{2\,\fweight{1}}$ and $\textcolor{blue}{\fweight{2}}$. The graph of this $\weyl$--invariant periodic function is depicted in the middle. The image of the generalized cosines \textcolor{blue}{$\Image$} on the right is the new feasible region of the polynomial optimization problem and the set of minimizers for $f$ is \textcolor{red}{$\{1/6\} \times [-11/24,-1/3]$}.}
\label{A2level2}
\end{center}
\end{figure}
\end{example}

\section{Optimization in terms of generalized Chebyshev polynomials}
\label{section_optimization}
\setcounter{equation}{0}

In the previous section, we have shown that the trigonometric optimization problem with crystallographic symmetry from \Cref{OptiProblemExpo} is equivalent to optimizing a linear combination of generalized Chebyshev polynomials
\begin{equation}\label{eq_ObjectivePoly}
	f(z)
=	\sum\limits_{\weight\in S} c_\weight\,T_\weight(z)
\in	\RX
\end{equation}
with $S\subseteq \Weights^+$ finite and $c_\weight\in\R$. Here, $\Image$ is the image of the generalized cosines, a compact basic semi--algebraic set that can be represented as
\[
	\Image
=	\{ \gencos{}(u) \,\vert\, u\in\R^n \}
=	\{ z\in\R^n \,\vert\, \posmat(z) \succeq 0 \},
\]
where $\posmat\in\RX^{n\times n}$ is a symmetric matrix polynomial, for example given by \Cref{thm_HermiteCharacterization}. In the present section, we show how to solve this new polynomial optimization problem
\begin{equation}\label{eq_polyoptprob}
	f^*
	=	\min\limits_{z\in\Image} f(z)
	=		\begin{array}[t]{rl}
		\min		&	f(z) \\
		\mbox{s.t.}	&	z\in \R^n,\,\posmat(z)\succeq 0
	\end{array}
\end{equation}
numerically. We do this by adapting Lasserre's hierarchy. The novelty lies in exploiting the representation of the objective function in terms of generalized Chebyshev polynomials, which leads to a new notion of the hierarchy order.

\subsection{Matrix version of Putinar's theorem}

In \cite{lasserre01}, Lasserre proposes a hierarchy of dual moment relaxations and sums of squares (SOS) reinforcements based on Putinar's Positivstellensatz \cite{putinar93} to solve such problems, when the polynomial matrix inequality $\posmat(z)\succeq 0$ (PMI) is replaced by finitely many scalar constraints. In principle, our problem falls in this setting. Indeed, the PMI can be rewritten to scalar inequalities by taking the coefficients of the characteristic polynomial and using Descartes' rule of signs \cite[Theorem 2.33]{roy13}. We would prefer to avoid such an approach, since the degrees of the so obtained scalar constraints are generically much larger than the entries of the matrix polynomial $\posmat$.

To overcome this degree problem, Henrion and Lasserre \cite{henrion06} suggest using another Positivstellensatz due to Hol and Scherer, see \Cref{thm_holscherer}, and propose a hierarchy of dual moment relaxations and matrix SOS reinforcements, that benefits from the matrix structure.

\subsubsection{Matrix SOS reinforcement}

A matrix polynomial $\mathbf{Q}\in\RX^{n\times n}$ is said to be a \textbf{sum of squares}, if there exist $k\in\N$ and $\mathbf{Q}_1,\ldots,\mathbf{Q}_k\in \RX^n$, such that
\[
	\mathbf{Q}(z)=\sum\limits_{i=1}^k \mathbf{Q}_i(z)\,\mathbf{Q}_i(z)^t.
\]
We write $\mathbf{Q}\in\mathrm{SOS}(\RX^n)$ and denote by
\[
	\mathrm{QM}(\posmat)
:=	\{ q+\trace(\posmat\,\mathbf{Q})\,\vert\, q \in \mathrm{SOS}(\RX),\,\mathbf{Q} \in \mathrm{SOS}(\RX^n) \}
\]
the quadratic module of $\posmat$. Then every element of $\mathrm{QM}(\posmat)$ is nonnegative on $\Image$ and enforcing this constraint gives a lower bound
\begin{equation}\label{OptiProblemPositivity}
		f^* 
=		\begin{array}[t]{rl}
		\max		&	\lambda	\\
		\mbox{s.t.}	&	\lambda\in\R ,\, \forall\, z\in \Image: \, f(z) - \lambda \geq 0
		\end{array}
\geq	f_{\mathrm{sos}}
:=		\begin{array}[t]{rl}
		\sup		&	\lambda	\\
		\mbox{s.t.}	&	\lambda\in\R ,\, f - \lambda \in \mathrm{QM}(\posmat)
		\end{array}
.
\end{equation}

\subsubsection{Moment relaxation}

A linear functional $\functional\in\RX^*$ is said to have a \textbf{representing probability measure on $\Image$}, if there exists a probability measure $\eta$ on $\R^n$ with support in $\Image$, such that, for all $p \in \RX$, $\int_\Image p(z) \, \mathrm{d} \eta(z) = \functional(p)$. Such a functional is nonnegative on $\mathrm{QM}(\posmat)$ and relaxing to this constraint gives another lower bound
\begin{equation}\label{eq_momentrewrite}
		f^*
=		\begin{array}[t]{rl}
		\min		&	\functional(f) \\
		\mbox{s.t.}	&	\functional\in\RX^* \mbox{ has a representing} \\
					&	\mbox{probability measure on } \Image
		\end{array}
\geq	f_{\mathrm{mom}}
:=		\begin{array}[t]{rl}
		\inf		&	\functional(f)	\\
		\mbox{s.t.}	&	\functional\in\RX^*,\,\functional(1)=1,\\
					&	\forall f\in \mathrm{QM}(\posmat):\,\functional (f) \geq 0
		\end{array}
.
\end{equation}

We have $f_{\mathrm{sos}} \leq f_{\mathrm{mom}}$. Indeed, if $\functional$ is feasible for $f_{\mathrm{mom}}$ and $\lambda$ is feasible for $f_{\mathrm{sos}}$, then
\begin{equation}\label{eq_DualityGap}
	\functional(f) - \lambda
=	\functional(\underbrace{f - \lambda}_{\in\mathrm{QM}(\posmat)})
\geq 0.
\end{equation}
We say that $\mathrm{QM}(\posmat)$ is \textbf{Archimedean}, if there exists $p\in \mathrm{QM}(\posmat)$, such that $\{z\in\R^n\,\vert\,p(z)\geq 0\}$ is compact. 

\begin{theorem}\label{thm_holscherer}
	\emph{\cite{holscherer05,holscherer06}} Assume that $\mathrm{QM}(\posmat)$ is Archimedean.
	\begin{enumerate}
		\item Let $p\in\RX$. If $p>0$ on $\Image$, then $p\in \mathrm{QM}(\posmat)$.
		\item Let $\functional\in\RX^*$. If $\functional\geq 0$ on $\mathrm{QM}(\posmat)$, then $\functional$ has a representing probability measure on $\Image$.
		\item Equality holds in \Cref{OptiProblemPositivity} and \Cref{eq_momentrewrite}.
	\end{enumerate}
\end{theorem}

\begin{remark}
In our case, the Archimedean property can be enforced by adding an explicitly known ball constraint. Indeed, for $z\in\Image$, we have $n \geq \norm{z}^2$, and thus $\Image = \{z\in\R^n\,\vert\,\widehat{\posmat}(z)\succeq 0\}$, where $\widehat{\posmat}:=\mathrm{diag}(\posmat,n - \norm{z}^2)\in\RX^{(n+1) \times (n+1)}$. With $\mathbf{Q} = \mathrm{diag}(0,\ldots,0,1)\in\mathrm{SOS}(\RX^{n+1})$, we have $\trace(\widehat{\posmat}\,\mathbf{Q})\in \mathrm{QM}(\widehat{\posmat})$ and the set $\{z\in\R^n\,\vert\,\trace(\widehat{\posmat}(z)\,\mathbf{Q}(z))\}$ is compact.
\end{remark}

\subsection{Lasserre hierarchy with Chebyshev polynomials}

The conditions $f-\lambda \in \mathrm{QM}(\posmat)$ from \Cref{OptiProblemPositivity} and $\functional\geq 0$ on $\mathrm{QM}(\posmat)$ from \Cref{eq_momentrewrite} can be parametrized through positive semi--definite conditions, but for computations we need to restrict to finite dimensional subspaces of $\RX$. We shall now introduce these conditions in the basis of generalized Chebyshev polynomials and then adapt Lasserre's hierarchy \cite{lasserre01} to approximate the optimal value $f^*$ with semi--definite programs \cite{boyd96}. In particular, we present these positive semi--definite conditions in the way they are implemented in our Maple package\footnote{\href{https://github.com/TobiasMetzlaff/GeneralizedChebyshev}{https://github.com/TobiasMetzlaff/GeneralizedChebyshev}}.

\subsubsection{Chebyshev filtration}

For $\functional \in\RX^*$, we define the infinite symmetric matrix $\momm^\functional := \functional(\mathbf{T} \, \mathbf{T}^t)$, where $\functional$ applies entry--wise and $\mathbf{T}$ is the vector of basis elements $T_\weight$ with $\weight\in\Weights^+$.

Then we can also define the $\posmat$--localized matrix $\momm^{\posmat * \functional} := \functional(\posmat\otimes (\mathbf{T} \, \mathbf{T}^t))$. Here, $\functional$ applies entry--wise and $\otimes$ denotes the Kronecker product. The entries of this infinite matrix, indexed by $\weight,\nu\in\Weights^+$, are symmetric blocks of size $n$.

As in \cite{henrion06}, we see that $\functional\geq 0$ on $\mathrm{QM}(\posmat)$ is equivalent to $\momm^{\functional} \succeq 0$ and $\momm^{\posmat * \functional} \succeq 0$. By \Cref{eq_TPolyRecurrence}, for $\weight,\nu\in\Weights^+$, the entries of $\momm^{\functional}$ are
\begin{equation}\label{eq_momentmatrix}
\momm^{\functional}_{\weight\,\nu}
=	\functional (T_\weight\, T_\nu )
=	\dfrac{1}{\nops{\weyl}} \sum\limits_{A\in\weyl} \functional(T_{A\,\weight + \nu}) \in \R.
\end{equation}
Furthermore, let us assume that the matrix $\posmat$ in \Cref{eq_polyoptprob} is represented in the Chebyshev basis as
\[
\posmat(z) = \sum\limits_{\gamma\in\Weights^+} \posmat_\gamma \, T_\gamma(z) \in\RX^{n\times n} 
\]
with $\posmat_\gamma\in\R^{n\times n}$. The entries of $\momm^{\posmat * \functional}$ are
\begin{equation}\label{eq_localizedmomentmatrix}
\momm^{\posmat * \functional}_{\weight\,\nu}
=	\sum\limits_{\gamma\in\Weights^+} \posmat_\gamma \,\functional(T_\weight\, T_\nu\,T_\gamma )
=	\dfrac{1}{\nops{\weyl}^2} \sum\limits_{\gamma \in \Weights^+} \posmat_\gamma \sum\limits_{A,B\in\weyl} \functional(T_{A\weight+B\nu+\gamma})
\in	\R^{n\times n}.
\end{equation}
Restricting $\functional$ to a finite dimensional subspace of $\RX$ in \Cref{eq_momentrewrite} means to truncate the matrices $\momm^{\functional}$ and $\momm^{\posmat * \functional}$ at the corresponding rows and columns. However, since we have chosen the Chebyshev polynomials as a basis, we need to ensure that these matrices are well--defined: For an index of the form $A\,\weight + \nu$ in \Cref{eq_momentmatrix}, there is a unique dominant weight in the same $\weyl$--orbit, say $\tilde{\weight}$, and $\functional$ must be defined on $T_{\tilde{\weight}}$, so that we can compute the matrix entries of $\momm^{\functional}$ (and analogously for $\momm^{\posmat * \functional}$).

\begin{proposition}\label{thm_FiltrationDegree}
Let $\Roots$ be an irreducible root system with highest root $\highestroot$. For $d \in \N$, we define the finite dimensional $\R$--vector subspace
\[
\filt{d}
:=	\langle \{ T_\weight \,\vert\, \weight\in\Weights^+,\,\sprod{\weight,\,\highestroot^\vee} \leq d \}\rangle_\R
\]
of $\RX$. Then $(\filt{d})_{d\in\N}$ is a filtration of $\RX$ as an $\R$--algebra, that is,
\begin{enumerate}
	\item $\RX=\bigcup\limits_{d\in\N}\filt{d}$ and
	\item $\filt{d_1}\,\filt{d_2} \subseteq \filt{d_1 + d_2}$ whenever $d_1,d_2\in\N$.
\end{enumerate}
\end{proposition}
\begin{proof}
\emph{1.} Let $p=\sum_\weight c_\weight\,T_\weight\in\RX$ and choose $d\in\N$ with $d\geq \sprod{\weight,\highestroot^\vee}$ whenever $c_\weight\neq 0$. Then $p\in\filt{d}$.
	
\emph{2.} Let $T_\weight\in\filt{d_1}$ and $T_\nu\in\filt{d_2}$. Then $\nops{\weyl}\,T_\weight\,T_\nu=\sum_A T_{\weight+A\nu}$. For all $A\in\weyl$, there exists $B\in\weyl$, such that $B(\weight+A\nu)\in\Weights^+$. By \cite[Chapitre VI, \S 1, Proposition 18]{bourbaki456}, $\weight - B \weight$ and $\nu - B A \nu$ are sums of positive roots. Hence, there exists $\alpha\in \N^n$, such that
\[
\sprod{B (\weight + A \nu),\highestroot^\vee}
=		\sprod{\weight + \nu,\highestroot^\vee} - \sum\limits_{i=1}^n \alpha_i \sprod{\roots_i,\highestroot^\vee}.
\]
By \cite[Chapitre VI, \S 1.8, Proposition 25]{bourbaki456}, we have $\highestroot^\vee \in \overline{\PC}$ and thus $\sprod{\roots_i,\highestroot^\vee}\geq 0$. We obtain
\[
\sprod{B ( \weight + A \nu),\highestroot^\vee}
\leq	\sprod{\weight + \nu,\highestroot^\vee}
\leq	d_1 + d_2.
\]
Therefore, $T_\weight\,T_\nu\in\filt{d_1+d_2}$.
\end{proof}

\begin{remark}
For irreducible root systems, the filtration induces a weighted degree on $\RX$. Otherwise, we can always construct a filtration by choosing an order on the irreducible components. From now on, we may therefore assume all root systems to be irreducible.
\end{remark}

\subsubsection{Modified Lasserre hierarchy}

When $\functional$ is only defined on $\filt{2d}$, that is, $\functional\in\filt{2d}^*$, then the matrix $\momm^\functional$ is by \Cref{thm_FiltrationDegree} well--defined for all rows and columns up to weighted degree $d$. We denote this truncated matrix of size $\dim(\filt{d})$ by $\momm^\functional_d$. Analogously, for
\[
d
\geq D
:=	\,\min \{ \lceil \ell/2 \rceil \,\vert\,\ell\in\N,\, \posmat \in (\filt{\ell})^{n\times n} \},
\]
the truncated $\posmat$--localized matrix $\momm^{\posmat * \functional}_{d-D}$ is well-defined and of size $n\,\dim(\filt{d-D})$.

On the other hand, if $\mathbf{Q}_1,\ldots,\mathbf{Q}_k\in\filt{d}^n$ are polynomial vectors with entries of weighted degree at most $d$, then the polynomial matrix $\mathbf{Q}=\sum_i \mathbf{Q}_i\,\mathbf{Q}_i^t\in\filt{2d}^{n\times n}$ is a sum of squares. We write $\mathbf{Q}\in\mathrm{SOS}(\filt{d}^n)$ and see that the truncated quadratic module
\[
	\mathrm{QM}(\posmat)_d
:=	\{ q+\trace(\posmat\,\mathbf{Q})\,\vert\, q \in \mathrm{SOS}(\filt{d}) ,\,\mathbf{Q} \in \mathrm{SOS}(\filt{d-D}^n) \}
\]
is contained in $\filt{2d}$. We fix a \textbf{hierarchy order} $d\in\N$, that has to satisfy
\begin{equation}\label{eq_MinimalOrderOfRelaxation}
d \geq \max\{ \min\{\lceil  \ell/2 \rceil \,\vert\,\ell\in\N,\, f\in\filt{\ell} \} ,\, D \} ,
\end{equation}
where $f$ is the objective function from \Cref{eq_polyoptprob}. The \textbf{Chebyshev moment and SOS hierarchy of order $d$} is
\begin{equation}\label{OptiProblemMomRelax}
	f_{\mathrm{mom}}^d
:=	\begin{array}[t]{rl}
	\inf		&	\functional(f) \\
	\mbox{s.t.}	&	\functional\in\filt{2 d}^*,\,\functional(1)=1, \\
				&	\momm^\functional_d,\,\momm^{p*\functional}_{d-D} \succeq 0
\end{array}
\tbox{and}
	f_{\mathrm{sos}}^d 
:=	\begin{array}[t]{rl}
	\sup		&	\lambda \\
	\mbox{s.t.}	&	\lambda\in\R,	\\
				&	f - \lambda \in \mathrm{QM}(\posmat)_d
	\end{array}
.
\end{equation}

\begin{theorem}\label{thm_SosHierarchy}
The following statements hold.
\begin{enumerate}
\item The sequences $(f_{\mathrm{sos}}^d)_{d\in\N}$ and $(f_{\mathrm{mom}}^d)_{d\in\N}$ are monotonously non--decreasing.
\item For $d\in\N$, we have $f_{\mathrm{sos}}^d \leq f_{\mathrm{mom}}^d$.
\item If $\mathrm{QM}(\posmat)$ is Archimedean, then $\lim\limits_{d \to \infty} f_{\mathrm{sos}}^d = \lim\limits_{d \to \infty} f_{\mathrm{mom}}^d = f^*$.
\end{enumerate}
\end{theorem}
\begin{proof}
\emph{1.} follows from the chain of inclusions $\filt{1}\subseteq\filt{2}\subseteq\ldots$

\emph{2.} is analogous to \Cref{eq_DualityGap}.

\emph{3.} By \Cref{thm_holscherer}, for any $\varepsilon > 0$, there exist sums of squares $q$ and $\mathbf{Q}$, such that
\[
	f - f^* + \varepsilon
=	q + \trace(\posmat\,\mathbf{Q}).
\]
Since $\varepsilon$ is arbitrary and $\bigcup\limits_{d\in\N}\filt{d}=\RX$, we obtain $\lim\limits_{d \to \infty} f_{\mathrm{sos}}^d = f^*$. With \emph{2.}, the same holds for $f_{\mathrm{mom}}^d$.
\end{proof}

\subsubsection{SDP formulation}

We translate \Cref{OptiProblemMomRelax} to a semi--definite program (SDP), so that the problem can be implemented and a solution be approximated with solvers such as \textsc{Mosek}\footnote{{O}ptimizer {API} for {P}ython 3 \href{docs.mosek.com/latest/pythonapi/index.html}{docs.mosek.com/latest/pythonapi/index.html}}. For $d\in\N$ and a linear functional $\functional\in\filt{2d}^*$, we write
\begin{equation}\label{MomentMatrixCoeff}
	\begin{pmatrix}
	\momm^{\functional}_{d}	&	0	\\
	0						&	\momm^{\posmat*\functional}_{d-D}
	\end{pmatrix}
=	\sum\limits_{\weight\in \Weights^+} \functional(T_\weight) \, \mathbf{A}_\weight,
\end{equation}
where $\mathbf{A}_\weight$ is the symmetric matrix coefficient of $\functional(T_\weight)$. For $d\geq D$,  $\functional(T_\weight)$ is well--defined whenever $\mathbf{A}_\weight\neq 0$. We write $\mathrm{Sym}^{(d)} := \mathrm{Sym}^{\dim(\filt{d})} \times \mathrm{Sym}^{n\,\dim(\filt{d-D})}$ for the space of symmetric matrices with two blocks. The positive semi--definite elements are denoted by $\mathrm{Sym}^{(d)}_{\succeq 0}$ and we define the dual problems
\begin{equation}\label{OptiMomPrimalDualRelax}
\primalprob \begin{array}[t]{rl}
\inf		&	\sum\limits_{\weight \in S}
				c_\weight \, \mathbf{y}_\weight	\\
\mbox{s.t.}	&	\mathbf{y}\in\R^{\dim(\filt{2 d})} ,\, \mathbf{y}_0=1,	\\
			&	\mathbf{Z} = \sum\limits_{\weight\in\Weights^+} \mathbf{y}_\weight \, \mathbf{A}_\weight\in\mathrm{Sym}^N_{\succeq 0}
\end{array}
\tbox{and} \quad
\dualprob \begin{array}[t]{rl}
\sup		&	c_0 - \trace(\mathbf{A}_0\,\mathbf{X}) \phantom{\sum\limits_{S}} \\
\mbox{s.t.}	&	\mathbf{X}\in\mathrm{Sym}^{(d)}_{\succeq 0} ,\, \forall\,\weight\in S\setminus\{0\}:	\\
			&	\trace(\mathbf{A}_\weight\,\mathbf{X}) = c_\weight
\end{array}
.
\end{equation}

\begin{proposition}\label{prop_SDPDuality}
The optimal value of $\primalprob$ is $f^d_{\mathrm{mom}}$ and the optimal value of $\dualprob$ is $f^d_{\mathrm{sos}}$.
\end{proposition}
\begin{proof}
The statement for $\primalprob$ follows immediately with $\mathbf{y}_\weight=\functional(T_\weight)$ and $\mathbf{Z} = \mathrm{diag} (\momm^{\functional}_{d} , \momm^{\posmat*\functional}_{d-D})$. Let $\functional\in\filt{2d}^*$ and $\lambda \in \R$ be feasible for \Cref{OptiProblemMomRelax}. Then there exist $q\in\mathrm{SOS}(\filt{d})$ and $\mathbf{Q}\in \mathrm{SOS}(\filt{d-D}^n)$ with
\[
	\functional(f) - \lambda
=	\functional(f - \lambda)
=	\functional(q) + \functional(\trace(\posmat\,\mathbf{Q})).
\]
We construct a feasible matrix $\mathbf{X}=\mathrm{diag}(\mathbf{X}_1,\mathbf{X}_2)$ for $\dualprob$ as follows. Since $\mathbf{Q}$ is a sum of squares, we can write $\mathbf{Q} = \mathbf{Q}_1\,\mathbf{Q}_1^t + \ldots + \mathbf{Q}_k\,\mathbf{Q}_k^t$ and denote by $\mathbf{T}_{d-D}$ the vector of generalized Chebyshev polynomials $T_\weight \in \filt{d-D}$. For $1\leq i\leq k$, we have $\mathbf{Q}_i=\mathbf{mat}(\mathbf{Q}_i)\,\mathbf{T}_{d-D}$, where $\mathbf{mat}(\mathbf{Q}_i)$ is the coordinate matrix of the polynomial vector $\mathbf{Q}_i$ in the Chebyshev basis with $n$ rows and $\dim(\filt{d-D})$ columns. Then
\[
\begin{array}{rcl}
	\trace(\posmat\,\mathbf{Q})
&=&	\sum\limits_{i=1}^k \trace(\posmat \, \mathbf{mat}(\mathbf{Q}_i)\,\mathbf{T}_{d-D}\,\mathbf{T}_{d-D}^t \, \mathbf{mat}(\mathbf{Q}_i)^t)\\
&=&	\trace((\posmat \otimes \mathbf{T}_{d-D} \, \mathbf{T}_{d-D}^t) \underbrace{\sum\limits_{i=1}^k \mathbf{vec}(\mathbf{mat}(\mathbf{Q}_i)) \, \mathbf{vec}(\mathbf{mat}(\mathbf{Q}_i))^t}_{=:\mathbf{X}_2} ),
\end{array}
\]
where $\mathbf{vec}(\mathbf{mat}(\mathbf{Q}_i)):=((\mathbf{mat}(\mathbf{Q}_i)_{\cdot 1})^t,\ldots,(\mathbf{mat}(\mathbf{Q}_i)_{\cdot N_{d-D}})^t)^t$ are the stacked columns of $\mathbf{mat}(\mathbf{Q}_i)$. The matrix $\mathbf{X}_2$ is symmetric positive semi--definite of size $n\,\dim(\filt{d-D})$. By definition of the truncated localized moment matrix, we have $\functional(\trace(\posmat\,\mathbf{Q})) = \trace(\momm_{d-D}^{\posmat * \functional}\,\mathbf{X}_2)$. Analogously, there exists a symmetric positive semi--definite $\mathbf{X}_1$ of size $\dim(\filt{d})$ with $\functional(q) = \trace(\momm_{d}^{\functional}\,\mathbf{X}_1)$. When we fix $\mathbf{X}:=\mathrm{diag}(\mathbf{X}_1,\mathbf{X}_2)\in \mathrm{Sym}^{(d)}_{\succeq 0}$ and $\mathbf{A}_\weight$ as in \Cref{MomentMatrixCoeff}, comparing coefficients yields
\[
	\lambda
=	c_0 \, \functional(1) - \functional(q(0)) - \functional(\trace(\posmat(0)\,\mathbf{Q}(0)))
=	c_0 - \trace(\mathbf{A}_0\,\mathbf{X})
\]
and, for $\weight\neq 0$, we have $c_\weight = \trace(\mathbf{A}_\weight\,\mathbf{X})$.

Conversely, we can always construct sums of squares $q$ and $\mathbf{Q}$ from a matrix $\mathbf{X}=\mathrm{diag}(\mathbf{X}_1,\mathbf{X}_2)$ by writing $\mathbf{X}_1$ and $\mathbf{X}_2$ as sums of rank $1$ matrices.
\end{proof}

If $(\mathbf{X},\mathbf{y},\mathbf{Z})$ are optimal for $\primalprob$ and $\dualprob$, then the duality gap of the Chebyshev moment and SOS hierarchy in \Cref{OptiProblemMomRelax} is $f^d_{\mathrm{mom}} - f^d_{\mathrm{sos}} = \trace(\mathbf{X}\,\mathbf{Z}) \geq 0$.

\begin{remark}
The coefficients $c_\weight$ are known from the original problem in \emph{\Cref{eq_polyoptprob}}. The key in setting up \emph{\Cref{OptiMomPrimalDualRelax}} is the computation of the matrices $\mathbf{A}_\weight$. For fixed order $d$, we define
\begin{itemize}
\item the \textbf{matrix size} $N := \dim(\filt{d}) + n\,\dim(\filt{d-D})$ and
\item the \textbf{number of constraints} $m := \dim(\filt{2d}) - 1$.
\end{itemize}
Note that $m$ is the number of matrices $\mathbf{A}_\weight$ with $\weight \neq 0$ and $N$ is their size. Then \emph{\Cref{OptiMomPrimalDualRelax}} is a semi--definite program with primal formulation $\primalprob$ over the cone $\filt{2d}^*\cong\R^{m+1}$ with dual cone $0$ and with dual formulation $\dualprob$ over the self--dual cone $\mathrm{QM}(\posmat)_d \cong \mathrm{Sym}^{N}_{\succeq 0}$.

Computing the matrices $\mathbf{A}_\weight$ of the SDP involves the recurrence formula from \emph{\Cref{eq_TPolyRecurrence}} and is not numerical. If we used the standard monomial basis $\{1,z_1,z_2,\ldots,z_1^2,z_1 z_2,\ldots\}$, this computation would be trivial, but the matrices would be larger when truncating at the usual degree instead of the weighted degree. Hence, our technique is more efficient, if the numerical effort to solve a larger SDP in the standard monomial basis is bigger than the combined effort to numerically solve a smaller SDP in the Chebyshev basis plus matrix computation. Another upside is that, since the matrices $\mathbf{A}_\weight$ only depend on the root system $\Roots$ and the order $d$, but not on the objective function $f$, the same matrices can be used to solve several problems as a preprocess.

\begin{table}[H]
	\begin{center}
		\scalebox{0.65}{
			\begin{tabular}{|c||c|c|c|c|c|c|c|c|c|}
				\hline
				$\Roots\backslash d$	&	$2$				&	$3$				&	$4$				&	$5$					&	$6$					&
				$7$				&	$8$				&	$9$				&	$10$				\\
				\hline
				\hline
				$\RootB[2],\RootC[2]$	&	$6+2,14$		&	$10+6,27$		&	$15		+12,44$	&	$21		+20,65$		&	$28	+30,90$			&
				$36+42,119$		&	$45+56,152$		&	$55	+72,189$	&	$66	+90,230$		\\
				\hline
				$\RootG[2]$				&	$-		$		&	$6+3,15$		&	$9	+6,24$		&	$12		+12,35$		&	$16	+18,48$			&
				$20+27,63$		&	$25+36,80$		&	$30		+48,99$	&	$36	+60,120$		\\
				\hline
				$\RootA[2]$				&	$-		$		&	$10+3,27$		&	$15		+9,44$	&	$21		+18,65$		&	$28	+30,90$			&
				$36+45,119$		&	$45+63,152$		&	$55	+84,189$	&	$66	+108,230$		\\
				\hline
				$\RootB[3]$				&	$-		$		&	$13+3,49$		&	$22	+9,94$		&	$34	+21,160$		&	$50	+39,251$		&
				$70+66,371$		&	$95+102,524$	&	$125+150,714$	&	$161	+210,945$	\\
				\hline
				$\RootC[3]$				&	$-		$		&	$20+3,83$		&	$35	+12,164$	&	$56	+30,285$		&	$84	+60,454$		&
				$120+105,679$	&	$165+168,968$	&	$220+252,1329$	&	$286	+360,1770$	\\
				\hline
				$\RootA[3]$				&	$-		$		&	$-		$		&	$35		+4,164$	&	$56	+16,285$		&	$84	+40,454$		&
				$120+80,679$	&	$165+140,968$	&	$220+224,1329$	&	$286	+336,1770$	\\
				\hline
				$\RootB[4]$				&	$-		$		&	$-		$		&	$30		+4,174$	&	$50	+12,335$		&	$80	+32,587$		&
				$120+64,959$	&	$175+120,1484$	&	$245+200,2199$	&	$336	+320,3145$	\\
				\hline
				$\RootC[4]$				&	$-		$		&	$-		$		&	$70+4,494$		&	$126	+20,1000$	&	$210	+60,1819$	&
				$330+140,3059$	&	$495+280,4844$	&	$715+504,7314$	&	$1001	+840,10625$	\\
				\hline
				$\RootD[4]$				&	$-		$		&	$-		$		&	$46+4,294$		&	$80	+16,580$		&	$130	+44,1035$	&
				$200+96,1715$	&	$295+184,2684$	&	$420+320,4014$	&	$581	+520,5785$	\\
				\hline
			\end{tabular}
		}
	\end{center}
	\caption{The SDP parameters $(N,m)$ for \Cref{OptiMomPrimalDualRelax} depend on the root system $\Roots$ and the order $d$.}
	\label{SDPMatrixNumberSizeTable}
\end{table}
\end{remark}

\subsection{Optimizing on coefficients}

For a finite set $S\subseteq\Weights^+\setminus\{0\}$ of dominant weights, we shall be confronted in \Cref{section_chromatic} with a bilevel optimization problem, where we not only have to minimize the objective function $f$ from \Cref{eq_ObjectivePoly} with respect to $z\in\Image$, but also maximize with respect to the coefficients $c_\weight$ under some compact affine constraints. The problem can be represented as
\[
	F(S)
:=	\begin{array}[t]{rl}
	\max\limits_{c} \min\limits_{z}	&	\sum\limits_{\weight \in S} c_\weight\,T_\weight(z) \\
	\mbox{s.t.}						&	z\in\Image ,\, c\in\R^{S} ,\, b^t\,c = 1 ,\\
									&	\ell_\weight \leq c_\weight \leq u_\weight \tbox{for} \weight	\in		S
\end{array}
,
\]
where $b\in\R^S$ and $\ell_\weight \leq u_\weight\in\R$. For scalar polynomial constraints defining $\Image$, a hierarchy of SDPs to approximate $F(S)$ was introduced in \cite[Chapter 13]{lasserre09}. With our polynomial matrix constraint, the theory is similar. For $d\in\N$ large enough, that is, $T_\weight\in\filt{2d}$ whenever $\weight\in S$, we define
\[
F(S,d)
:=	\begin{array}[t]{rl}
	\sup		&	- \trace(\mathbf{A}_0\,\mathbf{X}) \\
	\mbox{s.t.}	&	\mathbf{X} \in\mathrm{Sym}^{(d)}_{\succeq 0},\, \sum\limits_{\weight\in S} \alpha_\weight\,\trace(\mathbf{A}_\weight\,\mathbf{X}) = 1,	\\
	&	\ell_\weight \leq \trace(\mathbf{A}_\weight\,\mathbf{X}) \leq u_\weight	\tbox{for} \weight	\in		S,	\\
	&	\trace(\mathbf{A}_\nu\,\mathbf{X}) = 0													\tbox{for} \nu		\notin	S\cup\{0\}
\end{array}
,
\]
where the $\mathbf{A}_0,\mathbf{A}_\weight, \mathbf{A}_\nu \in \mathrm{Sym}^{(d)}$ are the $\dim(\filt{2d})$ many matrices defined via \Cref{MomentMatrixCoeff}.

\begin{theorem}\label{thm_MaxMinConvergence}
The sequence $(F(S,d))_{d\in\N}$ is monotonously non--decreasing. If $\mathrm{QM}(\posmat)$ is Archimedean, then $\lim\limits_{d\to\infty} F(S,d) = F(S)$.
\end{theorem}
\begin{proof}
The proof is analogous to the one of \cite[Theorem 13.1]{lasserre09}, but uses the Positivstellensatz of Hol and Scherer instead of Putinar's. Let $\mathbf{X}$ be optimal for $F(S,d)$ and set $c_\weight := \trace (\mathbf{A}_\weight\,\mathbf{X})$ for $\weight \in S$. Then $F(S,d) \leq (f_c)^* \leq F(S)$, where $(f_c)^*$ denotes the minimum of $f_c := \sum\limits_{\weight\in S} c_\weight\,T_\weight \in\RX$ on $\Image$.

On the other hand, $\Image=\{ z\in \R^n \,\vert\, \posmat(z) \succeq 0 \}$ is compact and the $T_\weight$ are continuous. Hence, the map $g:c \mapsto (f_c)^*$ is continuous on a compact set and there exists a feasible $c^* \in \R^S$, such that $F(S)=g(c^*)$. For any $\varepsilon > 0$, the polynomial $\sum_{\weight\in S} c^*_\weight \, T_\weight - F(S) +\varepsilon$ is strictly positive on $\Image$. Thus, by \Cref{thm_holscherer}, there exist sums of squares $q\in\mathrm{SOS}(\RX)$ and $\mathbf{Q}\in\mathrm{SOS}(\RX^n)$, such that
\[
\sum_{\weight\in S} c^*_\weight \, T_\weight - (F(S) - \varepsilon)
=	q + \trace(\posmat \, \mathbf{Q}).
\]
For $d\in\N$ sufficiently large, we can follow our proof of \Cref{prop_SDPDuality} to construct a matrix $\mathbf{X} \in \mathrm{Sym}^{(d)}_{\succeq 0}$ with $- \trace(\mathbf{A}_0\,\mathbf{X})=c^*_0$, $- \trace(\mathbf{A}_0\,\mathbf{X})=F(S)-\varepsilon$, $\trace(\mathbf{A}_\weight\,\mathbf{X})=c^*_\weight$ for $\weight\in S$ and $\trace(\mathbf{A}_\nu\,\mathbf{X}) = 0$ for $0\neq \nu\notin S$. Then $\mathbf{X}$ is feasible for $F(S,d)$, and therefore $F(S,d) \geq F(S) - \varepsilon$. Since $\varepsilon > 0$ is arbitrary, the statement follows.
\end{proof}

\subsection{A case study}
\label{sec_casestudy}

We apply the Chebyshev moment and SOS hierarchy to solve a trigonometric polynomial optimization problem with crystallographic symmetry and compare with another technique. One alternative approach is to reinforce positivity constraints on trigonometric polynomials to \textbf{sums of Hermititan squares} (SOHS), which goes back to the generalized Riesz--F\'{e}jer theorem \cite[Theorem 4.11]{dumitrescu07}. With \cite[Equation (3.71)]{dumitrescu07}, one can then approximate the minimum of a trigonometric polynomial $f\in\R[\Weights]$ by solving a semi--definite program
\begin{equation}\label{eq_DumitrescuRelaxation}
f_{\mathrm{rf}}^S
=
\begin{array}[t]{rl}
\sup		&	\lambda	\\
\mbox{s.t.}	&	f-\lambda\in\mathrm{SOHS}(S)
\end{array}
\end{equation}
as in Riesz--Fej\'{e}r, where $S\subseteq \Weights$ is a finite set of exponents containing the support of $f$ up to central symmetry. This can be translated into SDP standard form with Kronecker products of elementary Toeplitz matrices, yielding a hierarchy of lower bounds.

\begin{example}\label{example_DumitrescuComparison}
We search the global minima $f^*$, $g^*$, $h^*$ and $k^*$ of the following $\weyl$--invariant trigonometric polynomials with graphs depicted in \emph{\Cref{FigureDumitrescu}}.
\begin{enumerate}
\item Let $\Roots=\RootG[2]$, $\weyl=\mathfrak{S}_3 \ltimes \{\pm 1 \}$, $\Weights = \Z\,\fweight{1}\oplus\Z\,\fweight{2} = \Z [0,-1,1]^t \oplus \Z [-1,-1,2]^t$ and
\begin{align*}
	f(u)
:=&	\,\gencos{2\,\fweight{1}}(u) + 2\,\gencos{\fweight{2}}(u)\\
=&	\,(
	\mcos{\sprod{2\,\fweight{1},u}} + 
	\mcos{\sprod{2\,\fweight{1} - 2\,\fweight{2},u}} + 
	\mcos{\sprod{4\,\fweight{1} - 2\,\fweight{2},u}} \\
&	\, + 
	2\,\mcos{\sprod{\fweight{2},u}} + 
	2\,\mcos{\sprod{3\,\fweight{1} - \fweight{2},u}} + 
	2\,\mcos{\sprod{3\,\fweight{1} - 2\,\fweight{2},u}} 
	)/3.
\end{align*}
In the coordinates $z=\gencos{}(u)$, we have $f(z) = 6\,z_1^2 - 2\,z_1 - 1$ $($see \emph{\Cref{example_A2PolyRewrite}}$)$.

\item Let $\Roots=\RootG[2]$, $\weyl=\mathfrak{S}_3 \ltimes \{\pm 1 \}$, $\Weights = \Z\,\fweight{1}\oplus\Z\,\fweight{2} = \Z [0,-1,1]^t \oplus \Z [-1,-1,2]^t$ and
\begin{align*}
	g(u)
:=&	\,2\,\gencos{\fweight{1}}(u) + \gencos{\fweight{2}}(u) + \gencos{\fweight{1} + \fweight{2}}(u) + 4\,\gencos{3\,\fweight{1}}(u)\\
=&	\,(
	2\,\mcos{\sprod{\fweight{1},u}} + 
	2\,\mcos{\sprod{\fweight{1}-\fweight{2},u}} + 
	2\,\mcos{\sprod{2\,\fweight{1}-\fweight{2},u}}\\
	&\,+ 
	\mcos{\sprod{\fweight{2},u}} + 
	\mcos{\sprod{3\,\fweight{1}-2\,\fweight{2},u}} + 
	\mcos{\sprod{3\,\fweight{1}-\fweight{2},u}}\\
	&\,+ 
	4\,\mcos{\sprod{3\,\fweight{1},u}} + 
	4\,\mcos{\sprod{3\,\fweight{1}-3\,\fweight{2},u}} + 
	4\,\mcos{\sprod{6\,\fweight{1}-3\,\fweight{2},u}}
	)/3\\
	&\,+ (
	\mcos{\sprod{\fweight{1},u} + \sprod{\fweight{2},u}} + 
	\mcos{\sprod{\fweight{1}-2\,\fweight{2},u}} + 
	\mcos{\sprod{4\,\fweight{1}-\fweight{2},u}} \\
	&\,+\mcos{\sprod{4\,\fweight{1}-3\,\fweight{2},u}} + 
	\mcos{\sprod{5\,\fweight{1}-2\,\fweight{2},u}} + 
	\mcos{\sprod{5\,\fweight{1}-3\,\fweight{2},u}}
	)/6.
\end{align*}
In the coordinates $z=\gencos{}(u)$, we have $g(z) = 144\,z_1^3 - 6\,z_1^2 - 69\,z_1\,z_2 - 33\,z_1 - 21\,z_2 - 7$.

\item Let $\Roots=\RootC[2]$, $\weyl=\mathfrak{S}_2 \ltimes \{\pm 1 \}^2$, $\Weights = \Z\,\fweight{1}\oplus\Z\,\fweight{2} = \Z [1,0]^t \oplus \Z [1,1]^t$ and
\begin{align*}
	h(u)
:=&	\,2\,\gencos{\fweight{1}}(u) + \gencos{\fweight{2}}(u) - \gencos{2\,\fweight{2}}(u) - 3\,\gencos{\fweight{1}+\fweight{2}}(u)\\
=&	\,\mcos{\sprod{\fweight{1},u}} + 
	\mcos{\sprod{\fweight{1}-\fweight{2},u}} \\
&	\,+(
	\mcos{\sprod{\fweight{2},u}} + 
	\mcos{\sprod{2\,\fweight{1}-\fweight{2},u}} - 
	\mcos{\sprod{2\,\fweight{2},u}} - 
	\mcos{\sprod{4\,\fweight{1}-2\,\fweight{2},u}}
	)/2 \\
&	\,-3/4\,(
	\mcos{\sprod{\fweight{1}-2\,\fweight{2},u}} + 
	\mcos{\sprod{\fweight{1}+\fweight{2},u}} \\
&	\,+ \mcos{\sprod{3\,\fweight{1}-2\,\fweight{2},u}} + 
	\mcos{\sprod{3\,\fweight{1}-\fweight{2},u}}
	).
\end{align*}
In the coordinates $z=\gencos{}(u)$, we have $h(z) = 8\,z_1^2 - 6\,z_1\,z_2 - 4\,z_2^2 + 5\,z_1 - 3\,z_2 - 1$.

\item Let $\Roots=\RootC[2]$, $\weyl=\mathfrak{S}_2 \ltimes \{\pm 1 \}^2$, $\Weights = \Z\,\fweight{1}\oplus\Z\,\fweight{2} = \Z [1,0]^t \oplus \Z [1,1]^t$ and
\begin{align*}
	k(u)
:=&	\,2\,\gencos{2\fweight{1}}(u) + \gencos{2\,\fweight{2}}(u)\\
=&	\,\mcos{\sprod{2\,\fweight{1},u}} + \mcos{\sprod{2\,\fweight{1}-2\,\fweight{2},u}} + \mcos{\sprod{2\,\fweight{2},u}}/2 + \mcos{\sprod{4\,\fweight{1}-2\,\fweight{2},u}}/2 \\
&	
\end{align*}
In the coordinates $z=\gencos{}(u)$, we have $k(z) = 4\,z_2^2 - 1$.
\end{enumerate}

For $3 \leq d \leq 7$, we choose $\tilde{S}$ to be the set of all dominant weights $\weight \in \Weights^+$ with $\Cheblvl{T_{\weight}} \leq d$. In \emph{\Cref{eq_DumitrescuRelaxation}}, $S=(\tilde{S}-\tilde{S})\cap (H\setminus \{0\})$ is an admissible choice for any halfspace $H$, since $S$ contains all exponents of the objective functions up to central symmetry. In this case, we denote the optimal value by $f^d_{\mathrm{rf}}$. On the other hand, we apply the Chebyshev SOS reinforcement $f^d_{\mathrm{sos}}$ from \emph{\Cref{OptiProblemMomRelax}}, where we only need to take exponents up to Weyl group symmetry, that is, $\tilde{S}$ itself.

With the two techniques, we obtain the results in \emph{\Cref{TableDumitrescuComparison}}. $N$ denotes the matrix size and $m$ the number of constraints, depending on $d$. In practice, it is usually not possible to determine the exact minimal value. However, since we compare lower bounds, it suffices to check which bound is larger and therefore closer to the actual minimum. To solve the semi--definite programs, we rely on \emph{\textsc{Mosek}}.

\begin{figure}[H]
\begin{center}
\begin{subfigure}{.45\textwidth}
	\includegraphics[height=5.1cm]{ToeplitzCirclesG2.png}
	\caption{The graph of $f$ for $u=(u_1,u_2,-u_1-u_2)$.}
\end{subfigure}
\quad
\begin{subfigure}{.45\textwidth}
	\includegraphics[height=5.1cm]{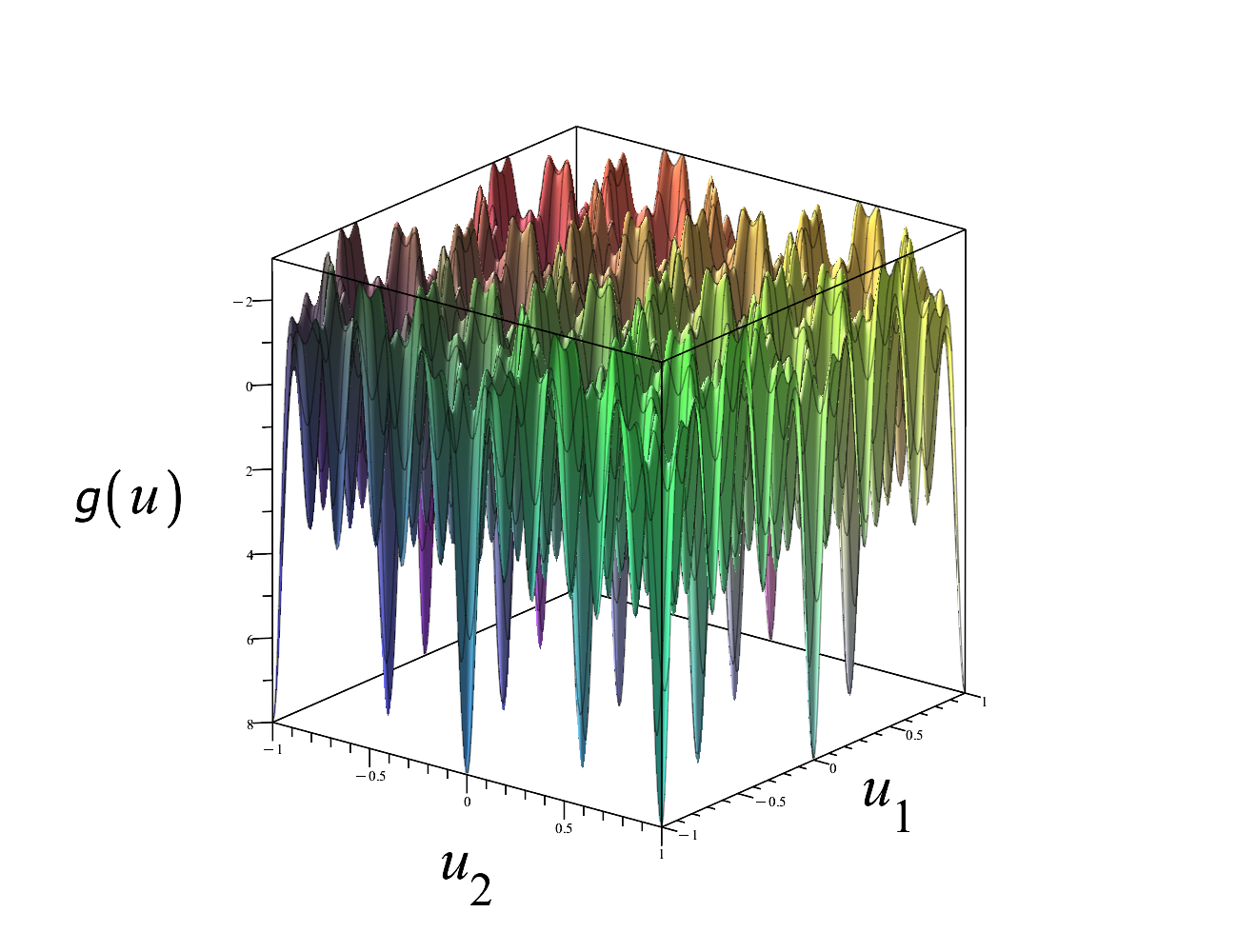}
	\caption{The graph of $g$ for $u=(u_1,u_2,-u_1-u_2)$.}
\end{subfigure}
\quad
\begin{subfigure}{.45\textwidth}
	\includegraphics[height=5.1cm]{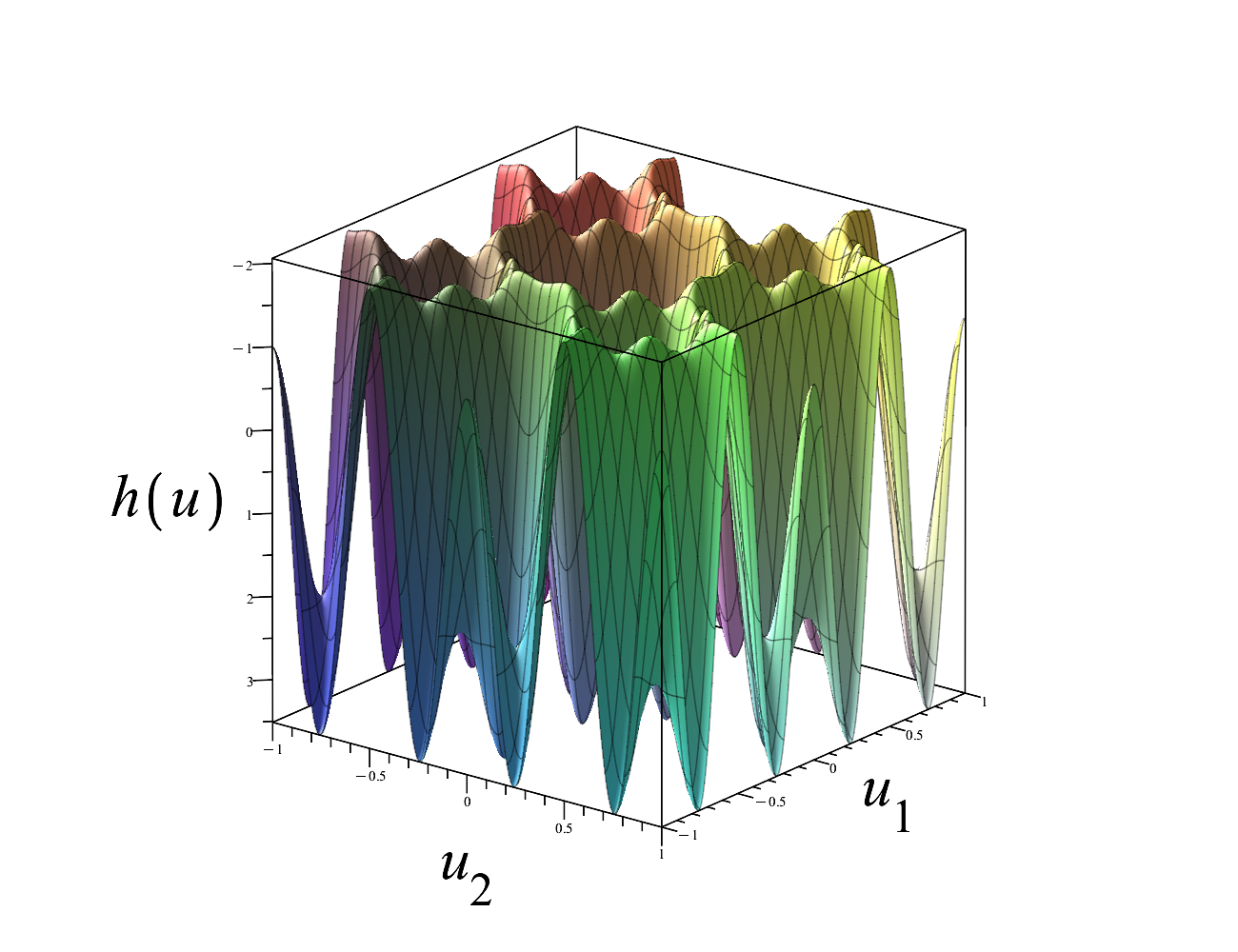}
	\caption{The graph of $h$ for $u=(u_1,u_2)$.}
\end{subfigure}
\quad
\begin{subfigure}{.45\textwidth}
	\includegraphics[height=5.1cm]{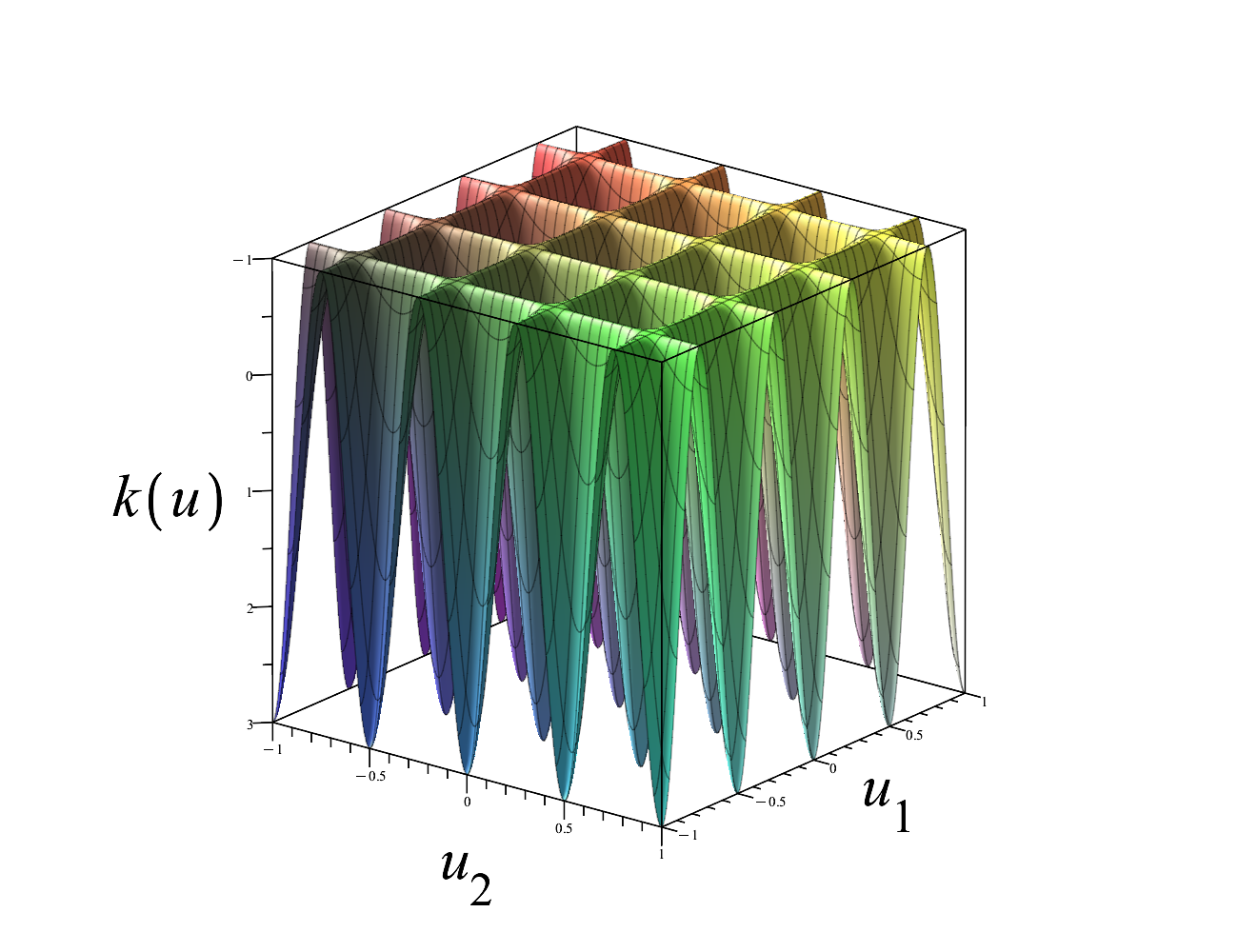}
	\caption{The graph of $k$ for $u=(u_1,u_2)$.}
\end{subfigure}
\caption{The graphs of the objective functions for $u \in \R^3 / [1,1,1]^t \cong \R^2$.}
\label{FigureDumitrescu}
\end{center}
\end{figure}

\begin{table}[H]
\begin{center}
\scalebox{0.8}{
\begin{tabular}{|c||c|c|c|c|c|}
\hline
$d$						&	$3$			&	$4$				&	$5$
						&	$6$			&	$7$					\\
\hline
\hline
$f^d_{\mathrm{rf}}$		&	$-1.18824$	&	$-1.180240$		&	$-1.17058$
&	$-1.16970$	&	$-1.16719$			\\
$N,m$				&	$49,33$		&	$81,58$			&	$121,90$
&	$169,129$	&	$225,175$		\\
\hline
$f^d_{\mathrm{sos}}$	&	$-1.16667$	&	$-1.16667$		&	$-1.16667$
&	$-1.16667$	&	$-1.16667$			\\
$N,m$					&	$9,15$		&	$15,24$			&	$24,35$
&	$34,48$		&	$47,63$				\\
\hline
\hline
$g^d_{\mathrm{rf}}$		&	$-3.50118$	&	$-3.40372$		&	$-3.31195$
						&	$-3.25383$	&	$-3.22049$			\\
$N,m$				&	$49,33$		&	$81,58$			&	$121,90$
						&	$169,129$	&	$225,175$		\\
\hline
$g^d_{\mathrm{sos}}$	&	$-3.20499$	&	$-3.10220$		&	$-2.98718$
						&	$-2.98718$	&	$-2.98718$			\\
$N,m$					&	$9,15$		&	$15,24$			&	$24,35$
						&	$34,48$		&	$47,63$				\\
\hline
\hline
$h^d_{\mathrm{rf}}$		&	$-2.12159$	&	$-2.10672$		&	$-2.1012$
						&	$-2.09959$	&	$-2.09073$			\\
$N,m$					&	$25,24$		&	$49,54$			&	$81,96$
						&	$121,150$	&	$169,217$		\\
\hline
$h^d_{\mathrm{sos}}$	&	$-2.27496$	&	$-2.06250$		&	$-2.06250$
						&	$-2.06250$	&	$-2.06250$			\\
$N,m$					&	$16,27$		&	$27,44$			&	$41,65$
						&	$58,90$		&	$78,119$			\\
\hline
\hline
$k^d_{\mathrm{rf}}$		&	$-1.00000$	&	$-1.00000$		&	$-1.00000$
						&	$-1.00000$	&	$-1.00000$			\\
$N,m$					&	$25,84$		&	$41,144$		&	$61,220$
						&	$85,312$	&	$113,420$		\\
\hline
$k^d_{\mathrm{sos}}$	&	$-1.00000$	&	$-1.00000$		&	$-1.00000$
						&	$-1.00000$	&	$-1.00000$			\\
$N,m$					&	$16,27$		&	$27,44$			&	$41,65$
						&	$58,90$		&	$78,119$			\\
\hline
\end{tabular}
}
\end{center}
\caption{Comparison of the two techniques in terms of approximation and SDP parameters $(N,m)$.}
\label{TableDumitrescuComparison}
\end{table}
\end{example}

\begin{remark}
In \emph{\Cref{TableDumitrescuComparison}}, we observe $f^*\geq f^d_{\mathrm{sos}} \geq f^d_{\mathrm{rf}}$ for $d\geq 4$. Hence, our approximation of $f^*$ appears to be better in those cases, while the parameters $N,m$ that indicate the size of the SDP are smaller $($analogous for $g,h,k)$. Differences in the quality of the approximation might depend on the stability of the SDP \emph{\cite{parrilo22}}.
\end{remark}

\section{Spectral bounds for set avoiding graphs}
\label{section_chromatic}

\setcounter{equation}{0}

In this last section, we apply our method for trigonometric optimization problems with crystallographic symmetry to the computation of spectral bounds for chromatic numbers. The chromatic number of a graph gives the minimal number of colors needed to paint the vertices, so that no edge connects two vertices of the same color. When dealing with set avoiding graphs, \cite{BdCOV} provides a lower bound, which involves minimizing the Fourier transformation of a measure.

While this bound has been used and strengthened for the graph $\R^n$ avoiding Euclidean distance $1$ \cite{Soifer09,deGrey18,BdCOV,BPT15}, it has not been widely used as a tool for polytopes. Crystallographic symmetry in the trigonometric optimization problem arises, when the polytope has Weyl group symmetry. Then we can rewrite the spectral bound in terms of generalized Chebyshev polynomials and use the results of \Cref{section_trigonometric,section_optimization}.

An advantage of our approach is that rewriting the optimization problem in terms of polynomials allows in several cases to compute bounds with simple proofs and to recover many results. In other cases, we compute numerical bounds with the modified Lasserre hierarchy from \Cref{section_optimization}. Our approach allows to study the quality of the spectral bound and to estimate the optimal involved measure, see \Cref{B3L1NormCoefficients}.

\subsection{Computing spectral bounds with Chebyshev polynomials}

Let $V\leq\R^n$ be an Abelian group and $S\subseteq V$ be bounded, centrally--symmetric with $0\notin\overline{S}$. We consider the \textbf{set avoiding graph} $G(V,S)$, where $V$ is the set of vertices and two vertices $u,v\in V$ are connected by an edge if and only if $u - v\in S$. In this context, we call $S$ the \textbf{avoided set}.

A set of vertices $I\subseteq V$ is called \textbf{independent} for $G(V,S)$, if no pair of vertices in $I$ are connected by an edge, that is, for all $u,v\in I$, we have $u-v\notin S$. A \textbf{measurable coloring} $X$ of $G(V,S)$ is a partition of $V$ in independent Lebesgue--measurable sets. The \textbf{measurable chromatic number} of $G(V,S)$ is
\[
	\chi_m(V,S)
:=	\inf\{ \nops{X} \,\vert\, X \mbox{ is a measurable coloring of } V \}.
\]

\subsubsection{The spectral bound}

In \cite{BdCOV}, Bachoc, Decorte, de Oliveira Filho and Vallentin generalized the Hoffman \cite{Hoffman1970} and Lovasz \cite{Lovasz1979} bounds for finite graphs to the case $V=\R^n$, using the framework of bounded self--adjoint operators. Showing that the result holds for any set avoiding graph $G(V,S)$ is a straightforward adaptation of \cite[\S 5.1]{DMMV}.

\begin{theorem}\label{thm_spectral_bound}
\emph{\cite[\S 3.1]{BdCOV}}
Let $\measure$ be a finite Borel measure supported on $S$ with Fourier transformation
\[
\widehat{\measure}(u)
=	\int_{S} \mexp{u,v} \, \mathrm{d}\measure (v).
\]
Then the measurable chromatic number of $G(V,S)$ satisfies
\[
\chi_m(V,S)
\geq	1-\frac{\sup\limits_{u\in\R^n}\widehat{\measure}(u)}{\inf\limits_{u\in\R^n}\widehat{\measure}(u)}.
\]
\end{theorem}

The problem of computing the measurable chromatic number of $G(V,S)$ gained fame after Hardwiger and Nelson in 1950 studied the case $V=\R^2$ and $S=\mathbb{S}^1$, the Euclidean unit sphere, which remains unsolved. Current bounds and the history of the problem can be found in \cite{Soifer09} and \cite{deGrey18}.

For $V=\R^n$ and $S=\mathbb{S}^{n-1}$ the Euclidean unit sphere, the bounds obtained from \Cref{thm_spectral_bound} have been computed for $\chi_m(\R^n,\mathbb{S}^{n-1})$, see for example \cite{BPT15}. In this case, the optimal measure is the surface measure on $\mathbb{S}^{n-1}$. Beyond the spectral bound, the computation of $\chi_m(\R^n,\mathbb{S}^{n-1})$ has been studied in \cite{BPS21,ambrus22,GM22}.

\subsubsection{Reformulation in terms of Chebyshev polynomials}

For a root system $\Roots$ in $\R^n$ with Weyl group $\weyl$ and weight lattice $\Weights$, we consider those $S\subseteq V$ with Weyl group symmetry, that is $\weyl\, S = S$. The $\weyl$--invariant trigonometric polynomials $\R[\Weights]^\weyl$ with support in $S$ are the Fourier transformations of atomic $\weyl$--invariant Borel measures supported on $\Weights\cap S$. We treat the optimization problem in \Cref{thm_spectral_bound} for this class of measures with the theory developped in \Cref{section_optimization}. In fact, by an averaging argument on all orbits, we see that an optimal measure for \Cref{thm_spectral_bound} is obtained from such a $\weyl$--invariant trigonometric polynomial. We denote by 
\[
\Image
=	\{ \gencos{}(u) \,\vert\, u\in\R^n \}
=	\{ z\in\R^n \,\vert\, \posmat(z) \succeq 0 \}
\]
the image of the generalized cosines and define
\begin{equation}
F(S)
:=	\begin{array}[t]{rl}
	\max\limits_{c} \min\limits_{z}	&	\sum\limits_{\weight \in S\cap\Weights^+} c_\weight\,T_\weight(z) \\
	\mbox{s.t.}						&	z\in\Image ,\, c\in\R^{S\cap\Weights^+}_{\geq 0} ,\, \sum\limits_{\weight\in S\cap\Weights^+} c_\weight = 1 
	\end{array}
.
\end{equation}

\begin{theorem}\label{thm_ChromaticChebyshevBound}
Let $\weyl\,S=S$ and $S\cap\Weights\neq\emptyset$. The measurable chromatic number of $G(V,S)$ satisfies
\[
		\chi_m(V,S)
\geq	1-\frac{1}{F(S)}.
\]
\end{theorem}
\begin{proof}
Since $S$ is bounded, the nonempty set $S\cap\Weights$ is finite. We consider the atomic Borel measure
\[
	\measure
=	\sum\limits_{\weight\in S\cap\Weights} \frac{c_\weight}{\nops{\weyl\weight}} \, \delta_\weight
\]
with $\delta_\weight$ Dirac and $0\leq c_\weight = c_{-\weight}\in\R$, such that, for all $A\in\weyl$, $c_{A\weight}=c_\weight$. Then the Fourier transformation is
\begin{align*}
	\widehat{\measure}(u)
	=	\int_{S} \mexp{u,v} \, \mathrm{d}\measure (v)
	=	\sum\limits_{\weight \in S\cap\Weights}
	\frac{c_\weight}{\nops{\weyl\weight}} \, \mexp{\weight,u} 
	=	\sum\limits_{\weight \in S\cap\Weights^+}
	c_\weight \, \gencos{\weight}(u)
	=	\sum\limits_{\weight \in S\cap\Weights^+}
	c_\weight \, T_\weight(\gencos{}(u)).
\end{align*}
In particular, we have
\begin{align*}
		\widehat{\measure}(u)
\leq	\sum\limits_{\weight \in S\cap\Weights} \frac{c_\weight}{\nops{\weyl\weight}}
=		\sum\limits_{\weight \in S\cap\Weights^+} c_\weight
\end{align*}
and equality holds for $u=0$. Optimizing over the coefficients $c$ under the condition $\sum_{\weight} c_\weight = 1$ and using \Cref{coro_OptiProbelmRewrite} with \Cref{thm_spectral_bound} gives the lower bound $1-1/F(S)$ for $\chi_m(V,S)$.
\end{proof}

In practice, the problem of computing $F(S)$ analytically is not always possible. Instead we can use the theory of \Cref{section_optimization} to lower bound it numerically. For $d \in \N$ sufficiently large, we consider the SDP
\begin{equation}
F(S,d)
:=	\begin{array}[t]{rl}
\sup		&	- \trace(\mathbf{A}_0\,\mathbf{X}) \\
\mbox{s.t.}	&	\mathbf{X} \in\mathrm{Sym}^{(d)}_{\succeq 0},\, \sum\limits_{\weight\in S\cap\Weights^+} \trace(\mathbf{A}_\weight\,\mathbf{X}) = 1,	\\
&	\trace(\mathbf{A}_\weight\,\mathbf{X})	\geq	0 \tbox{for} \weight	\in	S\cap\Weights^+,	\\
&	\trace(\mathbf{A}_\nu\,\mathbf{X})		=		0 \tbox{for} \nu		\in	\Weights^+\setminus (S\cup\{0\})
\end{array}
,
\end{equation}
where the semi--definite cone $\mathrm{Sym}^{(d)}_{\succeq 0}$ and the finitely many matrices $\mathbf{A}_0,\mathbf{A}_\weight,\mathbf{A}_\nu \in \mathrm{Sym}^{(d)}_{\succeq 0}$ are defined as in \Cref{MomentMatrixCoeff}.

\begin{corollary}\label{coro_ChromBoundLasserre}
\emph{[of \Cref{thm_MaxMinConvergence,thm_ChromaticChebyshevBound}]} Let $\weyl\,S=S$ and $S\cap\Weights\neq\emptyset$. The sequence $(F(S,d))_{d\in\N}$ is monotonously non--decreasing and we have
\[
\chi_m(V,S)
\geq	1-\frac{1}{F(S,d)}.
\]
Furthermore, if $\mathrm{QM}(\posmat)$ is Archimedean, then $\lim\limits_{d\to \infty} F(S,d) = F(S)$.
\end{corollary}

\subsection{The chromatic number of a coroot lattice}

For an $n$--dimensional lattice $V=\Corootlattice$ in $\R^n$, we call $\lambda \in \Corootlattice \setminus\{0\}$ a \textbf{strict Vorono\"{i} vector}, if the intersection $(\lambda + \Vor(\Corootlattice)) \cap \Vor(\Corootlattice)$ is a facet of $\Vor(\Corootlattice)$, that is, a face of dimension $n-1$ of the Vorono\"{i} cell. In this case, a natural choice for the avoided set $S$ is the set of all strict Vorono\"{i} vectors of $\Corootlattice$. The chromatic number $\chi(\Corootlattice)$ of the lattice $\Corootlattice$ is defined as the chromatic number of the graph $G(\Corootlattice) := G(\Corootlattice, S)$. 

\begin{figure}[H]
	\begin{center}
		\includegraphics[height=4cm]{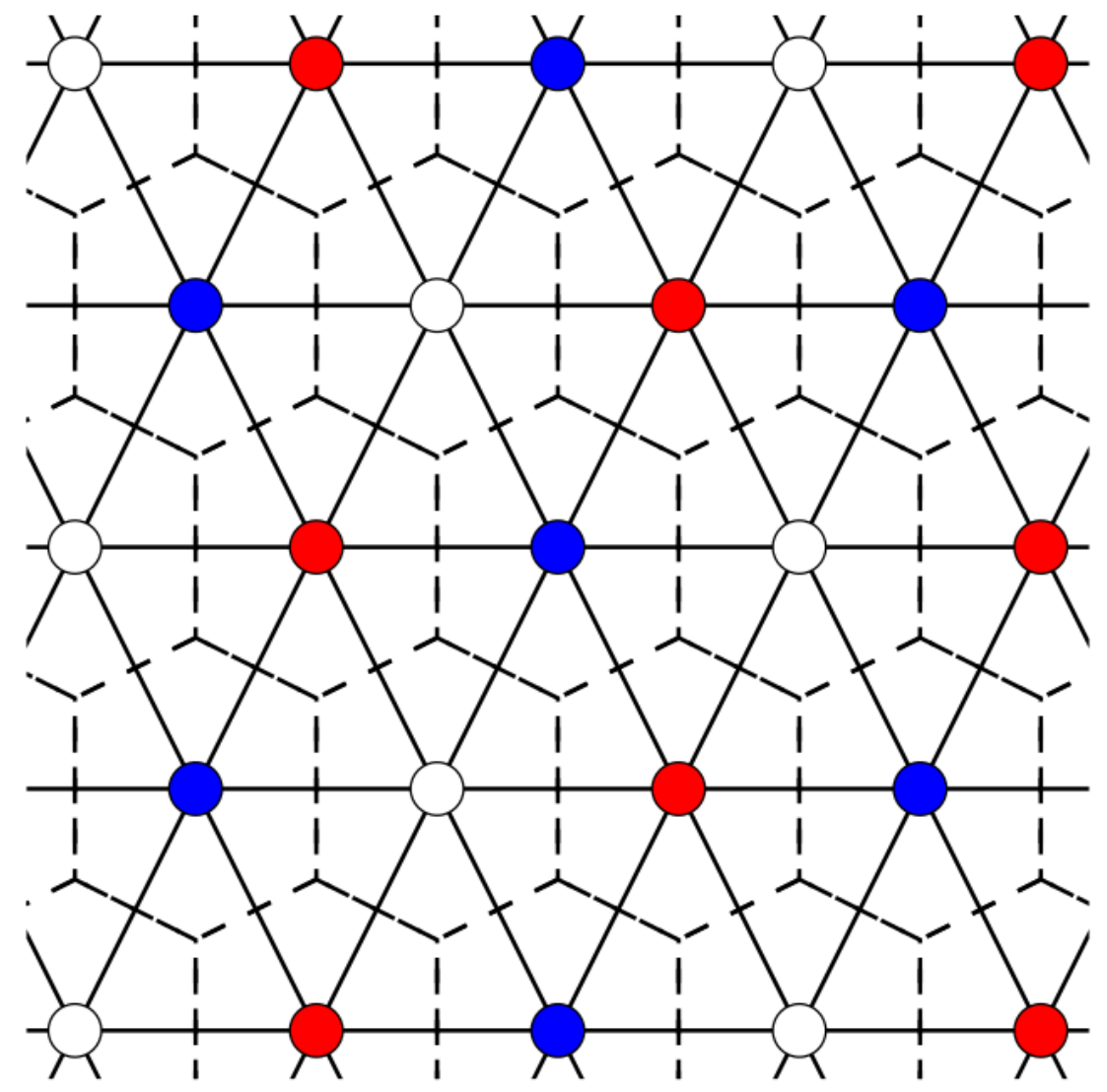}
	\end{center}
	\caption{The chromatic number of the $\RootA[2]$ coroot lattice is $\chi(\Corootlattice(\RootA[2])) = 3 $.}\label{fig_ChiLambda}
\end{figure}

The chromatic number of several instances of these graphs was computed in \cite{DMMV}, some of them through the spectral bound from \Cref{thm_spectral_bound}. In this subsection, we reprove the bounds for the case, where $\Corootlattice$ is the coroot lattice of an irreducible root system.

\begin{proposition}\label{prop_StrictVoronoiHighestRoot}
Assume that $\Corootlattice$ is the coroot lattice of an irreducible root system $\Roots$ with highest root $\highestroot$. Then the set of strict Vorono\"{i} vectors of $\Corootlattice$ is the orbit $S = \weyl \highestroot^\vee$.
\end{proposition}
\begin{proof}
By \cite[Chapitre VI, \S 1, Proposition 11 et 12]{bourbaki456}, there are at most two distinct root lengths and two roots have the same length if and only if they are in the same $\weyl$--orbit. If $\roots \in\Roots$, then $\sprod{ \highestroot , \highestroot } \geq \sprod{ \roots , \roots }$ and so
\[
\sprod{ \highestroot^\vee , \highestroot^\vee }
=		\frac{4}{\sprod{ \highestroot , \highestroot }}
\leq	\frac{4}{\sprod{ \roots , \roots }}
=		\sprod{ \roots^\vee , \roots^\vee }.
\]
Thus, $\highestroot^\vee$ is a short root of the coroot system $\Roots^\vee$. The lattice generated by $\Roots^\vee$ is $\Corootlattice$ and, by the discussion before \cite[Chapter 21, Theorem 8]{conway1988a}, the short roots $\weyl(\Roots^\vee)\highestroot^\vee$ are the strict Vorono\"{i} vectors. As $\weyl(\Roots)=\weyl(\Roots^\vee)$, the statement follows.
\end{proof}

If $\highestroot^\vee\in\Weights$, we obtain
\begin{equation}\label{eq_LatticeChebyshevBound}
		\chi(\Corootlattice)
\geq	1-\frac{1}{\min\limits_{z\in\Image} T_{\highestroot^\vee}(z)}.
\end{equation}
Indeed, since the strict Vorono\"{i} vectors form a single $\weyl$--orbit, there is no freedom for the coefficients in \Cref{thm_ChromaticChebyshevBound} and we are left with minimizing with respect to $z\in\Image$.

If $\highestroot^\vee\notin\Weights$, we can replace $T_{\highestroot^\vee}$ by $T_\weight$ with $\weight=\ell\highestroot^\vee\in\Weights$ for some $\ell>0$, because $\R^n$ is invariant under scaling. For example, this is the case for $\RootG[2]$, where $\highestroot^\vee = \highestroot/3 = \fweight{2}/3$ (and this is the only exception for the irreducible root systems). However, since the coroot lattice of $\RootG[2]$ is the hexagonal one from \Cref{fig_ChiLambda}, this case is covered by $\RootA[2]$.

We now reprove the bounds from \cite{DMMV}.

\begin{theorem}\label{thm_LatticeChromatic}
The following statements hold.
\begin{enumerate}
\item The spectral bound is sharp for $\chi(\Corootlattice(\RootC)) = 2$.
\item The spectral bound is sharp for $\chi(\Corootlattice(\RootA)) = n$.
\item We have $\chi(\Corootlattice(\RootB)) = \chi(\Corootlattice(\RootD)) \geq n$.
\end{enumerate}
\end{theorem}
\begin{proof}
\emph{1.} We have $\Corootlattice(\RootC)=\Z^n$. When we partition $\Z^n$ in elements with even and odd $\ell_1$--norm, then this gives an admissible coloring with $\chi(\Corootlattice(\RootC)) \leq 2$. To see that the spectral bound is sharp, note that $\highestroot^\vee = \highestroot/2 = \fweight{1}$ and consider the Chebyshev polynomial $T_{\highestroot^\vee}=T_{\fweight{1}}=z_1$. With \Cref{eq_LatticeChebyshevBound}, we obtain
\[
		\chi(\Corootlattice(\RootC))
\geq	1-\frac{1}{\min\limits_{z\in\Image} T_{\highestroot^\vee}(z)}
=		1-\frac{1}{\min\limits_{z\in\Image} z_1}
\geq	1-\frac{1}{-1}
=		2,
\]
because  $\Image\subseteq[-1,1]^n$.
	
\emph{2.} We have $\chi(\Corootlattice(\RootA)) = n$ \cite{DMMV} and $\highestroot^\vee = \highestroot = \fweight{1}+\fweight{n-1}$ with $-\fweight{1}\in\weyl\fweight{n-1}$. In \Cref{eq_LatticeChebyshevBound}, we consider
\[
	T_{\fweight{1}+\fweight{n-1}}
=	\nops{\weyl\,\fweight{1}}\,T_{\fweight{1}}\,T_{\fweight{n-1}} - \sum\limits_{\substack{\weight\in \weyl\,\fweight{1}\\ \weight\neq \fweight{1}}} T_{\weight+\fweight{n-1}}
=	n\,z_1\,z_{n-1} - (T_0+(n-2)\,T_{\fweight{1}+\fweight{n-1}}).
\]
The last equation follows from the fact that, if $\weight=-\fweight{n-1}$, then $\weight+\fweight{n-1}=0$, and, if $\weight \neq - \fweight{n-1}$, then $\weight+\fweight{n-1}\in \weyl(\fweight{1}+\fweight{n-1})$, see \Cref{equation_WeightsRootsA}. Since $-\fweight{1}\in\weyl\fweight{n-1}$, we also have $z_1\,z_{n-1} = z_1\,\overline{z_1} =\nops{z_1}^2$ for $z\in\Image$ (in the case of $\RootA$, $\Image$ is complex and can be embedded in $\R^{n-1}$ with \Cref{eq_realgencos}). Altogether, we obtain
\[
		\chi(\Corootlattice(\RootA))
\geq	1-\frac{1}{\min\limits_{z\in\Image} T_{\highestroot^\vee}(z)}
=		1-\frac{n-1}{\min\limits_{z\in\Image} n\,z_1\,z_{n-1}-1}
=		1-\frac{n-1}{\min\limits_{z\in\Image} n\,\nops{z_1}^2-1}
\geq	1-\frac{n-1}{-1}
=		n.
\]

\emph{3.} For $\Roots = \RootB[2]$, we are in the situation of \emph{1.} with $\chi(\Corootlattice(\RootB[2])) = 2$ (the cubic lattice). For $\Roots= \RootB[3]$, we are in the situation of \emph{2.} with $\chi(\Corootlattice(\RootB[3])) = 3$ (see \Cref{RhombicDodecahedron}). The root system $\RootD$ is not defined for $n\leq 3$. Thus, let $n\geq 4$ and $\Roots\in\{\RootB,\RootD\}$. For $1 \leq i \leq n-1$, we have $\roots_i^\vee(\RootB) = \roots_{i}^\vee(\RootD)$ and $\roots_{n}^\vee(\RootB) = \roots_{n}^\vee(\RootD) - \roots_{n-1}^\vee(\RootD)$ as well as $\roots_{n}^\vee(\RootD) = \roots_{n}^\vee(\RootB) + \roots_{n-1}^\vee(\RootB)$. Hence, we have $\Corootlattice(\RootB) = \Corootlattice(\RootD)$ with $\highestroot^\vee = \highestroot = \fweight{2}$. We consider $T_{\highestroot} = T_{\fweight{2}}(z) = z_2$ and minimize on $\Image$. By \Cref{thm_HermiteCharacterization}, we have $\Image = \{z\in \R^n \,\vert \, \posmat(z)\succeq 0 \}$ and the first entry of $\posmat$ is $4\,\posmat_{11}=T_0-T_{2\fweight{1}}$ with
\[
	T_{2\fweight{1}}
=	\nops{\weyl\,\fweight{1}}\,T_{\fweight{1}}^2 - \sum\limits_{\substack{\weight\in \weyl\,\fweight{1}\\ \weight\neq \fweight{1}}} T_{\weight+\fweight{1}}
=	2\,n\,z_1^2 - (1+2\,(n-1)\,z_2).
\]
The last equation follows from the fact that, if $\weight=-\fweight{1}$, then $\weight+\fweight{1}=0$, and, if $\weight \neq - \fweight{1}$, then $\weight+\fweight{1}\in \weyl(\fweight{2})$, see \Cref{equation_WeightsRootsB,equation_WeightsRootsD}. Thus, for $z\in\Image$, we have
\[
		0
\leq	4\,\posmat_{11}(z)
=		T_0(z) - T_{2\fweight{1}}(z)
=		1 - (2\,n\,z_1^2 - 1 - 2\,(n-1)\,z_2)
\Leftrightarrow
		z_2
\geq	\frac{n\,z_1^2-1}{n-1}
\geq	\frac{-1}{n-1}
\]
and obtain
\[
		\chi(\Corootlattice(\Roots))
\geq	1-\frac{1}{\min\limits_{z\in\Image} T_{\highestroot^\vee}(z)}
=		1-\frac{1}{\min\limits_{z\in\Image} T_{\fweight{2}}(z)}
=		1-\frac{1}{\min\limits_{z\in\Image} z_2}
\geq	1-\frac{n-1}{-1}
=		n.
\]
\end{proof}

\begin{remark}
Since, up to rescaling, two adjacent vertices in $G(\Corootlattice)$ are also adjacent in the graph $G(\Corootlattice, \Corootlattice \cap \partial\Vor(\Lambda))$, the value of $\chi(\Lambda)$ also gives a lower bound on $\chi_m(\R^n, \partial\Vor(\Lambda))$, even if the two numbers can be far from each other. For instance, we have $\chi(\Lambda(\RootA[n])) = n+1$, but $\chi(\R^n, \partial\Vor(\Lambda(\RootA[n]))) = 2^n$ \emph{\cite{BBMP}}.
\end{remark} 

\subsection{The chromatic number of $\Z^n$ for the crosspolytope}
\label{sssec_CP}

We consider the integer lattice $V=\Z^n$ together with the avoided set
\[
	\mathbb{B}^1_r
:=	\{ u\in\Z^n \, \vert \, \norm{ u }_1 = \nops{u_1}+\ldots+\nops{u_n}=r\}
\]
for $r\in\N$. Two vertices in the graph $G(\Z^n,\mathbb{B}^1_r)$ are adjacent, if the absolute values of the differences between their coordinates sums up to $r$. The convex hull of $\mathbb{B}^1_r$ is the ball of radius $r$ with respect to the $\ell_1$--norm, also known as the crosspolytope, see \Cref{pic_C3B3L1Norm}. Several bounds for the chromatic number $\chi(\Z^n,\mathbb{B}^1_r)$ were given in \cite{furedi04} without using spectral bounds, but through combinatorial arguments. If $\mathbb{B}^1_r \subseteq \Weights$ is contained in the weight lattice of some root system in $\R^n$, then we can compare by computing
\begin{equation}\label{eq_ZnChebyshevBound}
\chi(\Z^n , \mathbb{B}^1_r)
\geq	1-\frac{1}{F(r)},
\end{equation}
where $F(r):=F(\mathbb{B}^1_r)$ is defined as in \Cref{thm_ChromaticChebyshevBound}.

\begin{lemma}\label{prop_L1WeightsZ}
Let $0 < r\in\N$. If $\Roots$ is a root system of type $\RootB$, $\RootC$ or $\RootD$, then $\mathbb{B}^1_r \subseteq \Weights$ and the dominant weights are $\mathbb{B}^1_r\cap\Weights^+=$
\[
\begin{cases}
	\{	\alpha_1\,\fweight{1} + \ldots + \alpha_n\,\fweight{n} \, \vert \,
	\alpha\in\N^n,\,\sum\limits_{i=1}^n i\,\alpha_i = r \}
	,&	\tbox{if} \Roots = \RootC	\\
	\{	\alpha_1\,\fweight{1} + \ldots + \alpha_{n-1}\,\fweight{n-1} + 2\,\alpha_n\,\fweight{n} \, \vert \,
	\alpha\in\N^n,\,\sum\limits_{i=1}^n i\,\alpha_i = r \}
	,&	\tbox{if} \Roots = \RootB	\\
	\{	 \alpha_1\,\fweight{1} + \ldots + \alpha_{n-2}\,\fweight{n-2} + 2(\alpha_{n-1}\,\fweight{n-1} + \alpha_n\,\fweight{n}) \, \vert \,
	\alpha\in\N^n,\,\sum\limits_{i=1}^{n} i \alpha_i + \alpha_{n-1} = r \}
	,&	\tbox{if} \Roots = \RootD
\end{cases}	.
\]
\end{lemma}
\begin{proof}
This follows from \Cref{equation_WeightsRootsC,equation_WeightsRootsB,equation_WeightsRootsD}.
\end{proof}

\begin{figure}[H]
\begin{center}
	\begin{subfigure}{.3\textwidth}
		\centering
		\includegraphics[width=1.5cm]{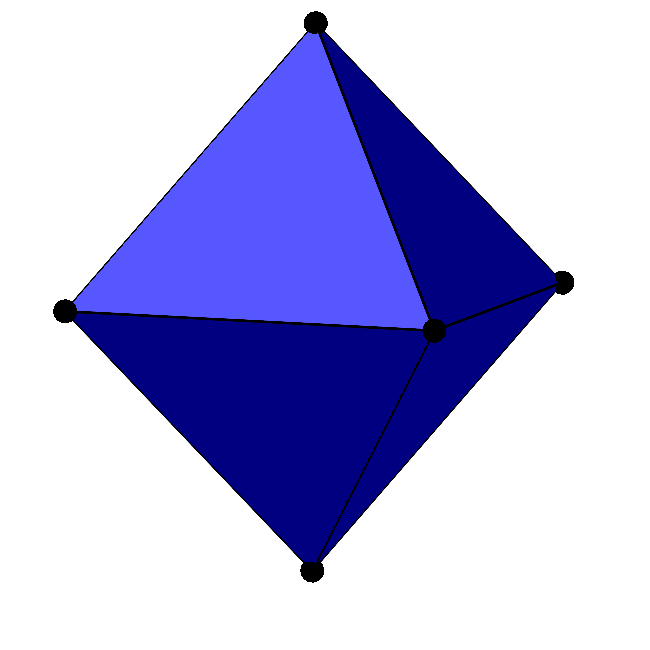}
		\caption{$r=1$}
		\label{L1Lvl1}
	\end{subfigure}
	\begin{subfigure}{.3\textwidth}
		\centering
		\includegraphics[width=3cm]{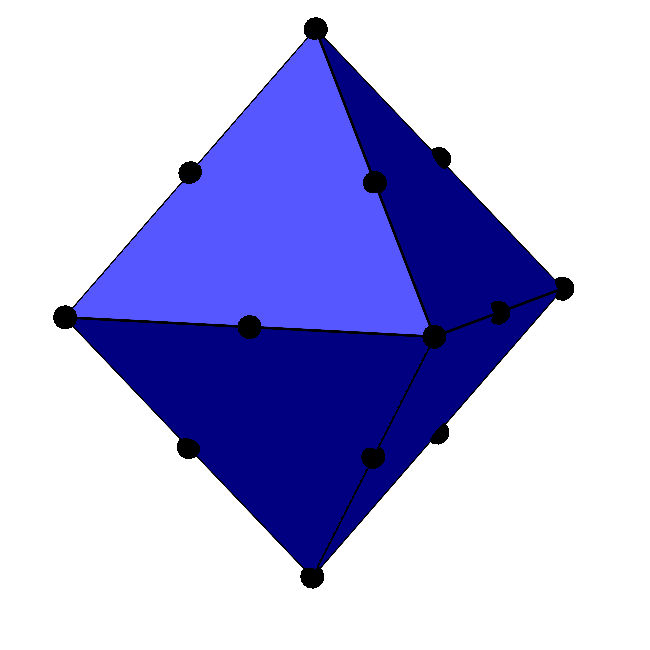}
		\caption{$r=2$}
		\label{L1Lvl2}
	\end{subfigure}
	\quad\quad\quad
	\begin{subfigure}{.3\textwidth}
		\centering
		\includegraphics[width=4.5cm]{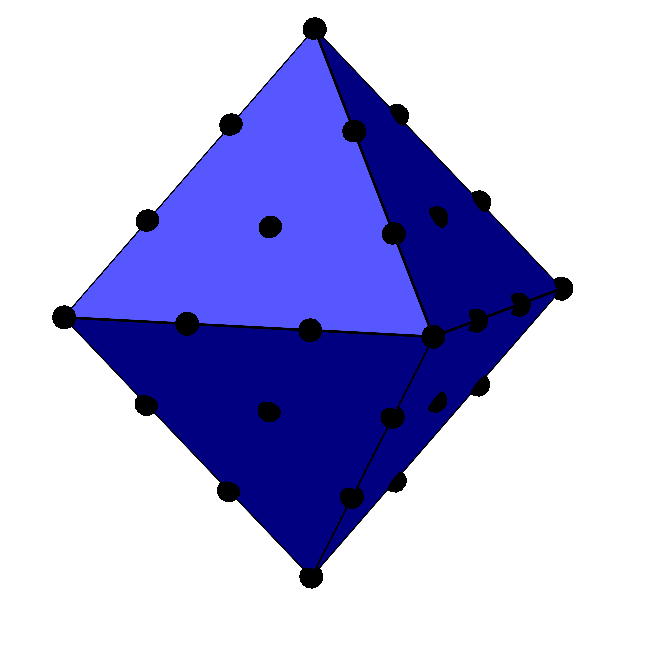}
		\caption{$r=3$}
		\label{L1Lvl3}
	\end{subfigure}
	\caption{The crosspolytope of radius $r$ with respect to the $\ell_1$--norm and the points $\mathbb{B}^1_r$ with integer coordinates on the boundary.}
	\label{pic_C3B3L1Norm}
\end{center}
\end{figure}

\begin{remark}\label{remark_ChromZnRn}
Denote by $\mathcal{P}$ the crosspolytope from \emph{\Cref{pic_C3B3L1Norm}} for $r=1$, that is, $\mathcal{P} = \mathrm{ConvHull}(\mathbb{B}^1_1)$. Then $G(\Z^n,\mathbb{B}^1_r)$ is a discrete subgraph of $G(\R^n,\partial(r\mathcal{P}))$ and, since $\R^n$ is scaling invariant, we have
\[
		\chi_m(\R^n,\partial\mathcal{P})
=		\chi_m(\R^n,\partial(r\mathcal{P}))
\geq	\chi(\Z^n,\mathbb{B}^1_r).
\]
Hence, computing the spectral bound for the chromatic number of $\Z^n$ always yields a lower bound for the chromatic number of $\R^n$.
\end{remark}

\subsubsection{Analytical bounds}

We compute the spectral bound for $\chi(\Z^n,\mathbb{B}^1_r)$ first for the cases, where our rewriting technique allows for an analytical proof.

\begin{proposition}\label{coro_SpectralBoundZodd}
Let $r\in \N$ be odd. The spectral bound is sharp for $\chi (\Z^n,\,\mathbb{B}^1_r) = 2$.
\end{proposition}
\begin{proof}
Since $r$ is odd, partitioning the vertices of $G(\Z^n,\mathbb{B}^1_r)$ in those with even and those with odd $\ell_1$--norm yields two independent sets. Hence, $\chi (\Z^n,\,\mathbb{B}^1_r) = \chi (\Z^n,\,\mathbb{B}^1_1) = 2$. To see that the spectral bound is sharp, let $\Roots$ be a root system of type $\RootC$. By \Cref{prop_L1WeightsZ}, we have $\mathbb{B}^1_1 = \weyl \fweight{1}$ and so
\[
	\chi (\Z^n,\,\mathbb{B}^1_1) \geq 1 - \frac{1}{\min\limits_{z\in\Image} z_1} \geq 1-\frac{1}{-1} = 2 .
\]
\end{proof}

The chromatic number of $\Z^n$ for $\ell_1$--distance $r=2$ is $2\,n$. This was proven in \cite[Theorem 1]{furedi04} with a purely combinatorial argument by fixing a coloring and showing that it is admissible and minimal.

\begin{theorem}\label{thm_ZnL1r2bound}
The spectral bound is sharp for $\chi (\Z^n,\,\mathbb{B}^1_2) = 2\,n$.
\end{theorem}
\begin{proof}
Let $\Roots$ be a root system of type $\RootC$. Thanks to \Cref{prop_L1WeightsZ}, we have $\mathbb{B}^1_2 = \weyl (2\,\fweight{1}) \cup \weyl\fweight{2}$. We choose $c=1/(2\,n-1)$ and consider
\[
	c\,T_{2\fweight{1}} + (1-c)\,T_{\fweight{2}}
=	\frac{2\,n\,z_1^2 - 2(n-1)z_2 - 1}{2\,n-1} + \frac{2(n-1)z_2}{2\,n-1}
=	\frac{2\,n\,z_1^2 - 1}{2\,n-1} ,
\]
where the expression for $T_{2\fweight{1}}$ is obtained as in the proof of \Cref{thm_LatticeChromatic} (\emph{3.}). By \Cref{eq_ZnChebyshevBound}, we have
\[
		\chi (\Z^n,\,\mathbb{B}^1_2)
\geq	1-\frac{1}{\min\limits_{z\in\Image} c\,T_{2\fweight{1}}(z) + (1-c)\,T_{\fweight{2}}(z)}
\geq	1-\frac{1}{(2\,n\,z_1^2 - 1)/(2\,n-1)}
\geq	1-\frac{2\,n-1}{ - 1}
=		2\,n.
\]
\end{proof}

\begin{corollary}\label{coro_Z2L1rEven}
Let $0 < r \in \N$ be even. The spectral bound is sharp for $\chi (\Z^2,\,\mathbb{B}^1_r) = 4$.
\end{corollary}
\begin{proof}
For $r=2$, this is a special case of \Cref{thm_ZnL1r2bound}. Since $2$ divides $r$ whenever $r$ is even, the spectral bound gives at least $4$ for $\chi (\Z^2,\,\mathbb{B}^1_r)$. Let $\mathcal{P}=\mathrm{ConvHull}(\mathbb{B}^1_1)$ be the crosspolytope in $\R^2$, that is, a square. By \cite{BBMP} and \Cref{remark_ChromZnRn}, we have
\[
4 = \chi_m(\R^2,\,\partial\mathcal{P}) = \chi_m(\R^2,\,\partial(r\mathcal{P})) \geq \chi(\Z^2,\,\mathbb{B}^1_r) \geq \chi(\Z^2,\,\mathbb{B}^1_2) \geq 4.
\]
\end{proof}

\subsubsection{Numerical bounds}

We will now give spectral bounds for $\chi(\Z^n,\mathbb{B}^1_r)$ numerically for the dimensions $n=3$ and $n=4$. In order to do so, we apply \Cref{coro_ChromBoundLasserre} and compute $F(r,d):=F(\mathbb{B}^1_r,d)$ for $d\in \N$ sufficiently large.

\subsubsection*{Dimension $n = 3$}

\begin{table}[H]
\begin{center}
	\begin{tabular}{|c|c||c|c|c|c|c|c|c|}
		\hline
		$\Roots$	&	$d \backslash r$	&	$2$			&	$4$			&	$6$			&	$8$			&
		$10$		&	$12$		&	$14$		\\
		\hline
		\hline
		$\RootB[3]$	&	$3$					&	$6.00000$	&	$6.28148$	&	$6.01551$	&	$-		$	&
		$-		$	&	$-		$	&	$-		$	\\
		\hline
		&	$4$					&	$6.00000$	&	$6.28148$	&	$6.07717$	&	$6.28148$	&
		$-		$	&	$-		$	&	$-		$	\\
		\hline
		&	$5$					&	$6.00000$	&	$6.28148$	&	$6.29004$	&	$6.28183$	&
		$6.12543$	&	$-		$	&	$-		$	\\
		\hline
		&	$6$					&	$6.00000$	&	$6.28148$	&	$6.30244$	&	$6.29799$	&
		$6.27850$	&	$6.28234$	&	$-		$	\\
		\hline
		&	$7$					&	$6.00000$	&	$6.28148$	&	$6.30269$	&	$6.30435$	&
		$6.30031$	&	$6.29708$	&	$6.27830$	\\
		\hline
		&	$8$					&	$6.00000$	&	$6.28148$	&	$6.30269$	&	$6.30463$	&
		$6.30053$	&	$6.30088$	&	$6.29604$	\\
		\hline
		&	$9$					&	$6.00000$	&	$6.28148$	&	$6.30269$	&	$6.30501$	&
		$6.30502$	&	$6.30227$	&	$6.301858$	\\
		\hline
		\hline
		$\RootC[3]$	&	$3$					&	$6.00000$	&	$6.28148$	&	$6.02310$	&	$-		$	&
		$-		$	&	$-		$	&	$-		$	\\
		\hline
		&	$4$					&	$6.00000$	&	$6.28148$	&	$6.29021$	&	$6.28198$	&
		$-		$	&	$-		$	&	$-		$	\\
		\hline
		&	$5$					&	$6.00000$	&	$6.28148$	&	$6.30182$	&	$6.29951$	&
		$6.29810$	&	$-		$	&	$-		$	\\
		\hline
		&	$6$					&	$6.00000$	&	$6.28148$	&	$6.30269$	&	$6.30455$	&
		$6.30048$	&	$6.30069$	&	$-		$	\\
		\hline
		&	$7$					&	$6.00000$	&	$6.28148$	&	$6.30269$	&	$6.30494$	&
		$6.30057$	&	$6.30229$	&	$6.30156$	\\
		\hline
	\end{tabular}
\end{center}
\caption{The bound $\chi (\Z^3,\,\mathbb{B}^1_r) \geq 1-1/F(r,d)$.}
\label{B3C3L1Table2}
\end{table}

The value $\chi (\Z^3,\,\mathbb{B}^1_2)=6$ is obtained immediately with $F(2,1)$. The highest value in the table is given by $F(9,10)$ for $\RootB[3]$. Furthermore, $F(4,d)$ seems to be stable in the case of both root systems. We give the obtained optimal coefficients, which coincide for $\RootB[3]$ and $\RootC[3]$ in \Cref{B3L1NormCoefficients,C3B3L1NormTable} and \Cref{table_appendix}.

\begin{figure}[H]
	\begin{center}
		\begin{subfigure}{.4\textwidth}
			\centering
			\includegraphics[width=6cm]{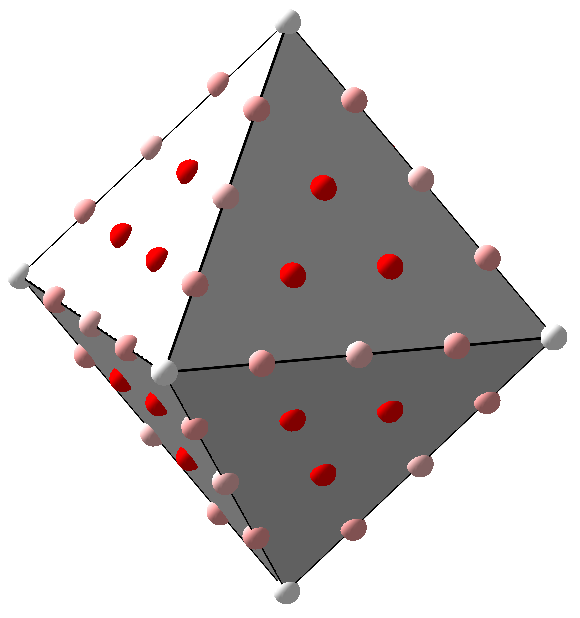}
			\caption{$r=4$}
		\end{subfigure}
		\quad
		\begin{subfigure}{.4\textwidth}
			\centering
			\includegraphics[width=6cm]{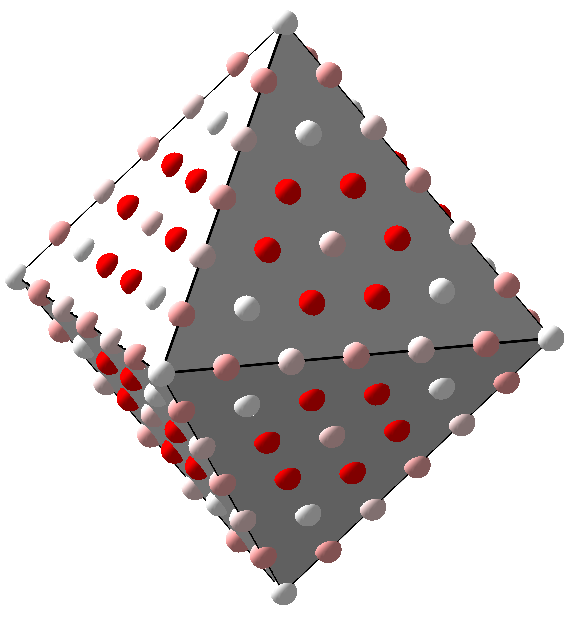}
			\caption{$r=6$}
		\end{subfigure}
		\quad
		\begin{subfigure}{.4\textwidth}
			\centering
			\includegraphics[width=6cm]{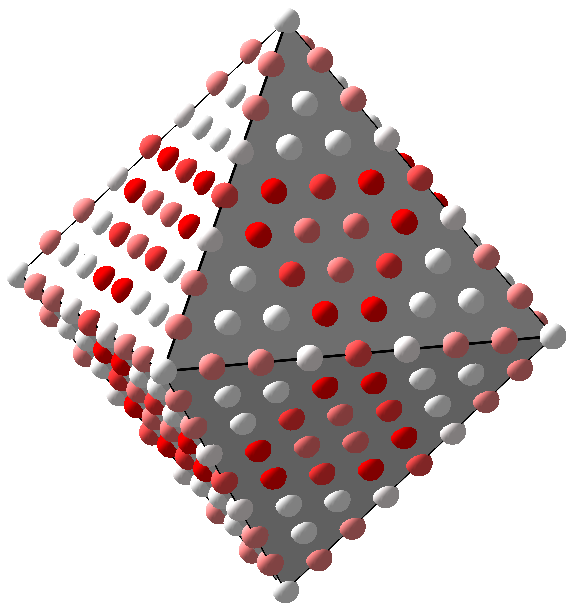}
			\caption{$r=8$}
		\end{subfigure}
		\quad
		\begin{subfigure}{.4\textwidth}
			\centering
			\includegraphics[width=6cm]{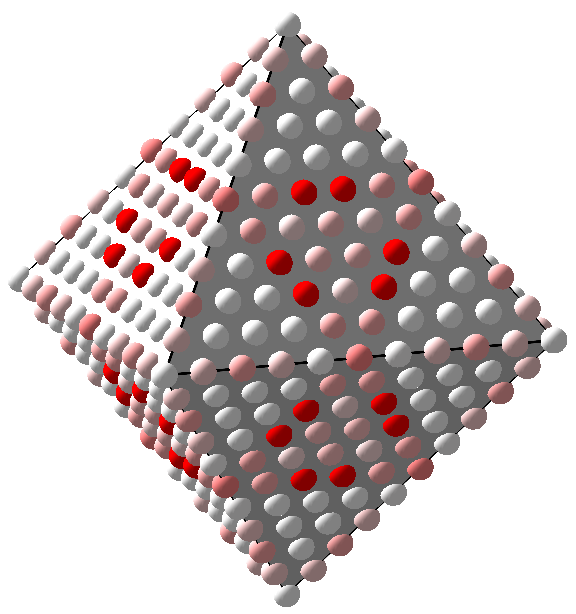}
			\caption{$r=10$}
		\end{subfigure}
		\quad
		\begin{subfigure}{.4\textwidth}
			\centering
			\includegraphics[width=6cm]{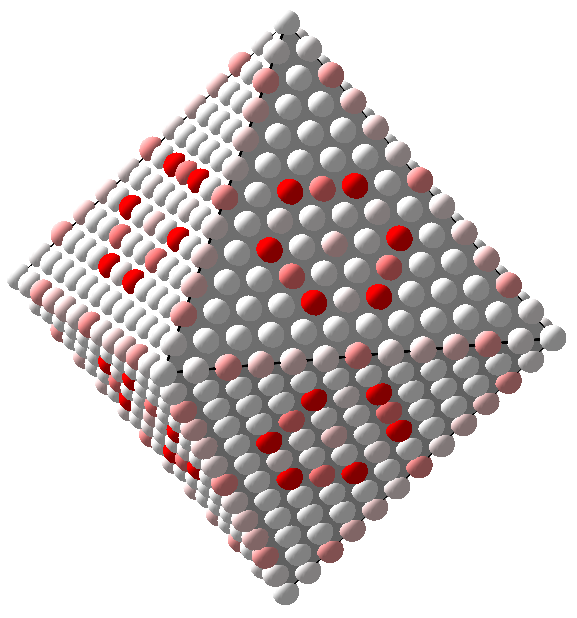}
			\caption{$r=12$}
		\end{subfigure}
		\quad
		\begin{subfigure}{.4\textwidth}
			\centering
			\includegraphics[width=6cm]{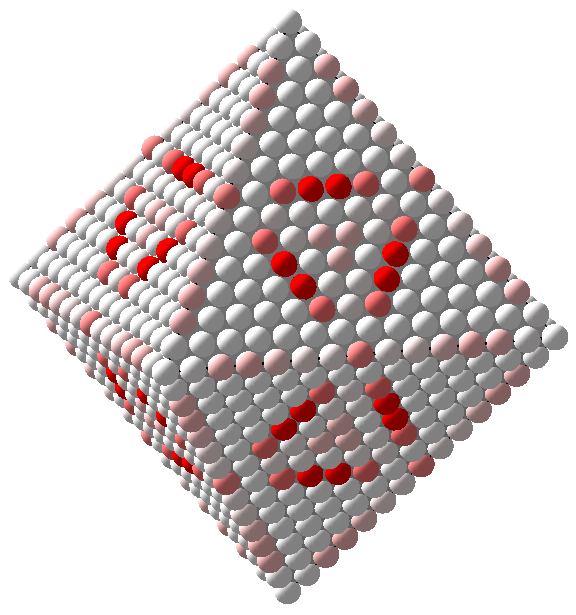}
			\caption{$r=14$}
		\end{subfigure}
		\caption{The coefficients $c_\alpha$ for $F(r,9)$ in the case of $\RootB[3]$, indicated by the intensity of the color as $\mathrm{RGB}(1,1-(c_\alpha-c_{\min})/(c_{\max}-c_{\min}),1-(c_\alpha-c_{\min})/(c_{\max}-c_{\min}))$.}
		\label{B3L1NormCoefficients}
	\end{center}
\end{figure}

\begin{figure}[H]
	\begin{minipage}{0.2\textwidth}
		\begin{flushright}
			\begin{overpic}[width=1\textwidth,,tics=10]{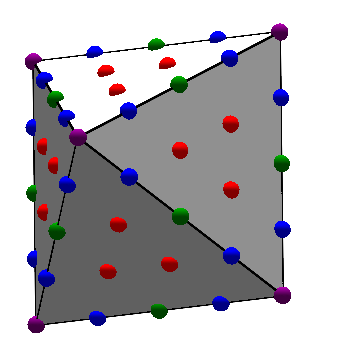}
				\put (83, 90) {\large \textcolor{violet}{$\displaystyle 0.01754$}}
				\put (83, 73) {\large \textcolor{blue}{$\displaystyle 0.22680$}}
				\put (70, 60) {\large \textcolor{red}{$\displaystyle 0.59375$}}
				\put (83, 47) {\large \textcolor{OliveGreen}{$\displaystyle 0.16189$}}
			\end{overpic}
		\end{flushright}
	\end{minipage} \hfill
	\begin{minipage}{0.75\textwidth}
		\begin{table}[H]
			\begin{center}
				\begin{tabular}{|c|c||c|c|}
					\hline
					\multicolumn{2}{|c|}{$\RootC[3]$}								&	\multicolumn{2}{|c|}{$\RootB[3]$}								\\
					\hline
					\hline
					$1-1/F(4,7)$	&	$c_\alpha$									&	$1-1/F(4,9)$	&	$c_\alpha$									\\
					\hline
					$6.28148$		&	\textcolor{violet}{$c_{400}=0.01752	$}		&	$6.28148$		&	\textcolor{violet}{$c_{400}=0.01754$}		\\
					&	\textcolor{blue}{$c_{210}=0.22681	$}		&					&	\textcolor{blue}{$c_{210}=0.22680$}			\\
					&	\textcolor{red}{$c_{101}=0.59380	$}		&					&	\textcolor{red}{$c_{102}=0.59375$}			\\
					&	\textcolor{OliveGreen}{$c_{020}=0.16185$}	&					&	\textcolor{OliveGreen}{$c_{020}=0.16189$}	\\
					\hline
				\end{tabular}
			\end{center}
		\end{table}
	\end{minipage}
	\centering
	\caption{The crosspolytope with radius $r=4$ and the obtained optimal coefficients. Supporting points $\weight=\alpha_1\,\fweight{1}+\alpha_2\,\fweight{2}+\alpha_3\,\fweight{3}$ in the same Weyl group orbit have the same coefficients $c_\alpha$, denoted by red, blue, green and purple dots.}
	\label{C3B3L1NormTable}
\end{figure}

\begin{remark}
This computation confirms the lower bound $7$ from \emph{\cite[Prop. 9]{furedi04}}.
\end{remark}
	
\subsubsection*{Dimension $n = 4$}

\begin{table}[H]
\begin{center}
	\begin{tabular}{|c|c||c|c|c|c|c|c|c|c|}
		\hline
		$\Roots$	&	$d \backslash r$	&	$2$			&	$4$			&	$6$			&
		$8$			&	$10$		&	$12$		&	$14$		\\
		\hline
		\hline
		$\RootB[4]$	&	$4$					&	$8.00000$	&	$10.33968$	&	$9.09234$	&
		$10.33968$	&	$-		$	&	$-		$	&	$-		$	\\
		\hline
		&	$5$					&	$8.00000$	&	$10.33969$	&	$9.72339$	&
		$10.33969$	&	$9.17503$	&	$-		$	&	$-		$	\\
		\hline
		&	$6$					&	$8.00000$	&	$10.83655$	&	$10.18050$	&
		$10.33969$	&	$9.90514$	&	$10.33968$	&	$-		$	\\
		\hline
		&	$7$					&	$8.00000$	&	$10.86019$	&	$10.51696$	&
		$10.51282$	&	$10.16103$	&	$10.33968$	&	$10.03938$	\\
		\hline
		\hline
		$\RootC[4]$	&	$4$					&	$8.00000$	&	$10.33993$	&	$9.72014$	&
		$10.33968$	&	$-		$	&	$-		$	&	$-		$	\\
		\hline
		&	$5$					&	$8.00000$	&	$10.83902$	&	$10.07664$	&
		$10.33968$	&	$9.94864$	&	$-		$	&	$-		$	\\
		\hline
		\hline
		$\RootD[4]$	&	$4$					&	$8.00000$	&	$10.34750$	&	$9.08887$	&
		$10.33969$	&	$-		$	&	$-		$	&	$-		$	\\
		\hline
		&	$5$					&	$8.00000$	&	$10.39184$	&	$9.72430$	&
		$10.34011$	&	$9.52887$	&	$-		$	&	$-		$	\\
		\hline
		&	$6$					&	$8.00000$	&	$10.83844$	&	$10.34886$	&
		$10.35578$	&	$9.97888$	&	$10.33971$	&	$-		$	\\
		\hline
	\end{tabular}
\end{center}
\caption{The bound $\chi (\Z^4,\,\mathbb{B}^1_r) \geq 1-1/F(r,d)$.}
\label{B4C4D4L1Table2}
\end{table}

The value $\chi (\Z^4,\,\mathbb{B}^1_2)=8$ is obtained immediately with $F(2,1)$. The highest value is $F(4,7)$ for $\RootB[4]$. None of the computed bounds $F(r,d)$ is stable in $d$ and we are limited by the size of the semi--definite program, see \Cref{SDPMatrixNumberSizeTable}. Again, in the case of $\RootB[4]$ for example, we see that $F(4,7)\geq F(8,7)$, because we do not take the limit.

\begin{remark}
This computation improves the lower bound $9$ from \emph{\cite[Prop. 9]{furedi04}} by $+2$.
\end{remark}

\subsection{The chromatic number of $\R^n$ for Vorono\"{i} cells}

Finally we consider the case of the Euclidean space $V=\R^n$ as a set of vertices, where the avoided set $S=\partial\mathcal{P}$ is the boundary of a convex centrally--symmetric polytope $\mathcal{P}$. This setting was studied in \cite{BBMP}, giving bounds on $\chi_m(\R^n,\partial\mathcal{P})$ without using spectral bounds. There it was proven that $\chi_m(\R^n,\partial\mathcal{P}) \leq 2^n$ whenever $\mathcal{P}$ tiles $\R^n$ and equality is conjectured. We now investigate the strength of the spectral bound for certain instances of this graph.

\begin{figure}[H]
	\begin{center}
		\includegraphics[height=4cm]{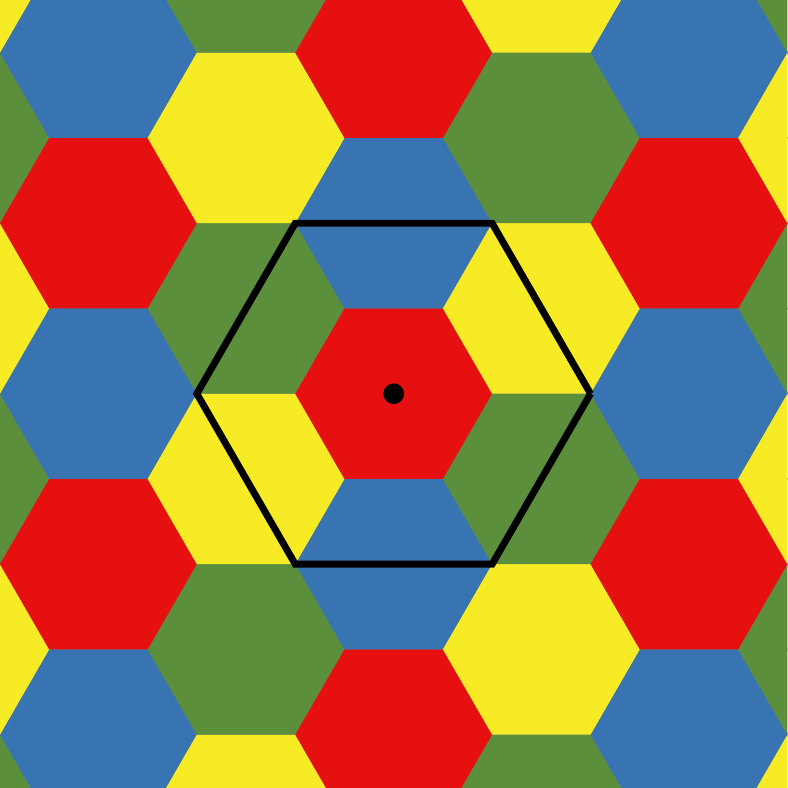}
	\end{center}
	\caption{The chromatic number of $\R^2$ for the hexagon is $2^2=4$ \cite{BBMP}.}\label{fig_chromaticnumberhexagon}
\end{figure}

Given a Weyl group $\weyl$ associated to a root systems in $\R^n$, the Vorono\"{i} cell of the coroot lattice $\Corootlattice$ is a convex centrally--symmetric polytope, invariant under $\weyl$ and tiles $\R^n$ by $\Corootlattice$--translation, see \Cref{eq_VoronoiTiles}. If the root system is irreducible with highest root $\highestroot$, then we have $\Vor( \Corootlattice ) = \weyl\,\fundom$, where
\[
\fundom
=	\{ u \in \R^n \,\vert\, \forall \, 1\leq i \leq n: \, \sprod{u,\roots_i} \geq 0 \mbox{ and } \sprod{u,\highestroot} \leq 1 \}
\]
is a fundamental domain of the affine Weyl group $\weyl\ltimes \Corootlattice$, see \Cref{prop_FundomAffineWeyl}. In particular, the part of the boundary $\partial\Vor( \Corootlattice ) \cap \overline{\PC}$, which is also contained in the fundamental Weyl chamber, lies on a hyperplane parallel to $\sprod{ \cdot , \highestroot^\vee } = 0$. Rescaling the polytope $\Vor(\Corootlattice)$ by a factor $\tilde{r}>0$ does not affect the chromatic number, that is, $\chi_m(\R^n,\partial\Vor( \Corootlattice )) = \chi_m(\R^n,\partial(\tilde{r}\,\Vor( \Corootlattice )))$. If we choose $\tilde{r}= r\,\sprod{\highestroot,\highestroot}/2$ for some $0 \neq r \in \N$, then $\partial(\tilde{r}\,\Vor( \Corootlattice )) \cap \Weights \neq \emptyset$ and we obtain a hierarchy of lower bounds
\begin{equation}\label{eq_VoronoiLowerBound}
		\chi_m ( \R^n , \partial \Vor ( \Corootlattice ) )
\geq	\ldots
\geq 	1-\frac{1}{F(4r)}
\geq 	1-\frac{1}{F(2r)}
\geq 	1-\frac{1}{F(r)}
\geq 	1-\frac{1}{F(1)},
\end{equation}
where $F(r):=F(S_r)$ is as in \Cref{thm_ChromaticChebyshevBound} with $S_r:=\weyl\{ u \in \overline{\PC} \,\vert\, \sprod{u,\highestroot^\vee} = r \}$. 

\begin{remark}
The quantity $1-1/F(r)$ is a lower bound for $\chi_m ( \R^n , \partial \Vor ( \Corootlattice ) )$. More precisely, we have
\[
		\chi_m ( \R^n , \partial \Vor ( \Corootlattice ) )
\geq	\chi ( \Weights , S_r )
\geq	1-\frac{1}{F(r)}
\]
and $F(r)$ is the minimum of the Fourier transformation of the optimal measure $\measure$ (with mass $1$) in \emph{\Cref{thm_spectral_bound}} for the graph $G(\Weights,S_r)$.
\end{remark}

To compute $F(r)$ numerically, we use \Cref{coro_ChromBoundLasserre} and write $F(r,d) := F(S_r,d)$. Note that, in this case, $F(r,d) \geq F(\ell\,r,d)$ is only certain for $\ell\in\N$ when $d\to \infty$.

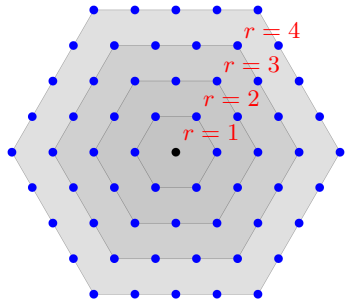
\begin{figure}[H]
	\begin{center}
		\begin{overpic}[height=4cm,,tics=10]{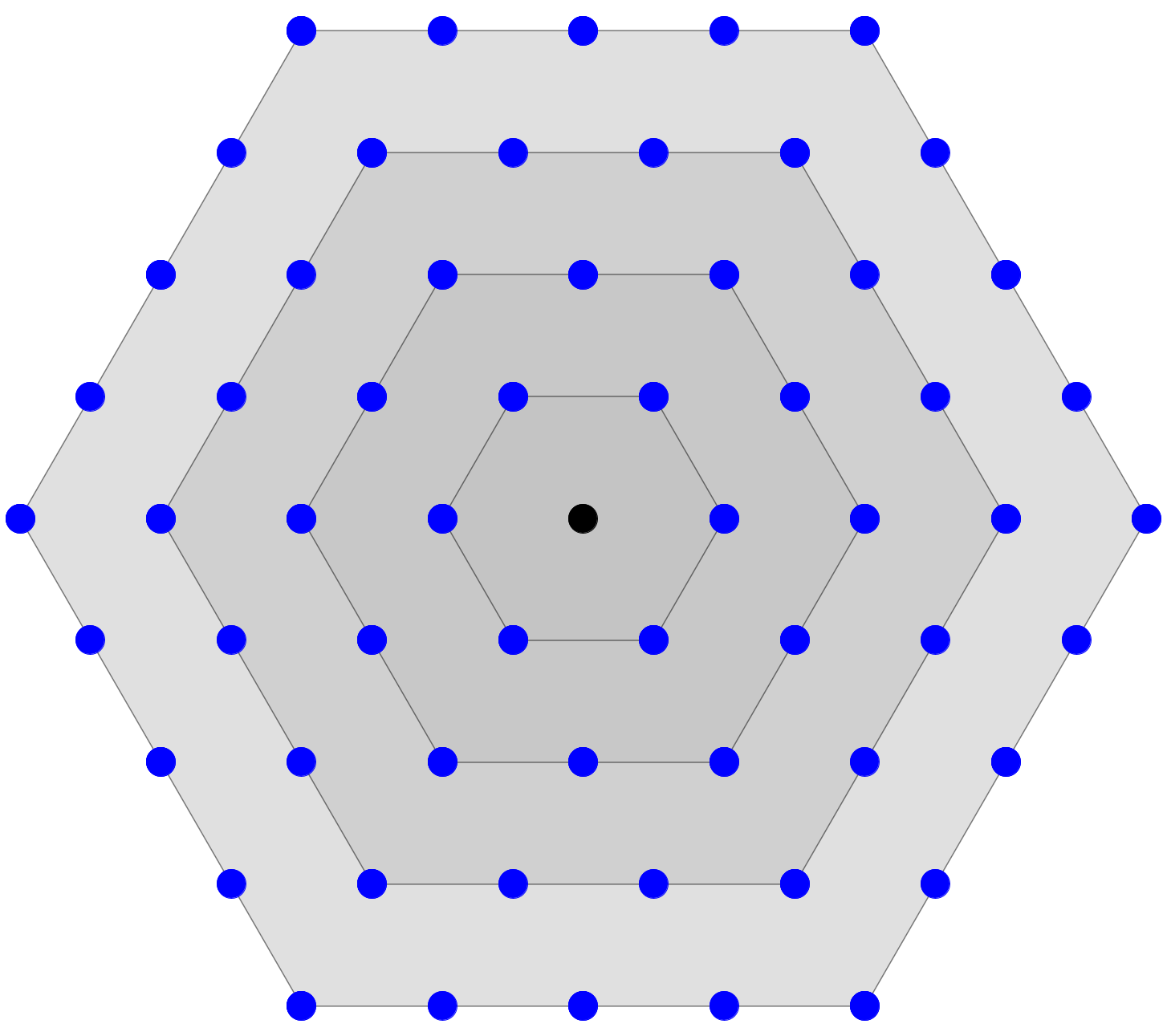}
			\put (52, 48) {\small \textcolor{red}{$\displaystyle r=1$}}
			\put (58, 58) {\small \textcolor{red}{$\displaystyle r=2$}}
			\put (64, 68) {\small \textcolor{red}{$\displaystyle r=3$}}
			\put (70, 78) {\small \textcolor{red}{$\displaystyle r=4$}}
		\end{overpic}
	\end{center}
	\caption{Rescaling the hexagon increases the number of weights $S_r\cap\Weights$ on the boundary.}\label{fig_UnitDist}
\end{figure}

\subsubsection{The hexagon in $\R^2$}

The hexagon in $\R^2\cong \R^3/\langle [1,1,1]^t\rangle$, as it has appeared several times now in the article, is the Vorono\"{i} cell of the coroot lattice $\Corootlattice$ for $\RootA[2]$ and $\RootG[2]$. It has $6$ vertices and $6$ edges.

For $\RootA[2]$, the vertices of the hexagon are the orbits of the fundamental weights $\fweight{1}$ and $\fweight{2}$. The centers of the edges are the orbit of $(\fweight{1}+\fweight{2})/2$. We fix a hierarchy order $d\geq 3$ and consider $F(r,d)$ for $1 \leq r\leq 2d$.

For $\RootG[2]$, the vertices are the orbit of $\fweight{1}/3$. The centers of edges are the orbit of $\fweight{2}/6$. If $r\in\N$ is not a multiple of $3$, then $S_r=\emptyset$. Thus we consider $F(3r,d)$ for $1 \leq r\leq 2d$, but still write $F(r,d)$.

The first column indicates the root system, that is, $\RootA[2]$ or $\RootG[2]$. Then the rows are indexed by the relaxation order $d$ and the columns by the radius $r$.

\begin{table}[H]
	\begin{center}
		\scalebox{0.65}{
			\begin{tabular}{|c|c||c|c|c|c|c|c|c|c|c|c|c|c|c|c|}
				\hline
				$\Roots$	&	$d \backslash r$	&	$1$			&	$2$			&	$3$			&	$4$			&	$5$			&	$6$			&	$7$			&	$8$			&
				$9$			&	$10$		&	$11$		&	$12$		&	$13$		&	$14$		\\
				\hline
				\hline
				$\RootA[2]$	&	$3$					&	$2.99386$	&	$3.57143$	&	$3.52451$	&	$3.57143$	&	$3.37484$	&	$3.57143$	&	$-		$	&	$-		$	&
				$-		$	&	$-		$	&	$-		$	&	$-		$	&	$-		$	&	$-		$	\\
				\hline
				&	$4$					&	$3.00000$	&	$3.57143$	&	$3.52911$	&	$3.57143$	&	$3.54698$	&	$3.57143$	&	$3.47461$	&	$3.57143$	&
				$-		$	&	$-		$	&	$-		$	&	$-		$	&	$-		$	&	$-		$	\\
				\hline
				&	$5$					&	$3.00000$	&	$3.57143$	&	$3.52912$	&	$3.57143$	&	$3.54789$	&	$3.57143$	&	$3.54016$	&	$3.57143$	&
				$3.51384$	&	$3.57143$	&	$-		$	&	$-		$	&	$-		$	&	$-		$	\\
				\hline
				&	$6$					&	$3.00000$	&	$3.57143$	&	$3.52912$	&	$3.57143$	&	$3.54789$	&	$3.57143$	&	$3.54786$	&	$3.57143$	&
				$3.55920$	&	$3.57143$	&	$3.47623$	&	$3.57143$	&	$-		$	&	$-		$	\\
				\hline
				&	$7$					&	$3.00000$	&	$3.57143$	&	$3.52912$	&	$3.57143$	&	$3.54789$	&	$3.57143$	&	$3.55183$	&	$3.57143$	&
				$3.55921$	&	$3.57143$	&	$3.51433$	&	$3.57143$	&	$3.14739$	&	$3.57143$	\\
				\hline
				&	$8$					&	$3.00000$	&	$3.57143$	&	$3.52912$	&	$3.57143$	&	$3.54789$	&	$3.57143$	&	$3.55347$	&	$3.57143$	&
				$3.55921$	&	$3.57143$	&	$3.53571$	&	$3.57143$	&	$3.25411$	&	$3.57143$	\\
				\hline
				\hline
				$\RootG[2]$	&	$3$					&	$2.99732$	&	$3.57143$	&	$3.39930$	&	$3.57143$	&	$2.47997$	&	$3.57143$	&	$-		$	&	$-		$	&
				$-		$	&	$-		$	&	$-		$	&	$-		$	&	$-		$	&	$-		$	\\
				\hline
				&	$4$					&	$2.99962$	&	$3.57143$	&	$3.52821$	&	$3.57143$	&	$3.41805$	&	$3.57143$	&	$2.54024$	&	$3.57143$	&
				$-		$	&	$-		$	&	$-		$	&	$-		$	&	$-		$	&	$-		$	\\
				\hline
				&	$5$					&	$3.00000$	&	$3.57143$	&	$3.52908$	&	$3.57143$	&	$3.49102$	&	$3.57143$	&	$2.76603$	&	$3.57143$	&
				$2.45902$	&	$3.57143$	&	$-		$	&	$-		$	&	$-		$	&	$-		$	\\
				\hline
				&	$6$					&	$3.00000$	&	$3.57143$	&	$3.52912$	&	$3.57143$	&	$3.52318$	&	$3.57143$	&	$3.39290$	&	$3.57143$	&
				$2.70265$	&	$3.57143$	&	$2.98423$	&	$3.57143$	&	$-		$	&	$-		$	\\
				\hline
				&	$7$					&	$3.00000$	&	$3.57143$	&	$3.52912$	&	$3.57143$	&	$3.54301$	&	$3.57143$	&	$3.54780$	&	$3.57143$	&
				$3.53627$	&	$3.57143$	&	$3.28144$	&	$3.57143$	&	$2.50993$	&	$3.57143$	\\
				\hline
				&	$8$					&	$3.00000$	&	$3.57143$	&	$3.52912$	&	$3.57143$	&	$3.54656$	&	$3.57143$	&	$3.55294$	&	$3.57143$	&
				$3.54181$	&	$3.57143$	&	$3.54139$	&	$3.57143$	&	$3.13764$	&	$3.57143$	\\
				\hline
			\end{tabular}
		}
	\end{center}
	\caption{The bound $\chi_m (\R^2,\,\partial \Vor(\Corootlattice(\RootA[2]))) = \chi_m (\R^2,\,\partial \Vor(\Corootlattice(\RootG[2]))) \geq 1-1/F(r,d)$ for the hexagon.}
	\label{A2HexagonTable}
\end{table}

For $r=1$, there is no choice for the coefficients $c_\weight$, as $S_1$ only contains one element in both cases $\RootA[2]$ and $\RootG[2]$. The value $F(1)$ is $-1/2$. This gives spectral bound $3$ and is obtained from $F(r,d)$ for $d\geq 4$, respectively $d\geq 5$. Furthermore, this fits with the bound from \Cref{thm_LatticeChromatic}, where $\chi(\Corootlattice) \geq n$ for $\RootA$.

For $r\geq 2$, the best possible bound we obtained is already assumed at $r=2$ and $d=3$. We display the optimal coefficients for the corresponding measure below. This bound is assumed in all $F(r,d)$ with $r$ even at lowest possible order. For $r$ odd, the value converges but does not stabilize. 

Although we recover that the chromatic number of $\R^2$ for the hexagon is $4$, see \Cref{fig_UnitDist}, our computations indicate that the spectral bound is not sharp and never will be with $r,d\to\infty$.

\begin{figure}[H]
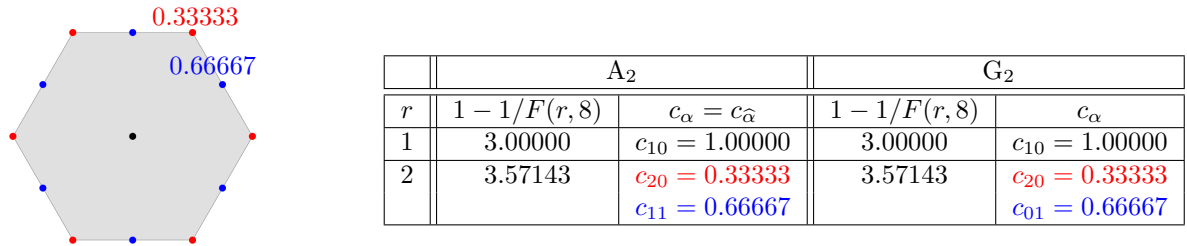

	\begin{minipage}{0.2\textwidth}
		\begin{flushright}
			\begin{overpic}[width=1\textwidth,,tics=10]{A2Level2Coeff}
				\put (58,95) {\textcolor{red}{$\displaystyle 0.33333$}}
				\put (65,75) {\textcolor{blue}{$\displaystyle 0.66667$}}
			\end{overpic}
		\end{flushright}
	\end{minipage} \hfill
	\begin{minipage}{0.75\textwidth}
		\begin{table}[H]
			\begin{center}
				\scalebox{1}{
					\begin{tabular}{|c||c|c||c|c|}
						\hline
						&	\multicolumn{2}{c||}{$\RootA[2]$}					&	\multicolumn{2}{c|}{$\RootG[2]$}			\\
						\hline
						\hline
						$r$		&	$1-1/F(r,8)$	&	$c_\alpha=c_{\conj{\alpha}}$
						&	$1-1/F(r,8)$	&	$c_\alpha$						\\
						\hline
						$1$		&	$3.00000$		&	$c_{10}=1.00000$				&	$3.00000$	&	$c_{10}=1.00000$	\\
						\hline
						$2$		&	$3.57143$		&	\textcolor{red}{$c_{20}=0.33333$}				&	$3.57143$	&	\textcolor{red}{$c_{20}=0.33333$}	\\
						&					&	\textcolor{blue}{$c_{11}=0.66667$}				&				&	\textcolor{blue}{$c_{01}=0.66667$}	\\
						\hline
					\end{tabular}
				}
			\end{center}
		\end{table}
	\end{minipage}
	\centering
	\caption{The scaled Vorono\"{i} cell and the optimal coefficients for $F(2,8)$. Supporting points $\weight=\alpha_1\,\fweight{1}+\alpha_2\,\fweight{2}$ in the same Weyl group orbit and their additive inverse $\conj{\weight}$ have the same coefficients $c_\alpha=c_{\conj{\alpha}}$, denoted by either red or blue dots.}
	\label{HexagonTable}
\end{figure}

From \Cref{HexagonTable}, we recover the coefficients $1/3$ for the vertices and $2/3$ for the centers of faces, indicating the best possible discrete measure. Indeed, for $r\in\N$, we have
\begin{equation}\label{remark_A2G2Min}
\begin{split}
	F(2r)
=&	\begin{cases}
	\min\limits_{z\in\Image} \frac{2}{3} \, T_{r\,r}(z) + \frac{1}{6} \, (T_{2r\,0}(z)+T_{0\,2r}(z)) = \min\limits_{z\in\Image} \frac{2}{3} \, T_{1\,1}(z) +  + \frac{1}{6} \, (T_{2\,0}(z)+T_{0\,2}(z)) , & \tbox{if} \Roots=\RootA[2]\\
	\min\limits_{z\in\Image} \frac{2}{3} \, T_{0\,r}(z) + \frac{1}{3} \, T_{2r\,0}(z) = \min\limits_{z\in\Image} \frac{2}{3} \, T_{0\,1}(z) + \frac{1}{3} \, T_{2\,0}(z) , & \tbox{if} \Roots=\RootG[2]
\end{cases}\\
=&	\min\limits_{z\in\Image} 2\,z_1^2 - 2/3\,z_1 - 1/3
=	-7/18
\end{split}
\end{equation}
(for $\RootA[2]$, we have to substitute $z_i=z_1\pm \mathrm{i}\,z_2$, so that $\Image\subseteq \R^2$). In both cases, $1-1/F(2r) = 25/7 \approx 3.57143$. Note that $F(2)$ corresponds to the trigonometric polynomial in \Cref{example_A2PolyRewrite} up to a factor $1/3$.

\begin{figure}[H]
\begin{center}		
	\begin{subfigure}{.4\textwidth}
		\centering
		\includegraphics[height=6.4cm]{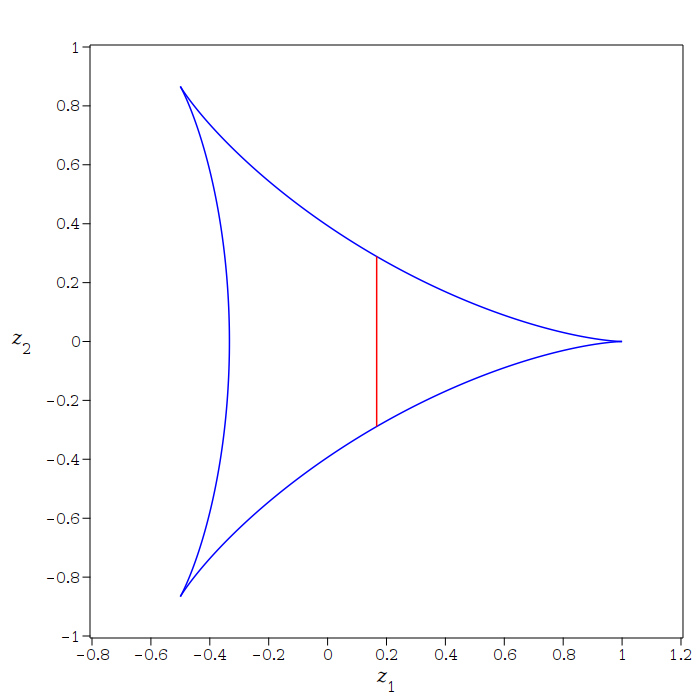}
		\hspace{.1cm}
		\includegraphics[height=6cm]{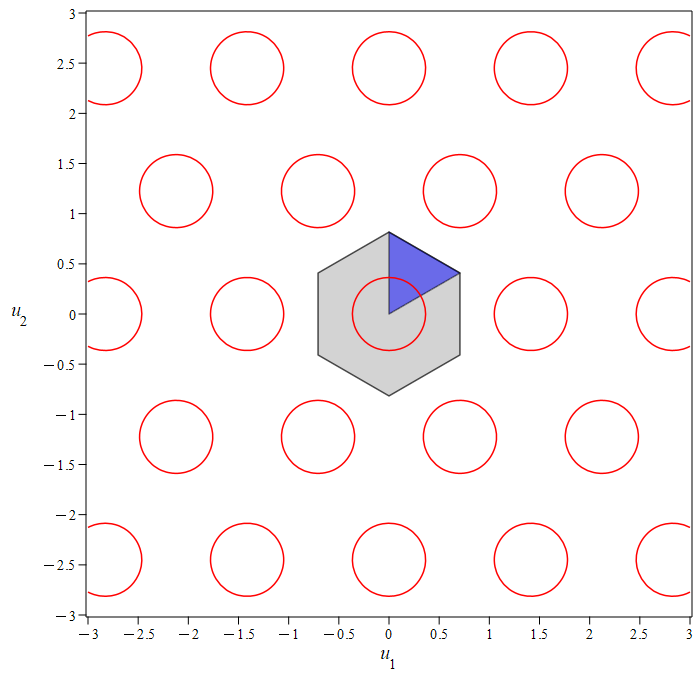}
		\caption{$\RootA[2]$}
		\label{A2Min}
	\end{subfigure}
	\quad
	\begin{subfigure}{.4\textwidth}
		\centering
		\includegraphics[height=6.4cm]{G2DeltoidMin.png}
		\hspace{.1cm}
		\includegraphics[height=6cm]{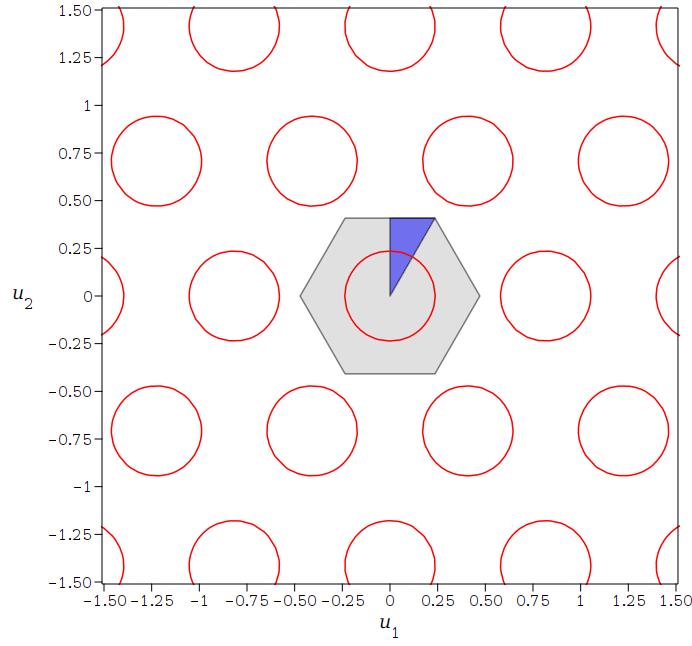}
		\caption{$\RootG[2]$}
		\label{G2Min}
	\end{subfigure}
	\caption{The minimizers $z$ (lines, above) for $F(2r)$ in the image $\Image$ of the generalized cosines with preimages $u$ (ovals, below). In the coordinates $u$, we can observe the periodicity with respect to the coroot lattice $\Corootlattice$ as well as the $\weyl$--invariance, yielding the crystallographic symmetry on the alcove $\fundom$ of $\weyl\ltimes\Corootlattice$ (simplex).}
	\label{fig_A2G2Min}
\end{center}
\end{figure}

\subsubsection{The rhombic dodecahedron in $\R^3$}

The rhombic dodecahedron in $\R^3$ (\Cref{RhombicDodecahedron}) is the Vorono\"{i} cell of the coroot lattice $\Corootlattice$ for $\RootA[3]$ and $\RootB[3]$. It has $14$ vertices, $24$ edges and $12$ faces.

For $\RootA[3]$, the vertices are the orbits of $\fweight{1}$, $\fweight{2}$ and $\fweight{3}$. The centers of the edges are the orbits of $(\fweight{i}+\fweight{2})/2$ for $i=1,2$, and the centers of the facets are the orbit of $(\fweight{1}+\fweight{3})/2$.

For $\RootB[3]$, the vertices are the orbits of $\fweight{1}$ and $\fweight{3}$. The centers of the edges are the orbit of $(\fweight{1}+\fweight{3})/2$, and the centers of the facets are the orbit of $\fweight{2}/2$.

\begin{figure}[H]
	\begin{center}
		\begin{subfigure}{.3\textwidth}
			\centering
			\includegraphics[width=4cm, height=4cm]{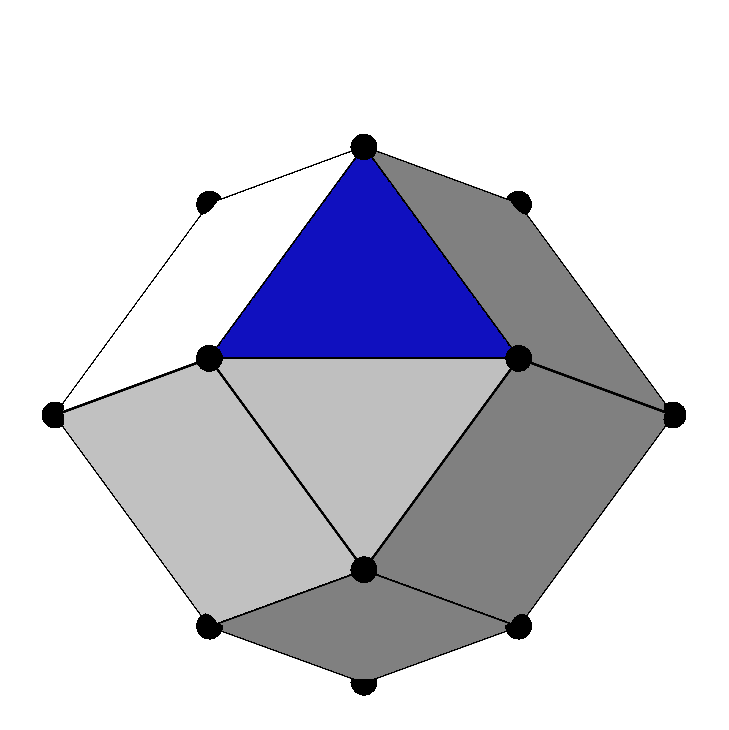}
			\caption{$\RootA[3]$}
			\label{RhombicDodecahedronA}
		\end{subfigure}
		\quad
		\begin{subfigure}{.3\textwidth}
			\centering
			\includegraphics[width=4cm, height=4cm]{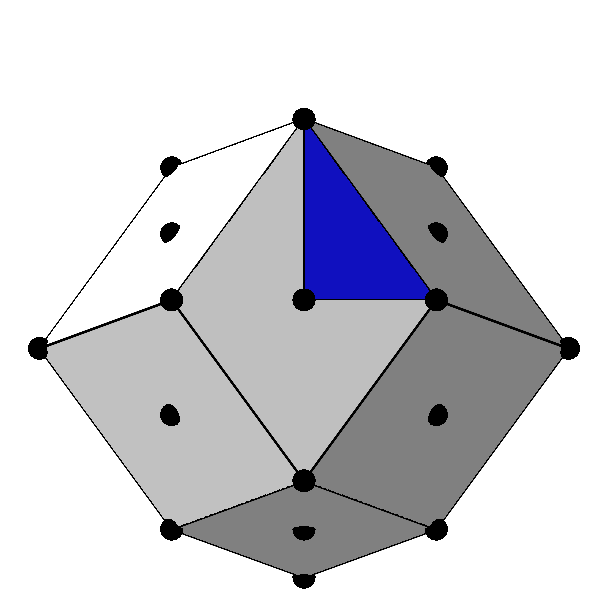}
			\caption{$\RootB[3]$}
			\label{RhombicDodecahedronB}
		\end{subfigure}
		\caption{The rhombic dodecahedron is the Vorono\"{i} cell of the coroot lattice for $\RootA[3]$ and $\RootB[3]$.}
	\label{RhombicDodecahedron}
\end{center}
\end{figure}

\begin{table}[H]
\begin{center}
	\scalebox{0.65}{
		\begin{tabular}{|c|c||c|c|c|c|c|c|c|c|c|c|c|c|c|c|}
			\hline
			$\Roots$	&	$d \backslash r$	&	$1$			&	$2$			&	$3$			&	$4$			&	$5$			&	$6$			&	$7$			&	$8$			&
			$9$			&	$10$		&	$11$		&	$12$		&	$13$		&	$14$		\\
			\hline
			\hline
			$\RootA[3]$	&	$4$					&	$3.99424$	&	$6.10767$	&	$5.86933$	&	$6.10766$	&	$5.81858$	&	$6.10766$	&	$4.77576$	&	$6.10766$	&
			$-		$	&	$-		$	&	$-		$	&	$-		$	&	$-		$	&	$-		$	\\
			\hline
			&	$5$					&	$3.99611$	&	$6.10767$	&	$5.86964$	&	$6.10766$	&	$5.90988$	&	$6.10767$	&	$5.85369$	&	$6.10766$	&
			$5.46888$	&	$6.10766$	&	$-		$	&	$-		$	&	$-		$	&	$-		$	\\
			\hline
			&	$6$					&	$3.99653$	&	$6.10767$	&	$5.86972$	&	$6.10767$	&	$5.93658$	&	$6.10767$	&	$5.85762$	&	$6.10766$	&
			$5.85825$	&	$6.10766$	&	$3.78978$	&	$6.10766$	&	$-		$	&	$-		$	\\
			\hline
			&	$7$					&	$3.99702$	&	$6.10767$	&	$5.86988$	&	$6.10767$	&	$5.94146$	&	$6.10766$	&	$5.96334$	&	$6.10767$	&
			$5.85986$	&	$6.10766$	&	$4.12186$	&	$6.10766$	&	$-		$	&	$6.10766$	\\
			\hline
			&	$8$					&	$3.99719$	&	$6.10767$	&	$5.86992$	&	$6.10767$	&	$5.94327$	&	$6.10767$	&	$6.05399$	&	$6.10767$	&
			$5.86357$	&	$6.10766$	&	$5.59839$	&	$6.10766$	&	$3.88490$	&	$6.10766$	\\
			\hline
			\hline
			$\RootB[3]$	&	$3$					&	$3.83791$	&	$6.10767$	&	$3.39918$	&	$6.10766$	&	$-		$	&	$6.10766$	&	$-		$	&	$-		$	&
			$-		$	&	$-		$	&	$-		$	&	$-		$	&	$-		$	&	$-		$	\\
			\hline
			&	$4$					&	$3.84571$	&	$6.10767$	&	$4.11626$	&	$6.10766$	&	$-		$	&	$6.10766$	&	$-		$	&	$6.10766$	&
			$-		$	&	$-		$	&	$-		$	&	$-		$	&	$-		$	&	$-		$	\\
			\hline
			&	$5$					&	$3.98454$	&	$6.10767$	&	$5.80542$	&	$6.10766$	&	$5.08174$	&	$6.10767$	&	$-		$	&	$6.10766$	&
			$-		$	&	$6.10766$	&	$-		$	&	$-		$	&	$-		$	&	$-		$	\\
			\hline
			&	$6$					&	$3.99667$	&	$6.10767$	&	$5.87057$	&	$6.10767$	&	$5.86644$	&	$6.10767$	&	$5.82630$	&	$6.10766$	&
			$-		$	&	$6.10766$	&	$-		$	&	$6.10766$	&	$-		$	&	$-		$	\\
			\hline
			&	$7$					&	$3.99872$	&	$6.10767$	&	$5.87057$	&	$6.10767$	&	$5.94578$	&	$6.10766$	&	$5.96989$	&	$6.10767$	&
			$5.88810$	&	$6.10766$	&	$-		$	&	$6.10766$	&	$-		$	&	$6.10766$	\\
			\hline
			&	$8$					&	$3.99925$	&	$6.10767$	&	$5.87057$	&	$6.10767$	&	$5.96374$	&	$6.10767$	&	$5.99825$	&	$6.10767$	&
			$5.94949$	&	$6.10766$	&	$5.92157$	&	$6.10766$	&	$5.31568$	&	$6.10766$	\\
			\hline
			&	$9$					&	$3.99972$	&	$6.10767$	&	$5.87057$	&	$6.10767$	&	$5.97050$	&	$6.10767$	&	$6.00193$	&	$6.10767$	&
			$5.98345$	&	$6.10767$	&	$5.98654$	&	$6.10766$	&	$5.93977$	&	$6.10766$	\\
			\hline
		\end{tabular}
	}
\end{center}
	\caption{The bound $\chi_m (\R^3,\,\partial \Vor(\Corootlattice(\RootA[3]))) = \chi_m (\R^3,\,\partial \Vor(\Corootlattice(\RootB[3]))) \geq 1-1/F(r,d)$ for the rhombic dodecahedron.}
\label{A3RhombicTable}
\end{table}

For $r=1$, the numerically computed bound seems to converge to $4$. For $r\geq 2$, the best possible bound we obtain is already assumed at $r=2$ and $d=3$, respectively $d=4$. We display the optimal coefficients for the corresponding measure below. This bound is approximately assumed in all $F(r,d)$ with $r$ even at lowest possible order $d$. For $r$ odd, the value does not stabilize with $r$ or $d$ growing. $\RootA[3]$ and $\RootB[3]$ give the same coefficients for the same supporting points.
As in the case of the hexagon, the gap between the spectral bound for such discrete measures and the actual chromatic number of $\R^3$ for the rhombic dodecahedron (known to be $8$ by \cite{BBMP}) seems quite large.

\begin{figure}[H]
\begin{minipage}{0.2\textwidth}
	\begin{flushright}
		\begin{overpic}[width=1\textwidth,,tics=10]{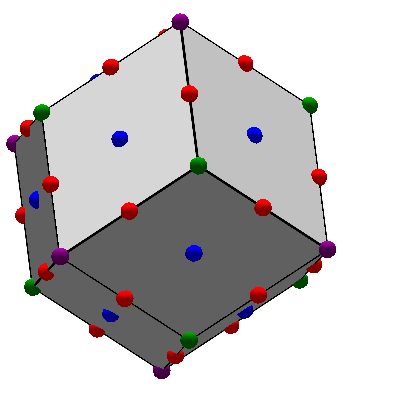}
			\put (47, 97) {\large \textcolor{violet}{$\displaystyle 0.10283$}}
			\put (67, 84) {\large \textcolor{red}{$\displaystyle 0.24388$}}
			\put (81, 73) {\large \textcolor{OliveGreen}{$\displaystyle 0.06050$}}
			\put (70, 61) {\large \textcolor{blue}{$\displaystyle 0.59279$}}
		\end{overpic}
	\end{flushright}
\end{minipage} \hfill
\begin{minipage}{0.75\textwidth}
	\begin{table}[H]
		\begin{center}
			\begin{tabular}{|c||c|c||c|c|}
				\hline
				&	\multicolumn{2}{c||}{$\RootA[3]$}												&	\multicolumn{2}{c|}{$\RootB[3]$}								\\
				\hline
				\hline
				$r$					&	$1-1/F(r,8)$	&	$c_\alpha=c_{\conj{\alpha}}$			&	$1-1/F(r,9)$	&	$c_\alpha$									\\
				\hline
				$1$					&	$3.99719$		&	$c_{010}=0.33298$						&	$3.99972$		&	$c_{100}=0.33332$							\\
				&					&	$c_{100}=0.66702$											&					&	$c_{001}=0.66668$							\\
				\hline
				$2$					&	$6.10767$		&	\textcolor{violet}{$c_{020}=0.10282$}	&	$6.10767$		&	\textcolor{violet}{$c_{200}=0.10283$}		\\
				&					&	\textcolor{red}{$c_{110}=0.24392$}							&					&	\textcolor{red}{$c_{101}=0.24388$}			\\
				&					&	\textcolor{OliveGreen}{$c_{200}=0.06050$}					&					&	\textcolor{OliveGreen}{$c_{002}=0.06050$}	\\
				&					&	\textcolor{blue}{$c_{101}=0.59276$}							&					&	\textcolor{blue}{$c_{010}=0.59279$}			\\
				\hline
			\end{tabular}
		\end{center}
	\end{table}
\end{minipage}
\centering
\caption{The scaled Vorono\"{i} cell and the obtained optimal coefficients. Supporting points $\weight=\alpha_1\,\fweight{1}+\alpha_2\,\fweight{2}+\alpha_3\,\fweight{3}$ in the same Weyl group orbit and their additive inverse $\conj{\weight}$ have the same coefficients $c_\alpha=c_{\conj{\alpha}}$, denoted by red, blue, green and purple dots.}
\label{RhombicDodecahedronTable}
\end{figure}

As we can observe, the most weight is put on the center of faces, then on the centers of edges and only small weight is put on the vertices. We investigate the minimizers of the associated sum of generalized Chebyshev polynomials. Similar to \Cref{remark_A2G2Min}, one finds the following.
\begin{enumerate}
	\item For $\Roots = \RootB[3]$, the minimizers for $F(2,8)$ are $z_{\mathrm{min}} \approx	(0.05927, z_2, 0.22212)$ with $z_2\in\R$ so that $z_{\mathrm{min}}\in\Image$.	
	\item For $\Roots = \RootA[3]$, the minimizers for $F(2,8)$ are $z_{\mathrm{min}} \approx	(0.22209, 0.05915, z_3)$ with $z_3\in\R$ so that $z_{\mathrm{min}}\in\Image$.
\end{enumerate}

\begin{figure}[H]
\begin{center}
	\begin{subfigure}{.3\textwidth}
		\centering
		\includegraphics[width=4cm, height=4cm]{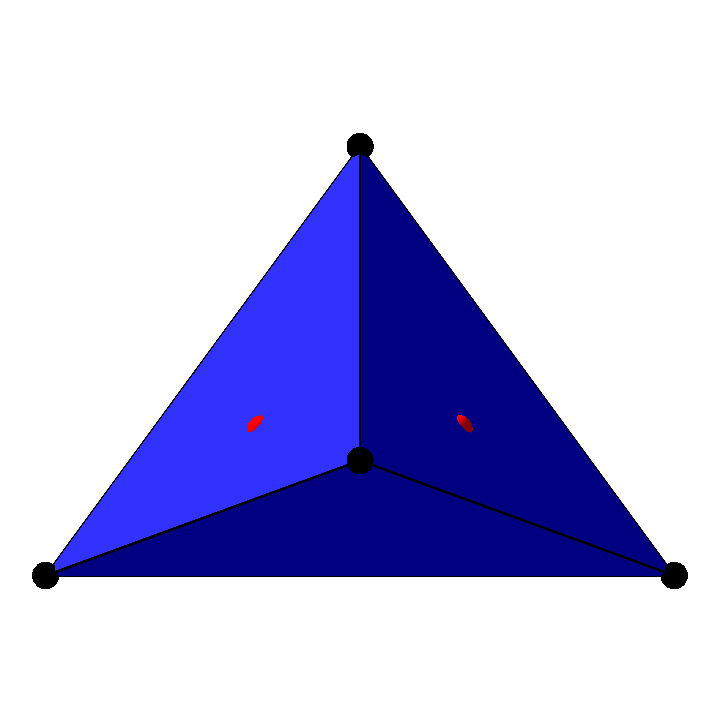}
		\caption{\textcolor{red}{$u_{\mathrm{min}}$}}
		\label{A3MinFundom}
	\end{subfigure}
	\quad
	\begin{subfigure}{.3\textwidth}
		\centering
		\includegraphics[width=4cm, height=4cm]{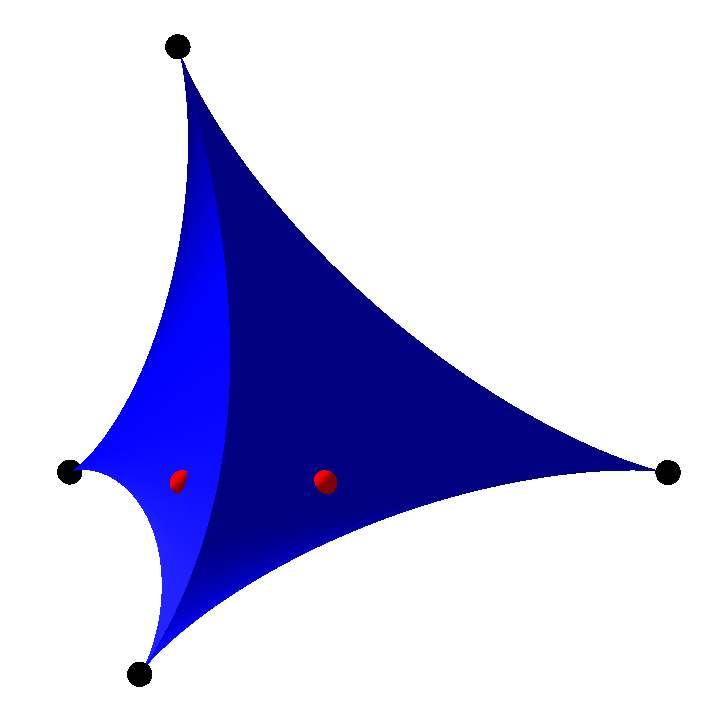}
		\caption{\textcolor{red}{$z_{\mathrm{min}}$}}
		\label{A3MinDeltoid}
	\end{subfigure}
	\caption{In the case of $\RootA[3]$, there are two minimizers \textcolor{red}{$z_{\mathrm{min}} \approx (0.22209, 0.05915, \pm 0.23708)$} for $F(2,8)$ on the boundary of \textcolor{blue}{$\Image$}, the image of the gernalized cosines, with two preimages \textcolor{red}{$u_{\mathrm{min}}\approx (0.40432, \pm 0.15713, 0.17550)$} on the boundary of \textcolor{blue}{$\fundom$}, the fundamental domain of $\weyl\ltimes\Corootlattice$.}
	\label{A3Min}
\end{center}
\end{figure}

\subsubsection{The icositetrachoron in $\R^4$}

The icositetrachoron in $\R^4$ is the Vorono\"{i} cell of the coroot lattice $\Corootlattice$ for $\RootB[4]$ and $\RootD[4]$. It has $24$ vertices, $96$ edges, $96$ faces and $24$ facets. The facets are octahedral cells.

For $\RootB[4]$, the vertices are the orbits of $\fweight{1}$ and $\fweight{4}$. The centers of edges are the orbits of $(\fweight{1}+\fweight{4})/2$ and $\fweight{3}/2$. The centers of faces are the orbit of $(\fweight{1}+\fweight{3})/3$. The centers of facets are the orbit of $\fweight{2}/2$.

For $\RootD[4]$, the vertices are the orbits of $\fweight{1}$, $\fweight{3}$ and $\fweight{4}$. The centers of edges are the orbits of $(\fweight{1}+\fweight{3})/2$, $(\fweight{1}+\fweight{4})/2$ and $(\fweight{3}+\fweight{4})/2$. The centers of faces are the orbit of $(\fweight{1}+\fweight{3}+\fweight{4})/3$. The centers of facets are the orbit of $\fweight{2}/2$.

\begin{table}[H]
\begin{center}
	\scalebox{0.7}{
		\begin{tabular}{|c|c||c|c|c|c|c|c|c|c|c|c|c|c|}
			\hline
			$\Roots$	&	$d\backslash r$	&	$1$			&	$2$			&	$3$			&	$4$			&	$5$			&	$6$			&	$7$			&
			$8$			&	$9$			&	$10$		&	$11$		&	$12$		\\
			\hline
			\hline
			$\RootB[4]$	&	$4$				&	$3.01160$	&	$10.00001$	&	$-		$	&	$10.00000$	&	$-		$	&	$10.0000$	&	$-		$	&
			$10.00000$	&	$-		$	&	$-		$	&	$-		$	&	$-		$	\\
			\hline
			&	$5$				&	$3.77462$	&	$10.00035$	&	$-		$	&	$10.00000$	&	$-		$	&	$10.00000$	&	$-		$	&
			$10.00000$	&	$-		$	&	$10.00000$	&	$-		$	&	$-		$	\\
			\hline
			&	$6$				&	$3.99453$	&	$10.02433$	&	$9.10927$	&	$10.01295$	&	$8.91701$	&	$10.00001$	&	$4.69147$	&
			$10.00000$	&	$-		$	&	$10.00000$	&	$-		$	&	$10.00000$	\\
			\hline
			&	$7$				&	$3.99961$	&	$10.02434$	&	$9.12574$	&	$10.01902$	&	$9.26148$	&	$10.00819$	&	$9.32108$	&
			$10.00000$	&	$8.35442$	&	$10.00000$	&	$4.15681$	&	$10.00000$	\\
			\hline
			\hline
			$\RootD[4]$	&	$4$				&	$3.07035$	&	$10.00004$	&	$-		$	&	$10.00000$	&	$-		$	&	$10.00000$	&	$-		$	&
			$10.00000$	&	$-		$	&	$-		$	&	$-		$	&	$-		$	\\
			\hline
			&	$5$				&	$3.94031$	&	$10.00231$	&	$-		$	&	$10.00000$	&	$-		$	&	$10.00000$	&	$-		$	&
			$10.00000$	&	$-		$	&	$10.00000$	&	$-		$	&	$-		$	\\
			\hline
			&	$6$				&	$3.99496$	&	$10.02432$	&	$9.11312$	&	$10.01314$	&	$8.93873$	&	$10.00001$	&	$5.12215$	&
			$10.00000$	&	$-		$	&	$10.00000$	&	$-		$	&	$10.00000$	\\
			\hline
		\end{tabular}
	}
\end{center}
	\caption{The bound $\chi_m (\R^4,\,\partial \Vor(\Corootlattice(\RootB[4]))) = \chi_m (\R^4,\,\partial \Vor(\Corootlattice(\RootD[4]))) \geq 1-1/F(r,d)$ for the icositetrachoron.}
\label{B4D4IcositetrachoronTable2}
\end{table}

For $r=1$, the numerically computed bound seems to converge to $4$. For $r\geq 2$, the best possible bound we obtained is assumed at $r=2$ and $d=7$, respectively $d=6$. For $r$ odd, the value is always smaller than for $r$ even.

We observe that, for $\RootB[4]$, we have $F(2,7) \geq F(4,7)$ although $2$ divides $4$. This is because the monotonous growth in \Cref{thm_ChromaticChebyshevBound} only holds for $d\to\infty$. In the $\RootD[4]$ case, we have the same for $F(2,6) \geq F(4,6)$. We display the optimal coefficients for the corresponding measure below.

\begin{table}[H]
\begin{center}
	\begin{tabular}{|c||c|c||c|c|}
		\hline
		&	\multicolumn{2}{c||}{$\RootB[4]$}		&	\multicolumn{2}{c|}{$\RootD[4]$}			\\
		\hline
		\hline
		$r$		&	$1-1/F(r,7)$	&	$c_\alpha$			&	$1-1/F(r,6)$	&	$c_\alpha$			\\
		\hline
		$1$		&	$3.99961$		&	$c_{1000}=0.33303$	&	$3.99496$		&	$c_{1000}=0.33305$	\\
		&					&						&					&	$c_{0010}=0.33348$	\\
		&					&	$c_{0001}=0.66697$	&					&	$c_{0001}=0.33348$	\\
		\hline
		$2$		&	$10.02434$		&	$c_{0100}=0.40062$	&	$10.02432$		&	$c_{0100}=0.40188$	\\
		&					&	$c_{1001}=0.35491$	&					&	$c_{1001}=0.17692$	\\
		&					&						&					&	$c_{1010}=0.17692$	\\
		&					&	$c_{0010}=0.17769$	&					&	$c_{0011}=0.17726$	\\
		&					&	$c_{0002}=0.04444$	&					&	$c_{0002}=0.02228$	\\
		&					&						&					&	$c_{0020}=0.02228$	\\
		&					&	$c_{2000}=0.02234$	&					&	$c_{2000}=0.02245$	\\
		\hline
	\end{tabular}
\end{center}
\caption{The optimal coefficients for $F(r,7)$, respectively $F(r,6)$. The coefficients associated to $\weight=\alpha_1\,\fweight{1}+\ldots+\alpha_4\,\fweight{4}$ are denoted by $c_\alpha$.}
\label{icositetrachoronTable}
\end{table}

Recall from \Cref{equation_WeightsRootsB,equation_WeightsRootsD} that the fundamental weights satisfy $\fweight{i}(\RootB[4]) = \fweight{i}(\RootD[4])$ for $i=1,2,4$ and $\fweight{3}(\RootB[4]) = \fweight{3}(\RootD[4]) + \fweight{4}(\RootD[4])$. In the case of $r=2$ in \Cref{icositetrachoronTable}, we observe that
\begin{enumerate}
\item the centers of facets are weighted with $0.40062\approx 0.40188$,
\item the centers of faces are not weighted,
\item the centers of edges are weighted with $0.35491\approx 0.17692+0.17692$ and $0.17769\approx 0.17726$ and
\item the vertices are weighted with $0.02234\approx 0.02245$ and $0.04444\approx 0.02228+0.02228$.
\end{enumerate}
Further computations are limited by the size of the semi--definite program, see \Cref{SDPMatrixNumberSizeTable}. Note that the chromatic number of $\R^4$ for the icositetrachoron polytope is at least $15$ \cite[Theorem 5]{BBMP}, proven analytically via a discrete subgraph and its clique density.

\subsubsection{The cube in $\R^n$}

The cube $[-1/2,1/2]^n$ is the Vorono\"{i} cell of the coroot lattice for the root system $\RootC$, that is, for the cubic lattice $\Corootlattice(\RootC) = \Z^n$. In this case, the chromatic number is known to be $2^n$, see \cite{BBMP} for a counting argument that does not involve spectral bounds. We reprove this fact with the spectral bound by taking a $\weyl$--invariant measure, which is supported on the vertices and centers of edges, faces, etc. of $\Vor(\Corootlattice(\RootC))$.

\begin{proposition}\label{thm_cubeRnbound}
The spectral bound is sharp for $\chi_m (\R^n,\,\partial \Vor(\Corootlattice(\RootC))) = 2^n$.
\end{proposition}
\begin{proof}
The set of dominant weights $\weight\in\Weights^+$ of $\RootC$ with $\sprod{\weight,\highestroot^\vee}=1$ is $\{\fweight{1},\ldots,\fweight{n}\}$. 
We set $(2^n-1)\,c_i := \binom{n}{i}$. Then $c_1,\ldots,c_n \geq 0$ with $c_1 + \ldots + c_n = 1$ and the polynomial
\[
	\sum_{\sprod{\weight,\highestroot^\vee}=1} c_\weight \, T_\weight (z)
=	\sum_{i=1}^n c_i \, z_i .
\]
is an admissible choice for \Cref{eq_VoronoiLowerBound}. We show that it provides the optimal bound $2^n$. 
To do so, we rely on the formula for the fundamental weights from \Cref{equation_WeightsRootsC}, which gives us
\[
	(2^n-1)\,c_i\,\gencos{i}(u) = \sigma_i(\cos(2\pi u_1),\ldots,\cos(2\pi u_n)),
\]
where $\sigma_i$ is the $i$--th elementary symmetric function. When we substitute $z_i=\gencos{i}(u)$ for $u\in\R^n$, then
\[
		(2^n-1)\,\sum_{i=1}^n c_i \, z_i
=		\sum_{i=1}^n (2^n-1)\,c_i \, \gencos{i}(u)
=		\sum\limits_{i=1}^n \sigma_i(\mcos{u_1},\ldots,\mcos{u_n})
=		\prod\limits_{k=1}^n \underbrace{(1+\mcos{u_k})}_{\geq 0}-1
\geq	-1
\]
follows from Vieta's formula and equality holds for $u = 1/2 \, \fweight{j}$. Hence,
\[
		\chi_m (\R^n,\,\partial \Vor(\Corootlattice(\RootC)))
\geq	1-\frac{1}{\min\limits_{z\in\Image} \sum_{i=1}^n c_i \, z_i }
\geq	1-\frac{2^n-1}{- 1}
=		2^n.
\]
\end{proof}

\begin{remark}
For small $n$ $(2\leq n\leq 10)$, one can observe experimentally that the polynomial
\[
	p
:=	1 + \sum\limits_{i=1}^n \binom{n}{i} z_i \in \RX.
\]
is one of two linear factors in $\det(\posmat)$, $\posmat$ being the matrix from \emph{\Cref{thm_HermiteCharacterization}}, and $\Image$ is contained in the halfspace $\{z\in\R^n\,\vert\,p(z)\geq 0\}$. We conjecture that it is true in general. This would simplify the proof of \emph{\Cref{thm_cubeRnbound}}, giving it completely in terms of generalized Chebyshev polynomials and providing a new motivation for the choice of coefficients.
\end{remark}

\subsection{Discussion on the results}
\label{sec_chromaticdiscussion}

In addition to provide bounds on the chromatic number of the graphs that we consider, our method gives information on the discrete measures supported on lattice points up to scaling.

For example, in the case of the hexagon, even by increasing the number of support points, we did not get a discrete measure providing a better bound, see \Cref{A2HexagonTable}. Our experiments then suggest that the optimal measure supported on rational points is the one supported by two orbits: the vertices of the hexagon, with weight $1/3$, and the middle of the edges, with weight $2/3$. 

In the case of the cross--polytope from \Cref{sssec_CP}, we observe a different phenomenon: when increasing the number of possible support points, the optimal measure distribution does not appear to stabilize. It seems then reasonable to expect the bound to get better when increasing the number of points, even though it is hard to conjecture for an optimal discrete measure after our experiments, see \Cref{B3L1NormCoefficients}.
Moreover, we note that the larger the set of possible support points is, the higher we need to go in the order of the hierarchy to get a good bound. 
This can be explained by the fact that the weighted degrees of the involved Chebyshev polynomials get higher, making the semi--definite programs harder to solve. 

Even if we could prove that the spectral bound is sharp for several of our set avoiding graphs, sometimes the bounds that we obtain look far from the expected chromatic number of $\R^n$. 
This might happen for several reasons. 
First, when considering our discrete measures supported on lattices, we are always implicitly computing a bound for a discrete subgraph of $\R^n$, that might have a chromatic number smaller than $\R^n$. 
However, this is not the only reason: getting back to the hexagon, the measure supported on the vertices and the middles of edges gives a bound for a discrete graph.
However, it was proven in \cite{BBMP} that this graph has chromatic number $4$. 
In this case, it is likely that the spectral bound is exactly $25/7$, and does not give the chromatic number. 
Such a phenomenon was already observed in \cite{DMMV}, where, for the lattice $\mathrm{E}_7$, the optimal spectral bound was computed to be $10$, while the chromatic number of this lattice is $14$.

\pagestyle{fancy}
\lhead[E. Hubert, T. Metzlaff, P. Moustrou, C. Riener]{}
\rhead[]{}
\cfoot{}
\rfoot[\today]{\thepage}
\lfoot[\thepage]{}
\renewcommand{\headsep}{8mm}
\renewcommand{\footskip}{15mm}

\section*{Conclusion}

We give an algorithm to minimize a trigonometric polynomial with crystallographic symmetry. 
To do so, we rewrite the problem in terms of generalized Chebyshev polynomials and use established techniques from polynomial optimization with matrix inequalities. 
This results in a hierarchy of SDPs. 
A Maple package to conduct computations with generalized Chebyshev polynomials and to obtain the matrices for the SDP is available\footnote{\href{https://github.com/TobiasMetzlaff/GeneralizedChebyshev}{https://github.com/TobiasMetzlaff/GeneralizedChebyshev}}.

To strengthen our approach, one could consider further techniques from symmetry exploitation, such as symmetry adapted bases \cite{Gatermann04}, and combinations with the exploitation of sparsity \cite{magron23}.

For the chromatic number of avoidance graphs, we present a hierarchy of lower semi--definite bounds that originates from a bilevel polynomial optimization problem. For such problems, it would be interesting to compute the spectral bound for continuous measures supported on the boundary of our polytopes, to conclude whether such an approach could be at least as powerful as the combinatorial approach. Improving the implementation would allow at some point to handle the famous $\RootE[8]$ lattice.

\section*{Acknowledgments}

We want to thank Christine Bachoc (Universit\'{e} de Bordeaux), since the idea of computing the spectral bound for polytope--distance graphs through polynomial optimization was initiated by her, in discussion with Philippe Moustrou. 

The authors are also grateful to Michal Kocvara (University of Birmingham), Milan Korda, Victor Magron (CNRS LAAS Toulouse) and Bernard Mourrain (Inria d'Universit\'{e} C\^{o}te d'Azur) for fruitful suggestions and discussions.

The majority of the work of Tobias Metzlaff was carried out during his doctoral studies \cite{TobiasThesis} at Inria d'Universit\'{e} C\^{o}te d'Azur, supported by European Union’s Horizon 2020 research and innovation programme under the Marie Sk\l odowska-Curie Actions, grant agreement 813211 (POEMA).
Minor changes were applied during his postdoctoral research at RPTU Kaiserslautern--Landau, supported by the Deutsche Forschungsgemeinschaft transregional collaborative research centre (SFB--TRR) 195 ``Symbolic Tools in Mathematics and their Application''.





\clearpage

\bibliographystyle{plain}
{\bibliography{mybib.bib}}

\clearpage

\appendix
\section{Irreducible root systems of type $\RootA$, $\RootC$, $\RootB$, $\RootD$, $\RootG[2]$}
\label{Appendix_IrredRootSys}
\setcounter{equation}{0}

\begin{small}

For $1\leq i\leq n$, we denote by $e_i\in\R^n$ the Euclidean standard basis vectors.

\subsection*{$\RootA$}

The group $\mathfrak{S}_{n}$ acts on $\R^n$ by permutation of coordinates and leaves the subspace $V = \R^n/\langle [1,\ldots,1]^t\rangle = \{u\in\R^{n}\,\vert\,u_1+\ldots +u_{n}=0\}$ invariant. The root system $\RootA$ given in \cite[Planche I]{bourbaki456} is a root system of rank $n-1$ in $V$ with base and fundamental weights
\begin{equation}\label{equation_WeightsRootsA}\tag{A}
\rho_i=e_i-e_{i+1}
\tbox{and}
\fweight{i}=\sum_{j=1}^i e_j -\dfrac{i}{n} \sum_{j=1}^{n} e_j = \dfrac{1}{n} (\underbrace{n-i,\ldots, n-i}_{i\,\mbox{\small{ times}}},\underbrace{ -i, \ldots, -i}_{n-i\,\mbox{\small{ times}}})^t
\end{equation}
for $1\leq i \leq n-1$. The Weyl group of $\RootA$ is $\weyl\cong\mathfrak{S}_{n}$ and the reflection $s_{\rho_i}$ permutes the coordinates $i$ and $i+1$. Thus, $-\fweight{n-i}\in\weyl\,\fweight{i}$ and the orbit $\weyl\,\fweight{i}$ has cardinality $\binom{n}{i}$ for $1\leq i \leq n-1$.

\subsection*{$\RootC$}

The groups $\mathfrak{S}_{n}$ and $\{\pm 1\}^n$ act on $\R^n$ by permutation of coordinates and multiplication of coordinates by $\pm 1$. The root system $\RootC$ given in \cite[Planche III]{bourbaki456} is a root system in $\R^n$ with base and fundamental weights
\begin{equation}\label{equation_WeightsRootsC}\tag{C}
\rho_i=e_i-e_{i+1},\quad \rho_n=2\,e_n
\tbox{and}
\fweight{i}=e_1 + \ldots + e_i.
\end{equation}
for $1\leq i\leq n$. The Weyl group of $\RootC$ is $\weyl \cong \mathfrak{S}_{n}\ltimes\{\pm 1\}^n$. We have $-I_n\in\weyl$ and thus, $-\fweight{i}\in\weyl\,\fweight{i}$. Furthermore, the orbit $\weyl\,\fweight{i}$ has cardinality $2^i\,\binom{n}{i}$ for $1\leq i \leq n$.

\subsection*{$\RootB$}

The root system $\RootB$ given in \cite[Planche II]{bourbaki456} is a root system in $\R^n$. Its Weyl group is isomorphic to that of $\RootC$. The base and fundamental weights are
\begin{equation}\label{equation_WeightsRootsB}\tag{B}
\rho_i=e_i-e_{i+1},\quad \rho_n=e_n
\tbox{and}
\fweight{i}=e_1 + \ldots + e_i,\quad \fweight{n}=(e_1+\ldots+e_n)/2.
\end{equation}
for $1\leq i\leq n$. The Weyl group of $\RootB$ is $\cong \weyl \cong \mathfrak{S}_{n}\ltimes\{\pm 1\}^n$. We have $-I_n\in\weyl$ and thus, $-\fweight{i}\in\weyl\,\fweight{i}$. Furthermore, the orbit $\weyl\,\fweight{i}$ has cardinality $2^i\,\binom{n}{i}$ for $1\leq i \leq n$.

\subsection*{$\RootD$}

The groups $\mathfrak{S}_{n}$ and $\{\pm 1\}^n_+:=\{\epsilon\in\{\pm 1\}^n\,\vert\,\epsilon_1\ldots\epsilon_n=1\}$ act on $\R^n$ by permutation of coordinates and multiplication of coordinates by $\pm 1$, where only an even amount of sign changes is admissible. The root system $\RootD$ given in \cite[Planche IV]{bourbaki456} is a root system in $\R^n$ with base and fundamental weights
\begin{equation}\label{equation_WeightsRootsD}\tag{D}
\begin{array}{rl}
&	\rho_i=e_i-e_{i+1},\quad \rho_n=e_{n-1} + e_n
\tbox{and}\\
&	\fweight{i} = e_1 + \ldots + e_i , \quad \fweight{n-1}=(e_1+\ldots +e_{n-1}-e_n)/2,\quad \fweight{n}=(e_1+\ldots + e_n)/2.
\end{array}
\end{equation}
The Weyl group of $\RootD$ is $\weyl \cong \mathfrak{S}_{n}\ltimes\{\pm 1\}^n_+$. For all $1\leq i \leq n$, we have $-\fweight{i}\in\weyl\,\fweight{i}$, except when $n$ is odd, where $-\fweight{n-1}\in\weyl\,\fweight{n}$. Furthermore, the orbit $\weyl\,\fweight{i}$ has cardinality $2^i\,\binom{n}{i}$ for $1\leq i \leq n-2$ and $\nops{\weyl\,\fweight{n-1}}=\nops{\weyl\,\fweight{n}}=2^{n-1}$.

\subsection*{$\RootG[2]$}

The group $\mathfrak{S}_{3} \ltimes \{\pm 1\}$ acts on $\R^3$ by permutation of coordinates and scalar multiplication with $\pm 1$. The subspace $V = \R^3/\langle[1,1,1]^t\rangle =\{u\in\R^{n}\,\vert\,u_1+u_2 +u_{3}=0\}$ is left invariant. The root system $\RootG[2]$ given in \cite[Planche IX]{bourbaki456} is a root system of rank $2$ in $V$ with base and fundamental weights
\begin{equation}\label{equation_WeightsRootsG}\tag{G}
	\rho_{1}=[1,-1,0]^t,\quad
	\rho_{2}=[-2,1,1]^t\tbox{and}
	\fweight{1}=[1,-1,0]^t,\quad
	\fweight{2}=[-2,1,1]^t.
\end{equation}
The Weyl group of $\RootG[2]$ is $\weyl\cong\mathfrak{S}_{3}\ltimes \{\pm 1\}$. We have $-I_3\in\weyl$ and thus, $-\fweight{1}\in\weyl\,\fweight{1}$ as well as $-\fweight{2}\in\weyl\,\fweight{2}$. Furthermore, $\nops{\weyl\,\fweight{1}} = \nops{\weyl\,\fweight{2}} = 6$.

\end{small}

\clearpage

\section{Coefficients for discrete measures}
\label{Appendix_Coeff}

\begin{table}[H]
	\begin{center}
		\scalebox{0.5}{
			\begin{tabular}{|c||l|l||l|l||l|l|}
				\hline
						&	\multicolumn{2}{c||}{$\RootG[2]$ (\Cref{HexagonTable})}		&	\multicolumn{2}{c||}{$\RootB[3]$ (\Cref{RhombicDodecahedronTable})}		&	\multicolumn{2}{c|}{$\RootB[3]$ (\Cref{B3L1NormCoefficients})}		\\
				\hline
				\hline
				$r$		&	$1-1/F(r,8)$			&	$c_\alpha$						&	$1-1/F(r,9)$				&	$c_\alpha$								&	$1-1/F(r,9)$				&	$c_\alpha$							\\
				\hline
				$2$		&	$3.5714293935747494$	&	$c_{01}=0.6666662750776622$		&	$6.107671348334947$			&	$c_{010}=0.5927896822445022$			&	$6.0000017072602425$		&	$c_{010}=0.799999985332756$			\\
						&							&	$c_{20}=0.33333370766934456$	&								&	$c_{002}=0.06049713057719272$			&								&	$c_{200}=0.20000000682364782$		\\
						&							&									&								&	$c_{101}=0.24388381852316104$			&								&										\\
						&							&									&								&	$c_{200}=0.10282935835880404$			&								&										\\
				\hline
				$4$		&	$3.571429076541122$		&	$c_{02}=0.6666630238845522$		&	$6.107671578689443$			&	$c_{020}=0.5927767228148009$			&	$6.281482412640609$			&	$c_{102}=0.5937675654811545$		\\
						&							&	$c_{21}=5.533750816723066e-06$	&								&	$c_{012}=1.1060691764569475e-07$		&								&	$c_{020}=0.16188833861404459$		\\
						&							&	$c_{40}=0.33333143067593424$	&								&	$c_{111}=3.8973084159378557e-07$		&								&	$c_{210}=0.22680579314997618$		\\
						&							&									&								&	$c_{210}=2.072336714731282e-08$			&								&	$c_{400}=0.017538297991656945$		\\
						&							&									&								&	$c_{004}=0.060493918939264466$			&								&										\\
						&							&									&								&	$c_{103}=2.498258988290966e-06$			&								&										\\
						&							&									&								&	$c_{202}=0.24390270753567078$			&								&										\\
						&							&									&								&	$c_{301}=1.1237155333847226e-07$		&								&										\\
						&							&									&								&	$c_{400}=0.10282351530189676$			&								&										\\
				\hline
				$6$		&	$3.571428681101453$		&	$c_{03}=0.6666623416514681$		&	$6.107669002121958$			&	$c_{030}=0.5927778669897568$			&	$6.302692297425513$			&	$c_{004}=0.0949148422912926$		\\
						&							&	$c_{22}=4.988015651434592e-06$	&								&	$c_{022}=6.061390472114625e-07$			&								&	$c_{112}=0.5014281939941977$		\\
						&							&	$c_{41}=5.706892501421417e-07$	&								&	$c_{121}=1.8206124414166247e-06$		&								&	$c_{302}=7.315642871000283e-08$		\\
						&							&	$c_{60}=0.3333320956275223$		&								&	$c_{220}=4.46761370259674e-08$			&								&	$c_{030}=0.1561352016875235$		\\
						&							&									&								&	$c_{014}=3.593810809967429e-08$			&								&	$c_{220}=0.06437337530336916$		\\
						&							&									&								&	$c_{113}=5.812486718152765e-08$			&								&	$c_{410}=0.18314795448892407$		\\
						&							&									&								&	$c_{212}=6.666867988051883e-08$			&								&	$c_{600}=3.493798257127312e-07$		\\
						&							&									&								&	$c_{311}=2.3184776079239813e-08$		&								&										\\
						&							&									&								&	$c_{410}=1.4463186537305717e-08$		&								&										\\
						&							&									&								&	$c_{006}=0.060493220330598535$			&								&										\\
						&							&									&								&	$c_{105}=3.667230456631809e-07$			&								&										\\
						&							&									&								&	$c_{204}=2.767871177637715e-06$			&								&										\\
						&							&									&								&	$c_{303}=0.24389840991317713$			&								&										\\
						&							&									&								&	$c_{402}=2.2724832493027647e-07$		&								&										\\
						&							&									&								&	$c_{501}=3.6487920151002896e-08$		&								&										\\
						&							&									&								&	$c_{600}=0.10282443292790608$			&								&										\\
				\hline
				$8$		&	$3.571428656208869$		&	$c_{04}=0.6666503161482014$		&	$6.107665541792629$			&	$c_{040}=0.5927413721445046$			&	$6.305009836734212$			&	$c_{014}=0.13422046544583938$		\\
						&							&	$c_{23}=1.5147651853886996e-05$	&								&	$c_{032}=1.546207606818728e-05$			&								&	$c_{204}=0.19985959349100152$		\\
						&							&	$c_{42}=3.3861885617269103e-06$	&								&	$c_{131}=2.4044130958217336e-05$		&								&	$c_{122}=0.24975682959474593$		\\
						&							&	$c_{61}=2.3138911862176023e-06$	&								&	$c_{230}=6.319148130873309e-07$			&								&	$c_{312}=5.427029125502913e-07$		\\
						&							&	$c_{80}=0.3333288267491457$		&								&	$c_{024}=3.806799155209317e-07$			&								&	$c_{502}=4.5084485519007733e-07$	\\
						&							&									&								&	$c_{123}=9.087421097250022e-07$			&								&	$c_{040}=0.1749148298840411$		\\
						&							&									&								&	$c_{222}=7.360949551230705e-07$			&								&	$c_{230}=0.007446177711287559$		\\
						&							&									&								&	$c_{321}=2.4903674630203645e-07$		&								&	$c_{420}=0.11450575956253939$		\\
						&							&									&								&	$c_{420}=1.478631702703237e-07$			&								&	$c_{610}=0.11929412932078559$		\\
						&							&									&								&	$c_{016}=2.943300451878697e-07$			&								&	$c_{800}=1.2143477202148506e-06$	\\
						&							&									&								&	$c_{115}=2.9297946653624157e-07$		&								&										\\
						&							&									&								&	$c_{214}=4.382670764843666e-07$			&								&										\\
						&							&									&								&	$c_{313}=4.754482270338291e-07$			&								&										\\
						&							&									&								&	$c_{412}=2.15224185541249e-07$			&								&										\\
						&							&									&								&	$c_{511}=1.3489711070207265e-07$		&								&										\\
						&							&									&								&	$c_{610}=1.1966359088864953e-07$		&								&										\\
						&							&									&								&	$c_{008}=0.06047432425942938$			&								&										\\
						&							&									&								&	$c_{107}=5.940924769618481e-06$			&								&										\\
						&							&									&								&	$c_{206}=9.065959051197882e-06$			&								&										\\
						&							&									&								&	$c_{305}=3.120191201427434e-05$			&								&										\\
						&							&									&								&	$c_{404}=0.24385942897455937$			&								&										\\
						&							&									&								&	$c_{503}=5.1442857132297714e-06$		&								&										\\
						&							&									&								&	$c_{602}=5.107615032356929e-07$			&								&										\\
						&							&									&								&	$c_{701}=3.282825376188037e-07$			&								&										\\
						&							&									&								&	$c_{800}=0.10282813861868173$			&								&										\\
				\hline
				$10$	&	$3.5714286753163695$	&	$c_{05}=0.6666580152642103$		&	$6.107665208855795$			&	$c_{050}=0.5927564386327037$			&	$6.305020412263947$			&	$c_{106}=0.08316846319737575$		\\
						&							&	$c_{24}=6.815116335719704e-06$	&								&	$c_{042}=9.682802113876587e-06$			&								&	$c_{024}=0.045246108638833285$		\\
						&							&	$c_{43}=2.193358658091023e-06$	&								&	$c_{141}=1.670078365629944e-05$			&								&	$c_{214}=0.34658821329785183$		\\
						&							&	$c_{62}=7.690265068084644e-07$	&								&	$c_{240}=5.255662733075187e-07$			&								&	$c_{404}=3.348605887939886e-06$		\\
						&							&	$c_{81}=1.3120927502321667e-06$	&								&	$c_{034}=3.392005937306701e-07$			&								&	$c_{132}=0.10956846871243874$		\\
						&							&	$c_{100}=0.33333089013330625$	&								&	$c_{133}=7.020019227642831e-07$			&								&	$c_{322}=4.560289153963601e-06$		\\
						&							&									&								&	$c_{232}=5.878778783285085e-07$			&								&	$c_{512}=2.4403680757047186e-06$	\\
						&							&									&								&	$c_{331}=2.0282710528401052e-07$		&								&	$c_{702}=2.8049815946132317e-06$	\\
						&							&									&								&	$c_{430}=1.0633073501921022e-07$		&								&	$c_{050}=0.16787580057675633$		\\
						&							&									&								&	$c_{026}=1.5476534565169418e-07$		&								&	$c_{240}=0.003149346378010778$		\\
						&							&									&								&	$c_{125}=1.991295614776585e-07$			&								&	$c_{430}=0.06305418908391902$		\\
						&							&									&								&	$c_{224}=2.866037176316962e-07$			&								&	$c_{620}=0.11837765189988159$		\\
						&							&									&								&	$c_{323}=2.5625689173786394e-07$		&								&	$c_{810}=0.06295198879060124$		\\
						&							&									&								&	$c_{422}=1.308190531621658e-07$			&								&	$c_{1000}=6.609291152696872e-06$	\\
						&							&									&								&	$c_{521}=7.975379769098456e-08$			&								&										\\
						&							&									&								&	$c_{620}=6.171409199298282e-08$			&								&										\\
						&							&									&								&	$c_{018}=1.85758436253194e-07$			&								&										\\
						&							&									&								&	$c_{117}=1.701249900070453e-07$			&								&										\\
						&							&									&								&	$c_{216}=2.099121815359243e-07$			&								&										\\
						&							&									&								&	$c_{315}=2.487195823711977e-07$			&								&										\\
						&							&									&								&	$c_{414}=3.093875688657081e-07$			&								&										\\
						&							&									&								&	$c_{513}=1.3687347352043666e-07$		&								&										\\
						&							&									&								&	$c_{612}=8.735351860568483e-08$			&								&										\\
						&							&									&								&	$c_{711}=7.010541809893211e-08$			&								&										\\
						&							&									&								&	$c_{810}=7.739160553574774e-08$			&								&										\\
						&							&									&								&	$c_{0010}=0.06047862904298748$			&								&										\\
						&							&									&								&	$c_{109}=8.297185248986314e-06$			&								&										\\
						&							&									&								&	$c_{208}=2.1225449116373842e-06$		&								&										\\
						&							&									&								&	$c_{307}=5.4058462097554265e-06$		&								&										\\
						&							&									&								&	$c_{406}=1.4189605438770391e-05$		&								&										\\
						&							&									&								&	$c_{505}=0.24387355234988803$			&								&										\\
						&							&									&								&	$c_{604}=1.143617926217899e-06$			&								&										\\
						&							&									&								&	$c_{703}=3.817647143894467e-07$			&								&										\\
						&							&									&								&	$c_{802}=3.421556035470313e-07$			&								&										\\
						&							&									&								&	$c_{901}=4.1932268201038956e-07$		&								&										\\
						&							&									&								&	$c_{1000}=0.10282755980669979$			&								&										\\
				\hline
			\end{tabular}
		}
	\end{center}
\caption{The coefficient for the obtained bounds.}\label{table_appendix}
\end{table}

\end{document}